\documentclass[11pt]{amsart}

\usepackage[utf8]{inputenc}
\usepackage[T1]{fontenc}
\usepackage{pgf,tikz}
\usetikzlibrary{arrows}
\usetikzlibrary{patterns}
\usetikzlibrary{intersections}
\usetikzlibrary{calc}
\usetikzlibrary{decorations.markings}
\usepackage{enumerate}
\usepackage{float}
\usepackage{algorithm}
\usepackage{algpseudocode}
\usepackage{a4wide}
\usepackage{hyperref}
\usepackage[
    msc-links
]{amsrefs}
\BibSpec{software}{%
    +{}  {\PrintAuthors} {author}
    +{,} { \textit} {title}
    +{,} { \PrintDateB} {date}
    +{,} { \PrintDOI} {doi}
    +{,} { available at \url} {url}
}

\newtheorem{thm}{Theorem}[section]
\newtheorem{lem}[thm]{Lemma}
\newtheorem{cor}[thm]{Corollary}
\newtheorem{conj}[thm]{Conjecture}

\theoremstyle{definition}

\title{Topological inductive constructions for tight surface graphs} 

\author{James Cruickshank}
\address{School of Mathematics, Statistics and Applied Mathematics, National University of Ireland Galway, Ireland.}
\email{james.cruickshank@nuigalway.ie}
\author{Derek Kitson}
\address{Department of Mathematics and Computer Studies, Mary Immaculate College, Thurles, Ireland.}
\email{derek.kitson@mic.ul.ie}
\author{Stephen C. Power}
\address{Department of Mathematics and Statistics, Lancaster University,  Lancaster LA1 4YF, U.K.}
\email{s.power@lancaster.ac.uk}
\thanks{The second and third authors were supported by the Engineering and Physical Sciences Research Council [grant number EP/P01108X/1].}
\author{Qays Shakir}
\address{Middle Technical University, Baghdad, Iraq}
\email{qays.shakir@gmail.com}
\thanks{The fourth author gratefully acknowledges the financial support from the Iraqi Ministry of Higher Education and Scientific Research and Middle Technical University, Baghdad.}



\tikzset{
    vertex/.style = {inner sep = 0mm,circle,draw,minimum size=5pt,fill=black},
    edge/.style = {line width = 1pt}, 
    bounding box/.style = {dash pattern = on 3pt off 3pt}
}

\newcommand{\intt}{\mathrm{int}}
\newcommand{\extt}{\mathrm{ext}}
\newcommand{\cyc}{C}

\begin{document}

\begin{abstract}
    We investigate properties of sparse and tight surface graphs.
    In particular 
    we derive topological inductive constructions   for \( (2,2) \)-tight
    surface graphs in the case of the sphere, the plane, the twice punctured sphere and the torus. 
    In the case of the torus we identify all 
    116 irreducible base graphs 
    and provide a geometric application involving contact graphs of
    configurations of circular arcs.
\end{abstract}

\maketitle

\section{Introduction}

Inductive characterisations of various families of graphs play an 
important role in many parts of graph theory. 
A graph $G = (V,E)$ is said to be $(2,2)$-sparse if for any 
nonempty $V' \subset V$, we have $|E(V')| \leq 2|V'|-2$. If, in addition, 
$|E| = 2|V|-2$ we say that $G$ is $(2,2)$-tight. Such graphs
arise naturally in various parts of geometric graph 
theory, including framework rigidity, circle packings, and also in 
graph drawing. 

We will derive inductive characterisations of $(2,2)$-tight graphs
that are embedded without edge crossings 
in certain orientable surfaces of genus at most 1. Our characterisations
will be based on edge contractions and are in the spirit of 
well known results of Barnette, Nakamoto and others (\cites{MR1021367,MR1348564,MR1399674})
on irreducible triangulations and quadrangulations
of various surfaces. It may be worth noting, for example, that 
a graph is a quadrangulation of the plane if and only if it is 
$(2,4)$-tight. There are similar characterisations of quadrangulations
for various other surfaces. Thus it is clear that our results are 
related to, but distinct from, existing results on quadrangulations. 

Since our graphs are embedded, we consider inductive characterisations
based on topological edge contractions - that is to say that the 
contraction preserves the embedding of the graph (precise definitions
given below). This is a key point, since the well-known 
inductive characterisation of simple $(2,2)$-tight graphs by Nixon, Owen and Power 
is purely graph theoretic (\cite{MR3022162}). To further illustrate the 
significance of this we give an application of our main result to 
a recognition problem in graph drawing. The topological nature of the 
inductive characterisation is crucial in this context. 

Finally we note that there are similar topological 
inductive characterisations
of Laman graphs in the literature already 
(see Fekete et al. and Haas et al.), which have interesting 
geometric applications to pseudotriangulations and auxetic structures.

\subsection{Summary of main results}

Section \ref{sec_background} and the first part of Section \ref{sec_sparsity}
are background material for the rest of the paper. The main contributions of the paper 
are as follows

\begin{itemize}
    \item Theorem \ref{thm_hole_filling} presents an elementary but very useful principle 
        concerning sparsity counts and graphs embedded in surfaces. 
        While related results and special cases already exist in the literature, 
        our statement and proof emphasises that this is a general principle that applies to 
        wide range of sparsity counts and to surfaces of all genus. 
    \item In Section \ref{sec_inductive} we analyse the quadrilateral contraction 
        move. This operation is well known 
        in the context of quadrangulations. Here we examine its properties with respect to 
        $(2,2)$-sparsity and prove some structural results about non contractible quadrilaterals
        in this context.
    \item 
        Theorem \ref{thm_tight_sub_irred} 
        shows that if $G$ is an irreducible $(2,2)$-tight surface graph, then any 
        $(2,2)$-tight subgraph is also irreducible. This holds for surfaces of arbitrary 
        genus.
    \item We give topological inductive characterisations of $(2,2)$-tight 
        graphs embedded in the sphere, plane, annulus (Theorem \ref{thm_ann_irred}) and the
        torus (Theorem \ref{thm_main_torus}). The first two of these are relatively standard, whereas the 
        latter two characterisations are new. In the case of Theorem \ref{thm_main_torus} we have 
        also identified the 116 irreducible $(2,2)$-tight torus graphs. We do not give an explicit 
        description of these graphs in the paper for reasons of space, but the reader is referred to 
        \cite{irreducibles} for details.
    \item Finally we present an application of our results to a recognition problem in geometric graph theory.  Specifically we show that every $(2,2)$-tight torus graph can be realised as the contact graph of a collection of nonoverlapping circular arcs in the flat torus. We note that by passing to the universal cover this result may also be interpreted as a recognition result for contact graphs of doubly periodic collections of circular arcs in the plane.
\end{itemize}
We also conjecture a generalisation of our main results (Conjecture \ref{conj_finite}) to the set of irreducible $(2,2)$-tight surface graphs for surfaces of arbitrary genus. This would be analogous to results of Barnette, Nakamoto and others on triangulations and quadrangulations of surfaces.  As noted above and at appropriate points in the paper, many of our results are valid for a range of sparsity counts and for surfaces of arbitrary genus and these will be useful in future investigations of this conjecture.

\section{Graphs, surfaces and  embeddings}
\label{sec_background}

In this section we fix our conventions and terminology regarding 
topological graphs. Throughout, we use the word graph for a finite undirected 
multigraph. So loop and parallel edges are allowed a priori, although 
loop edges will not arise in most of the cases of interest.
If $\Gamma$ is a graph then $|\Gamma|$ is its geometric realisation. 

Suppose that $\Sigma$ is a compact real 2-dimensional manifold without
boundary. A $\Sigma$-graph, or surface graph, is a pair $(\Gamma,\varphi)$ where $\Gamma$ is a graph
and $\varphi: |\Gamma| \rightarrow \Sigma$ is a continuous embedding of the 
geometric realisation of $\Gamma$ in $\Sigma$. Given $\Sigma_i$-graphs $(\Gamma_i, \varphi_i)$ for $ i =1,2$ we say that they are isomorphic if there is a 
homeomorphism $h:\Sigma_1 \rightarrow \Sigma_2$ and a graph 
isomorphism $g: \Gamma_1 \rightarrow \Gamma_2$ such that $h\circ \varphi_1 = \varphi_2 \circ |g|$, where $|g|$ is the induced homeomorphism $|\Gamma_1| \rightarrow |\Gamma_2|$. By the Heffter-Edmonds-Ringel rotation principle, the surface
$\Sigma$ and the $\Sigma$-graph $(\Gamma,\varphi)$ are determined up to 
isomorphism by data consisting of a rotation system on $\Gamma$, a partition of the 
set facial walks associated to the rotation system and for each part of the partition 
a nonnegative integer that represents the genus of the corresponding facial region. Note that we do not assume that 
our surface graphs are cellular. See 
\cite{MR1844449} for details of this. We will be interested in determining certain 
classes of surface graphs up to isomorphism. The above-mentioned principle allows us to 
argue topologically using properties of surfaces and curves in surfaces to deduce combinatorial 
information and we choose to write our arguments using topological terminology based on this. 

Now we clarify the meaning of some standard terms which may have ambiguous interpretations in 
this topological context. Let $G=(\Gamma, \varphi)$ be a $\Sigma$-graph and let $e \in E(\Gamma)$. By $\Gamma/e$ we mean the graph obtained by identifying the end vertices of $e$ and deleting 
(only) the edge $e$. Thus edge contractions can create parallel edges and/or loops. 
By $G/e$ we mean the surface graph obtained by collapsing the arc corresponding to $e$ to a 
single point. Clearly the underlying graph of $G/e$ is $\Gamma/e$. A face, $F$,
of $G$ is a connected component of $\Sigma - \varphi(|\Gamma|)$. As is well known there 
is a well defined collection of closed boundary walks associated to $F$. We say that 
$F$ is non degenerate if no vertex occurs more than once in this collection of walks. 
A cellular face is one that is homeomorphic to $\mathbb R^2$. 
For a cellular face $F$, the degree of $F$, denoted $|F|$, 
is the edge length of its unique boundary 
walk (which of course 
may differ from the number of vertices or edges in degenerate cases). 
We write $f_i$ for the number of cellular faces of degree $i$. Note that if $\Sigma$ is 
connected then $f_0 =1$ if $\Sigma$ is a sphere and $\Gamma$ comprises a single vertex, and 
$f_0 = 0$ otherwise. 

Finally we note that we extend much of the standard language of graph theory concerning 
subgraphs, intersections and unions to surface graphs, understanding that these terms
apply to the underlying graphs. Thus if $G = (\Gamma,\varphi)$ is a $\Sigma$-graph, a 
$\Sigma$-subgraph of $G$ is a pair $(\Gamma',\varphi|_{\Gamma'})$ where 
$\Gamma'$ is a subgraph of $\Gamma$ and $\varphi|_{\Gamma'}$ is the restriction 
of $\varphi$ to $|\Gamma'|$. If $H_1,H_2$ are $\Sigma$-subgraphs of $G$ 
then $H_1 \cup H_2$, respectively $H_1\cap H_2$, is the $\Sigma$-graph whose underlying 
graph is the union, respectively intersection, of the underlying graphs of $H_1$ and $H_2$.

\section{Sparsity}
\label{sec_sparsity}

For a graph \( \Gamma = (V,E,s,t) \) as above, define 
\( \gamma(\Gamma) = 2|V| - |E| \). 
For \( l \leq 2\) we say that \( \Gamma \) is \( (2,l) \)-sparse (or just 
sparse if \( l \) is clear from the context) if, \( \gamma(\Gamma')
\geq l\) for every nonempty
subgraph \( \Gamma' \) of \( \Gamma \). We say that \( \Gamma \) is 
\( (2,l) \)-tight if it is \( (2,l) \)-sparse and \( \gamma(\Gamma)
=l\). We will be particularly interested in \( (2,2) \)-sparse graphs.
Note that \( (2,2) \)-tight graphs cannot have loop edges but can have 
parallel edges (but not triples of parallel edges).

We record some standard elementary facts for later use. The proofs are straightforward
and we omit them.
Suppose that \( \Gamma_1,\Gamma_2 \) are subgraphs of \( \Gamma \). 
Then
\begin{equation}
    \gamma(\Gamma_1 \cup \Gamma_2) = \gamma(\Gamma_1)+\gamma(\Gamma_2)
    -\gamma(\Gamma_1 \cap \Gamma_2)
    \label{lem_incl_excl}
\end{equation}

\begin{lem}
    \label{lem_connected}
    Suppose that \( \Gamma \) is \( (2,2) \)-sparse and that 
    \( \gamma(\Gamma') \leq 3 \) for some subgraph \( \Gamma' \)
    of \( \Gamma \). Then 
    \( \Gamma' \) is connected.\qed
\end{lem}

\begin{lem}
    \label{lem_int_unions}
    Suppose that \( \Gamma_1,\Gamma_2 \) are \( (2,2) \)-tight 
    subgraphs of a \( (2,2) \)-sparse
    graph \( \Gamma \). If  \( \Gamma_1 \cap \Gamma_2 \) is not empty 
    then both \( \Gamma_1 \cup \Gamma_2 \) and \( \Gamma_1 \cap \Gamma_2 \)
    are \( (2,2) \)-tight.\qed
\end{lem}

We also record a straightforward consequence of Euler's polyhedral 
formula.
\begin{thm}
    \label{thm_euler_sparsity}
    If \( \Sigma \) is a connected boundaryless compact orientable 
    surface of genus \( g \) and \( G \) is a 
    cellular \( \Sigma \)-graph then 
    \begin{equation}
        \sum_{i\geq 0} (4-i)f_i = 8-8g - 2\gamma(G)
    \end{equation}
\end{thm}

\begin{proof}
    Use the polyhedral formula and the fact that \( \sum if_i = 2|E| \).
\end{proof}

Next we will derive an elementary but subsequently very useful principle relating 
the function $\gamma$, which is defined purely in terms of the underlying graph,
to the embedding of the graph in the surface. 
We note that related results and 
special cases of this have appeared elsewhere (notably 
\cites{MR2165974,MR3575219,CKP2}). Here we attempt 
to express this principle in a more general terms: for surfaces of 
arbitrary genus and for a variety of sparsity counts.

In order to state the result we first need some notation. Let 
$H$ be a subgraph of the $\Sigma$-graph $G$ and suppose that 
$F$ is a face of $H$. Let $\intt_G(F)$ be the subgraph of $G$
consisting of all vertices and edges of $G$ that lie inside $\overline F$
(the topological closure of $F$ in $\Sigma$). Let $\extt_G(F)$ be the 
subgraph of $G$ consisting of all vertices and edge of $G$ that lie in 
$\Sigma - F$. 
Define $\partial F = \intt_G(F) \cap \extt_G(F)$ and note that 
$\partial F$ is the smallest subgraph of $G$ that supports the boundary walks of 
$F$. It follows that $G = \intt_G(F) \cup \extt_G(F)$ since any 
edge joining $\intt_G(F)$ to $\extt_G(F)$ must pass through $\partial F$.

\begin{thm}
    \label{thm_hole_filling}
    Suppose that \( l \leq 2 \) and that 
    \( G \) is a \( (2,l) \)-tight \( \Sigma \)-graph.
    If \( H \) is a subgraph of \( G \) and 
    \( F \) is a face of \( H \), then 
    \( \gamma(H \cup \intt_G(F)) \leq \gamma(H)\).
\end{thm}

\begin{proof}
    By (\ref{lem_incl_excl}) we have 
    \[ \gamma(H \cup \intt_G(F)) = \gamma(H) + \gamma(\intt_G(F)) -
    \gamma(H \cap \intt_G(F)).\]
    Now, \( H\cap \intt_G(F) = \extt_G(F) \cap \intt_G(F) \) and using 
    (\ref{lem_incl_excl}) again, we see that 
    \begin{eqnarray*}
        \gamma(H \cap \intt_G(F)) &= & \gamma(\extt_G(F) \cap \intt_G(F)) \\
        &= & \gamma(\intt_G(F)) + \gamma(\extt_G(F)) -   \gamma(\extt_G(F) \cup \intt_G(F)) \\
        &= & \gamma(\intt_G(F)) + \gamma(\extt_G(F)) -   \gamma(G) \\
        &= & \gamma(\intt_G(F)) + \gamma(\extt_G(F)) -   l \\
        & \geq & \gamma(\intt_G(F)). \\
    \end{eqnarray*}
    The last inequality above follows from applying the sparsity 
    of \( G \) to the nonempty subgraph \( \extt_G(F) \).
\end{proof}

\begin{cor}
    \label{cor_hole}
    Suppose that \( l \leq 2 \) and that 
    \( G \) is a \( (2,l) \)-tight \( \Sigma \)-graph.
    If \( H \) is a subgraph of \( G \) and 
    \( F \) is a face of \( H \), then 
    \( \gamma(\extt_G(F)) \leq \gamma(H)\).
\end{cor}

\begin{proof}
    Let \( J_1,\cdots,J_k \) be all the faces of \( H \) that 
    are different from \( F \). Then \( \extt_G(F) 
    = H \cup \bigcup_{i=1}^k \intt_G(J_i)\). Now the
    conclusion follows from repeated applications of 
    Theorem \ref{thm_hole_filling}.
\end{proof}

We remark that all of the results of this section admit straightforward 
adaptations to the function \( \gamma_k\) which maps 
\(H \mapsto k|V(H)| - |E(H)| \), for any positive integer 
\( k \).

We conclude this section by making some straightforward observations
about the subgraphs $\intt_G(F)$ and $\extt_F(G)$ that will be useful in the 
sequel.
Any face of \( \intt_G(F) \) that is contained in 
\( F \) is also a face of \( G \). 
On the other hand, there are one or more faces of \( \intt_G(F) \)
which are contained in \( \Sigma - \overline F \). We call such a
face an external face of \( \intt_G(F) \).
Such an external face need not be a face of \( G \).
Note that if
\( F \) has a unique boundary walk that is a simple cycle, then 
\( \intt_G(F) \) has just one external face. In general it may have
more than one external face.
Observe that \( \extt_G(F) \) has one exceptional face, namely \( F \),
such that all other faces of \( \extt_G(F) \)
are also faces of \( G \).

\section{Inductive operations on surface graphs}
\label{sec_inductive}

In this section we will focus on topological 
inductive operations on graphs that are natural in the context 
of \( (2,l) \)-tight graphs. We consider three types contractions
associated to cellular faces of degree $2$, $3$ and $4$, respectively 
called digons, triangles and quadrilaterals hereafter. In each case 
the contraction decreases the number of vertices by one and the number of 
edges by two. We investigate
necessary and sufficient conditions for these moves to preserve the property of 
being $(2,2)$-tight. We pay particular attention to degenerate cases as these play 
an important role later. 

\subsection{Digon and triangle contractions}
\label{subsec_digon_contr}
Let \( G \) be a \( \Sigma \)-graph and
suppose that \( D \) is a digon of \( G \) with boundary walk 
\( v_1,e_1,v_2,e_2,v_1 \) such that  \( v_1 \neq v_2 \) 
and \( e_1 \neq e_2 \). 
Let \( G_D = (G/e_1) - e_2 \). 
Observe that \( (G/e_1)-e_2 \) is canonically isomorphic to 
\( (G/e_2)-e_1 \), so \( G_D \) depends only on the digon and 
not the particular choice of labelling of the edges.  
We remark that, for a connected surface \( \Sigma \),  while a digon in a 
\( (2,2) \)-sparse \( \Sigma \)-graph necessarily has distinct vertices, 
it may 
have degenerate boundary, but only in the case that the graph is a single 
(non loop) edge and \( \Sigma \) is a sphere. 

The proof of the following is straightforward and we omit it.

\begin{lem}
    \label{lem_digon_contractible}
    \( G \) is \( (2,l) \)-sparse if and only if \( G_D \) is 
    \( (2,l) \)-sparse
\end{lem}

Now suppose that \( T \) is a triangle in \( G \) with 
boundary walk \( v_1,e_1,v_2,e_2,v_3,e_3,v_1 \)
such that \( v_1 \neq v_2 \) and \( e_1 \neq e_2 \). Let 
\( G_{T,e_1} = (G/e_1) - e_2 \). 
Again we omit the proof of the following lemma as it is a straightforward
consequence of the definitions.

\begin{lem}
    \label{lem_triangle_blocker}
    Suppose that \( G \) is \( (2,l) \)-sparse and that 
    \( G_{T,e_1} \) is not \( (2,l) \)-sparse. Then 
    there is a subgraph \( H \) of \( G \) that contains
    \( e_1 \) but not \( v_3 \) such that \( \gamma(H) = l \)
\end{lem}

We refer to the graph \( H \) whose existence is asserted in 
Lemma \ref{lem_triangle_blocker} as a blocker for the 
contraction \( G_{T,e_1} \).

We note that a triangle in a 
\( (2,2) \)-sparse surface graph necessarily has a non degenerate boundary
walk, 
since any degeneracy would entail a (forbidden) loop edge.
Thus, in this case there 
are three possible contractions (one for each of the 
edges) associated to any such face.

\begin{lem}
    \label{lem_tri_contractible}
    Suppose that 
    \( G \) is a \( (2,2) \)-sparse \( \Sigma \)-graph and that 
    \( T \) is a triangle with edges \( e_1,e_2,e_3 \). Then 
    at least two of the \( \Sigma \)-graphs 
    \( G_{T,e_1},G_{T,e_2},G_{T,e_3} \) 
    are \( (2,2) \)-sparse.
\end{lem}

\begin{proof}
    Suppose that there are blockers \( H_1 \), respectively 
    \( H_2 \), for \( G_{T,e_1} \) respectively \( G_{T,e_2} \). 
    Then \( v_1,v_3 \in H_1 \cup H_2 \). However 
    \( v_3\not\in H_1 \) and \( v_1 \not\in H_2 \) so
    \( e_3 \not\in H_1\cup H_2 \). However \( v_2 \in H_1\cap H_2 \)
    so by Lemma \ref{lem_int_unions}, \( H_1\cup H_2 \) is 
    \( (2,2) \)-tight. This contradicts the sparsity of \( G \).
\end{proof}

\subsection{Quadrilateral contractions}

In the case of quadrilaterals we consider a somewhat different contraction
move. In this case the analysis is a little more complicated and we include the 
details. 

Suppose that \( Q \) is a quadrilateral of \( G \) 
with possibly degenerate boundary walk \( v_1,e_1,v_2,e_2,v_3,e_3,v_4,e_4,v_1 \).
Suppose that \( v_1 \neq v_3 \) and \( e_1 \neq e_3 \). 
Let \( d \) be a new edge that joins \( v_1 \) and \( v_3 \) and 
is embedded as a diagonal of the quadrilateral \( Q \). Define 
\( G_{Q,v_1,v_3} \) to be \( (G\cup \{d\})/d - \{e_1,e_3\} \). Clearly the 
underlying graph of \( G_{Q,v_1,v_3} \) is obtained 
from \( \Gamma \) by identifying the vertices \( v_1 \) and 
\( v_3 \) and then deleting \( e_1 \) and \( e_3 \). Thus 
\( \gamma(G) = \gamma(G_{Q,v_1,v_3}) \). However this 
quadrilateral contraction move does not necessarily preserve
\( (2,l) \)-sparsity.

\begin{lem}
    \label{lem_quad_blocker}
    Suppose that \( G \) is \( (2,l) \)-sparse but 
    \( G_{Q,v_1,v_3} \) is not \( (2,l) \)-sparse. Then 
    at least one of the following statements is true.
    \begin{enumerate}
        \item There is some subgraph \( H \) of \( G \)
            such that \( v_1,v_3 \in H \), exactly 
            one of \( v_2,v_4  \) is in \( H \) and 
            \( \gamma(H) = l \). (\( H \) is called a type 1 blocker.)
        \item There is some subgraph \( K \) of \( G \) 
            such that \( v_1,v_3 \in K \), \( v_2,v_4 \not\in K \) 
            and \( \gamma(K) = l+1 \). (\( K \) is called a type 2 blocker.)
    \end{enumerate}
\end{lem}

\begin{proof}
    Let \( K \) be a maximal subgraph of \( G_{Q,v_1,v_3} \) satisfying 
    \( \gamma(K) \leq l-1 \). Let \( z \) be the vertex of 
    \( G_{Q,v_1,v_3} \) corresponding to \( v_1 \) and \( v_3 \). 
    Clearly \( z \in K \), otherwise \( K \) would also 
    be a subgraph of \( G \). Let \( H \) be the maximal subgraph 
    of \( G \) satisfying \( (H\cup\{d\})/d - \{e_1,e_3\} = K \). It is 
    clear that \( H \) is an induced subgraph, since \( K \) is an 
    induced subgraph. If \( \{v_2,v_4\} \subset H \), then 
    \( \gamma(H) = \gamma(K) \leq l-1 \) which contradicts the 
    sparsity of \( G \). So at most one of \( v_2,v_4  \) belongs
    to \( H \). Also, it is clear that \( l \leq \gamma(H) \leq \gamma(K)+2
    \leq l-1\). So \( \gamma(H) = l \) or \( l+1 \). If \( \gamma(H) = l \)
    and one of \( v_2,v_4 \in H \) then (1) is true. If \( \gamma(H) = l \) and 
    neither of \( v_2,v_4 \) is in \( H \), then let \( H' = H \cup \{v_2\} \cup 
    \{e_1,e_2\}\). Now observe that \( e_1 \neq e_2 \) since \( v_1 \neq v_3 \). 
    Thus \( \gamma(H') = \gamma(H) = l \) and, again, (1) is true. Finally 
    if \( \gamma(H) = l+1 \). Then \( \gamma(H) = \gamma(K)+2 \) and 
    since \( H \) is an induced graph, it follows 
    that neither of \( v_2,v_4 \) belongs to \( H \). Thus (2) is true in this case.
\end{proof}

In the special case that \( l =2 \), various degeneracies are forbidden.
Now suppose that 
\( G \) is a \( (2,2) \)-sparse \( \Sigma \)-graph 
and that 
\( Q \) is a quadrilateral face of \( G \) with boundary walk 
\( v_1,e_1,v_2,e_2,v_3,e_3,v_4,e_4,v_1\).

\begin{lem}
    \label{lem_quad_nondegen_1}
    For \( i = 1,2,3 \),
    \( v_i 
    \neq v_{i+1}\), and $v_1 \neq v_4$.
\end{lem}

\begin{proof}
    Loop edges are 
    forbidden in a \( (2,2) \)-sparse graph. 
\end{proof}

\begin{lem}
    \label{lem_orient_no_repeat_edges}
    Suppose that \( \Sigma \) is orientable and that \( G \) is 
    \( (2,2) \)-tight. Then \( e_i \neq e_j \) for \( 1\leq i < j\leq 4 \).
\end{lem}

\begin{proof}
    We observe that since \( \Sigma \) is orientable, 
    a repeated edge in \( \partial Q \)
    implies the existence either of a vertex of degree one or of a loop edge.
    Both of these are forbidden in a \( (2,2) \)-tight graph. 
\end{proof}

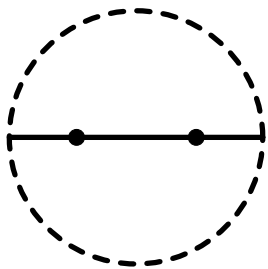
\begin{figure}
    \centering

    \definecolor{xdxdff}{rgb}{0,0,0}
\begin{tikzpicture}[line cap=round,line join=round,>=triangle 45,x=0.5cm,y=0.5cm]
    \draw [line width=2pt,dash pattern=on 5pt off 5pt] (-8.04,2.09) circle (1.684395cm);
\draw [line width=2pt] (-11.40879800522382,2.09)-- (-9.62,2.09);
\draw [line width=2pt] (-9.62,2.09)-- (-6.44,2.09);
\draw [line width=2pt] (-6.44,2.09)-- (-4.671201994776178,2.09);
\begin{scriptsize}
\draw [fill=xdxdff] (-9.62,2.09) circle (3pt);
\draw [fill=xdxdff] (-6.44,2.09) circle (3pt);
\end{scriptsize}
\end{tikzpicture}

\caption{A \( (2,2) \)-tight projective plane graph. Here we are using the representation of the projective plane as a disc with antipodal boundary points identified. This surface graph has a single quadrilateral face, with a degenerate boundary walk.}
    \label{fig_proj_plane}
\end{figure}

\begin{lem}
    \label{lem_degen_quad}
    Suppose that \( \Sigma \) is orientable, \( G \) is \( (2,2) \)-tight
    and that \( v_1=v_3 \). Then \( v_2 \neq v_4 \). Furthermore
    \( G_{Q,v_2,v_4} \) is also \( (2,2) \)-tight.
\end{lem}

\begin{proof}
    Suppose that \(v_2 = v_4\). By Lemma \ref{lem_quad_nondegen_1}
    and the sparsity of \( G \), \( \partial Q \) has exactly two vertices 
    and two edges. This contradicts Lemma \ref{lem_orient_no_repeat_edges}.
    Thus
    \( v_2 \neq v_4. \)

    Now suppose that \( G_{Q,v_2,v_4} \) is not \( (2,2) \)-tight.
    By Lemma \ref{lem_quad_blocker} there is a blocker for this 
    contraction. Since \( v_1 = v_3 \) by assumption, the blocker must 
    be a type 2 blocker. Thus we have a subgraph \( K  \) such that 
    \( \gamma(K) = 3 \), \( v_2,v_4 \in K \) and \( v_1 \not\in K \).
    However, by Lemma \ref{lem_orient_no_repeat_edges} there 
    are at least four edges joining \( v_1 \) to \( K \), 
    contradicting the sparsity of \( G \).
\end{proof}

See Figure \ref{fig_proj_plane} for an example of \( (2,2) \)-tight projective plane graph
whose only face is a quadrilateral with repeated edges in the boundary walk.
This example 
shows that orientability is a necessary hypothesis in the
statements of Lemmas \ref{lem_orient_no_repeat_edges} and \ref{lem_degen_quad}.

\begin{lem}
    \label{lem_blocked_quad}
    Suppose that \( \Sigma \) is orientable, \( G \) is \( (2,2) \)-tight
    and \( Q \) is a quadrilateral face of \( G \) such that
    neither \( G_{Q,v_1,v_3} \) nor \( G_{Q,v_2,v_4} \) is 
    \( (2,2) \)-sparse. Then \( Q \) has a non degenerate 
    boundary. Furthermore, if \( H_1 \) and \( H_2 \) are 
    blockers for \( G_{Q,v_1,v_3} \) respectively \( G_{Q,v_2,v_4} \),
    then both \( H_1 \) and \( H_2 \) are type 2 blockers and 
    \( H_1 \cap H_2 = \emptyset \).
\end{lem}

\begin{proof}
    The non degeneracy of the boundary walk of \( Q \)  
    follows immediately from 
    Lemmas \ref{lem_quad_nondegen_1}, \ref{lem_orient_no_repeat_edges} 
    and \ref{lem_degen_quad}.
    
    Now suppose that one of the blockers, say \( H_1 \), is of type 1 and 
    suppose that \( v_2 \not\in H_1 \). Then \( v_4 \in H_1 \cap H_2
    \).
    So 
    \( \gamma(H_1 \cup H_2) =  \gamma(H_1) +\gamma(H_2) -\gamma(H_1\cap H_2)
    \leq 2+\gamma(H_2) -2 = \gamma(H_2) \). Now if \( H_2 \) is also type 
    1 then \( \gamma(H_1 \cup H_2) = 2 \). However \( v_1,v_2,v_3 \in 
    H_1\cup H_2 \)
    but \( H_1 \cup H_2 \) does not contain one of \( e_1,e_2 \)
    which contradicts the sparsity of \( G \). Similarly if \( H_2 \)
    is type 2, then \( \gamma(H_1 \cup H_2) \leq 3 \), but \( H_1 \cup 
    H_2\) does not contain either of \( e_1,e_2 \), again contradicting 
    the sparsity of \( G \). 
    
    So both \( H_1 \) and \( H_2 \) are type 2 blockers. Moreover
    \( v_1,v_2,v_3,v_4 \in H_1 \cup H_2 \) but \( e_1,e_2,e_3,e_4 \not\in 
    H_1\cup H_2\). Now
    \begin{eqnarray*}
        2 &\leq& \gamma(H_1 \cup H_2 \cup \{e_1,e_2,e_3,e_4\})\\
        &=&  \gamma(H_1) +\gamma(H_2) -\gamma(H_1 \cap H_2) -4 \\
        & =& 2 - \gamma(H_1 \cap H_2) \\
    \end{eqnarray*}
    So \( \gamma(H_1 \cap H_2) \leq 0 \) which implies that 
    \( H_1 \cap H_2 = \emptyset\).
\end{proof}

\subsection{Simple loops in surfaces}
\label{subsec_simple_loops}

Now we briefly digress to review some necessary terminology and facts from low dimensional topology.
Proofs of all of the assertions below can be found in (or at least 
easily deduced from) many sources (for example \cite{MR2850125}).
A loop in a surface \( \Sigma \) is a continuous function \( \alpha:
S^1 \rightarrow \Sigma\). We say that \( \alpha \) is simple if it is
injective. We say that \( \alpha \) is non separating if \( \Sigma - \alpha(S^1)\)
has the same number of connected components as \( \Sigma \). 
Given simple loops \( \alpha,\beta \) in \( \Sigma \), recall 
that the geometric intersection number is defined 
by 
\[ i(\alpha,\beta) = \min |\alpha'(S^1) \cap \beta'(S^1)|  \]
where \( \alpha' \), respectively \( \beta' \), varies over 
all simple loops that are homotopic to \( \alpha \), respectively
\( \beta \). 
If \( i(\alpha,\beta) \neq 0 \) then both \( \alpha \) and \( \beta \)
are essential: that is to say they are not null homotopic.
If \( i(\alpha,\beta) = 1 \) then both \( \alpha \) and \( \beta \)
are non separating in \( \Sigma \). 
In the special case that \( \Sigma \) is the torus, 
if \( i(\alpha,\beta) = 0  \) and \( i(\beta,\delta) = 0 \) then 
\( i(\alpha,\delta) = 0 \).

Given 
a simple loop \( \alpha \) in a surface \( \Sigma \) 
we say that \( \Sigma - \alpha(S^1) \) is the surface 
obtained by cutting along \( \alpha \).
Given a surface \( \Sigma \) with boundary we can cap a boundary
component by gluing a copy of a closed disc to the surface along the 
given boundary component.

If \( \Sigma \) is an orientable surface
of genus \( g \) and \( \alpha \)
is a non separating simple loop in \( \Sigma \) then we form 
\( \Sigma^\alpha \) by removing a tubular neighbourhood of \( \alpha \) and then 
capping the two resulting new boundary components. Clearly \( \Sigma^\alpha \)
is an orientable surface of genus \( g-1 \).

Suppose that 
\( G \) is a \( \Sigma \)-graph and let \( F \) 
be a face of \( G \). Further suppose that \( \alpha \) is a non separating 
loop in \( \Sigma \) such that \( \alpha(S^1) \subset F \). 
By cutting and capping \( \Sigma \) along \( \alpha \) we can form 
a \( \Sigma^\alpha \)-graph, denoted \( G^\alpha \), which has the 
same underlying graph as \( G \). 
Observe that all faces of \( G^\alpha \) except the one(s) corresponding 
to \( F \) are also faces of \( G \).

Finally some terminology. If \( G = (\Gamma,\varphi)\) 
is a \( \Sigma \)-graph 
and \( \alpha \) is a loop in \( \Sigma \), we say that \( \alpha \)
is contained in \( G \) if \( \alpha(S^1) \subset \varphi(|\Gamma|) \).

Now we return to the situation of Lemma \ref{lem_blocked_quad}. 
Suppose that $Q$ is a quadrilateral in $G$ as in the statement of that lemma.
We say that the quadrilateral \(Q\) is blocked. By Lemma \ref{lem_connected} the blocker \(H_1\)
is connected so it is possible to find a simple walk from \(v_1\) to \(v_3\) in 
\(H_1\). By concatenating the geometric realisation of this walk with the diagonal of \(Q\) joining \(v_3\) and 
\(v_1\) we obtain a simple loop in \(\Sigma\), which we denote by \( \alpha_1\). 
Note that we can choose different parameterisations of this loop, but this ambiguity 
will make no difference in our context.
Similarly we construct another simple loop, denoted \(\alpha_2\), by concatenating 
a walk in \(H_2\) with the diagonal of \(Q\) that joins \(v_4\) and \(v_2\). Now since 
\(H_1 \cap H_2\) is empty by Lemma \ref{lem_blocked_quad}, we can choose these loops so that they intersect 
transversely at exactly one point (where the diagonals meet). Thus these loops 
have geometric intersection number equal to one. In particular, we note that
both \( \alpha_1 \) and \( \alpha_2\) must be non separating loops in \( \Sigma \). These loops will play an important role in the following sections.

\section{Irreducible surface graphs}
\label{sec_irredsurface}

Let \( G \) be a \( (2,2) \)-tight \( \Sigma \)-graph. In light 
of Lemmas 
\ref{lem_digon_contractible} and \ref{lem_tri_contractible} we say that 
\( G \) is irreducible if it has no digons, no triangles and if,
for every quadrilateral face of \( G \), both of the possible contractions
result in graphs that are not \( (2,2) \)-sparse.

For each of the contractions described in Section 
\ref{sec_inductive} there are the corresponding vertex splitting 
moves. More precisely, if \( G' = G_D \), 
respectively \( G' = G_{T,e} \), respectively 
\( G' = G_{Q,u,v} \) for some digon \( D \), respectively triangle 
\( T \) and edge \( e \in \partial T\), 
respectively quadrilateral \( Q \) and vertices \( u,v \in \partial Q \),
then we say that \( G \) is obtained from \( G' \) by a digon,
respectively triangle, respectively quadrilateral split. 
Thus every \( (2,2) \)-tight \( \Sigma \)-graph can be constructed 
from some irreducible by applying a sequence of digon/triangle/quadrilateral
splits. Our goal is to identify, for various surfaces, the set of 
irreducibles. 

\begin{conj}
    \label{conj_finite}
    If  \( \Sigma \) is a surface with finite genus and finitely many 
    boundary components and punctures,
    then there are finitely many distinct isomorphism classes of 
    irreducible \( (2,2) \)-tight \( \Sigma \)-graphs.
\end{conj}

We will address some special cases of Conjecture \ref{conj_finite} in this and later sections.
Let \( \mathbb S \) be the \( 2 \)-sphere. 

\begin{thm}
    \label{thm_sphere_irred}
    If \( G \) is a \( (2,2) \)-tight \( \mathbb S \)-graph with at least 
    two vertices then \( G \) has at least two faces of 
    degree at most 3. In particular, any \( (2,2) \)-tight \( \mathbb S \)-graph 
    can be constructed from a single vertex by a sequence of digon and/or
    triangle splits.
\end{thm}

\begin{proof}
    By Lemma 
    \ref{lem_connected}, \( G  \) is connected and therefore
    cellular. Since \( G \) has at least two vertices, \( f_0 = 0 \).
    Also 
    \(  f_1 = 0 \) by sparsity, so by Theorem \ref{thm_euler_sparsity},
    we see that \( 2f_2 +f_3 \geq 4 \).
\end{proof}

The case of plane graphs is similarly straightforward. 
\begin{cor}
    \label{cor_plane_irred}
    If \( G \) is a \( (2,2) \)-tight \( \mathbb R^2 \)-graph with at least 
    two vertices then \( G \) has at least one cellular face of 
    degree at most 3. In particular, any \( (2,2) \)-tight \( \mathbb R^2 \)-graph 
    can be constructed from a single vertex by a sequence of digon and/or
    triangle splits.
\end{cor}

\begin{proof}
    Cap (i.e fill in the puncture of) the non cellular face of \( G \) and then apply Theorem \ref{thm_sphere_irred}.
\end{proof}

Theorem \ref{thm_sphere_irred} and Corollary \ref{cor_plane_irred} 
are implicit already in other places in the literature, we include them
here for completeness. We note that in both cases the quadrilateral 
splitting move is not required in the inductive characterisation. 

Now 
let \( \mathbb A \) be the twice punctured sphere \( \mathbb R^2 - \{(0,0)\} \).  Observe that for any positive integer \( n \), 
it is straightforward to construct an \( \mathbb A \)-graph
that has no digons or triangles, but has \( n \) quadrilateral 
faces. So, in contrast to the cases of the sphere or plane, we do require 
the quadrilateral contraction move 
in order to have finitely many irreducible
\( (2,2) \)-tight \( \mathbb A \)-graphs.

There are two obvious examples of 
irreducible \( (2,2) \)-tight \( \mathbb A \)-graphs, with one vertex 
and two vertices respectively:
see Figure \ref{fig_annular_irreds}.

\begin{thm}
    \label{thm_ann_irred}
    If \( G \) is an irreducible \( (2,2) \)-tight \( \mathbb A \)-graph,
    then \( G \) is isomorphic to one of the \( \mathbb A \)-graphs shown in Figure 
    \ref{fig_annular_irreds}.
\end{thm}

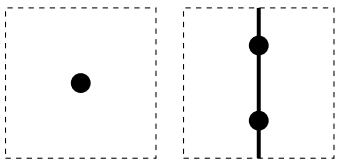
\begin{figure}
    \centering
\begin{tabular}{ccccc}
\begin{tikzpicture}[x=2.0cm,y=2.0cm]
\draw[dash pattern=on 2pt off 2pt] (0,0) rectangle (1,1);
\draw [fill=black,line width=1.5pt] (0.5,0.5) circle (3pt);

\draw (0.5,-0.15) node {$$};

\end{tikzpicture}
&
\begin{tikzpicture}[x=2.0cm,y=2.0cm]
\draw[dash pattern=on 2pt off 2pt] (0,0) rectangle (1,1);
\draw [fill=black,line width=1.5pt] (0.5,0.25) circle (3pt);
\draw [fill=black,line width=1.5pt] (0.5,0.75) circle (3pt);
\draw [line width=1.5pt] (0.5,0.25)--(0.5,0.75);
\draw [line width=1.5pt] (0.5,0.25)--(0.5,0.0);
\draw [line width=1.5pt] (0.5,1.0)--(0.5,0.75);

\draw (0.5,-0.15) node {$$};

\end{tikzpicture}
&
\end{tabular}

\caption{The two non cellular irreducible torus graphs. Here and 
in subsequent diagrams we use the standard representation of the torus as a square with opposite edges identified appropriately. Note that by cutting the torus along a non separating loop these graphs can also be viewed as graphs in the twice punctured sphere.}
    \label{fig_annular_irreds}
\end{figure}

\begin{proof}
    There are two cases to consider. First suppose that \( G \) 
    does not separate the two punctures of \( \mathbb A \). 
    In this it follows easily from Theorem \ref{thm_sphere_irred}
    that if $G$ has at least vertices then it has at least one triangular 
    face and so is not irreducible.

    Now suppose that \( G \) does separate the punctures of \( \mathbb A \).
    Clearly \( G \) has exactly two non cellular faces.
    By capping these 
    two faces, we create a \( (2,2) \)-tight \( \mathbb S \)-graph
    \( \tilde G \). This graph satisfies \( 2f_2+f_3 = 4+f_5+2f_6+
    \cdots \) and 
    since all but two of the faces of \( \tilde G\) are also 
    faces of the irreducible \( G \), it follows that the 
    two exceptional faces of \( \tilde G  \) are digons and 
    all other faces are quadrilateral faces of \( G \). 
    Thus it suffices to show that 
    there cannot be any quadrilateral faces in \( G \). 

    For a contradiction, suppose that \( Q \) is a quadrilateral.
    Since \( G \) is irreducible, both possible contractions of 
    \( Q \) are blocked and we infer the existence of 
    simple loops \(\alpha_1  \) and \( \alpha_2 \) as described 
    at the end of Section \ref{sec_inductive}. Recall that
    these loops intersect transversely at 
    exactly one point and thus
    \( \alpha_1 \) is non separating in \( \mathbb A \). However
    the Jordan Curve Theorem tells us 
    that any simple loop in \( \mathbb A \) must be 
    separating.
\end{proof}

\section{Subgraphs of irreducibles}
\label{sec_subgraphs}

Throughout this section, let \( \Sigma \) be an orientable boundaryless
surface and let \( G = (\Gamma,\varphi) \) 
be an irreducible \( (2,2) \)-tight 
\( \Sigma \)-graph. The goal of this section is to show that any $(2,2)$-tight
subgraph of $G$ is also irreducible.

Let \( H = (\Lambda,\varphi|_{|\Lambda|})\) be a 
subgraph of \( G \). 
We say that \( H \) is inessential if there is some embedded open 
disc \( U \subset \Sigma \) such that \( \varphi(|\Lambda|) \subset
U\). If there is no such disc then \( H \) essential.
Observe that if \( F \) is a cellular face of \( G \) that
has a non degenerate boundary walk, then \( \partial F \)
is inessential: let \( U \) be an open disc neighbourhood of 
the embedded closed disc \( \overline F \). 
We also note that if 
\( H \) is inessential and connected 
then it
has at most one non cellular face \( F \). Moreover
if 
we cut and cap along a maximal non separating set of loops in \( F \)
we obtain an \( \mathbb S \)-graph which, in this 
section, we will denote by \(\hat H\).

Let \( K_1 \) be the graph with one vertex and no edges.
Let \( K_2 \) be the complete graph on two vertices. For 
\( n \geq 2 \) let 
\( \cyc_n \) be the \( n \)-cycle graph
(in particular \( \cyc_2 \) has exactly two parallel
edges).

\begin{lem}
    \label{lem_bicycles_are_essential}
    Suppose that \( H \) is a subgraph of \( G \) whose 
    underlying graph is isomorphic to either \( \cyc_2 \)
    or \( \cyc_3 \). Then 
    \( H \) is essential.
\end{lem}

\begin{proof}
    Suppose that the underlying graph of \( H \) is isomorphic to
    \( \cyc_2 \). The other case is similar.
    Suppose that \( H \) is inessential. Let \( U \) be an
    open disc that contains \( \varphi(\Lambda) \). Clearly there is 
    a digon face \( D \) of \( H \) that is contained in \( U \). 
    Now let  \( K \) be the \( \mathbb S \)-graph obtained 
    by cutting and capping the external face of \( \intt_G(D) \). 
    By Theorem \ref{thm_hole_filling}, \( \gamma(K) =2\) and 
    by Theorem \ref{thm_euler_sparsity}, \( K \) has at least two faces
    of degree at most 3. One of these faces is also a face of \( G \)
    contradicting the irreducibility of \( G \). 
\end{proof}

\begin{lem}
    \label{lem_iness_tight}
    Suppose that \(H \) is an inessential subgraph of \( G \) and that
    \( \gamma(H) = 2 \). Then the underlying graph of \( H \) is \( K_1 \).
\end{lem}

\begin{proof}
    Suppose that \( H \) has at least two vertices. Then by 
    Theorem \ref{thm_euler_sparsity}, \( \hat H \) has at least
    two faces of degree at most 3. If one of these is a triangle
    or a digon with non degenerate boundary then the underlying graph
    of \( H \) contains a copy of \( \cyc_2 \) or \( \cyc_3 \)
    which contradicts
    Lemma \ref{lem_bicycles_are_essential}.
    Therefore \( \hat H \) must have two digon faces
    both of which have degenerate boundaries. 
    However, as pointed out in Section \ref{subsec_digon_contr},
    no \( \mathbb S \)-graph can have more than one 
    degenerate digon.
\end{proof}

\begin{lem}
    \label{lem_iness_3}
    Suppose that \(H \) is an inessential subgraph of \( G \) and that
    \( \gamma(H) = 3 \). Then the underlying graph of \( H \) is \( K_2\).
\end{lem}

\begin{proof}
    By Theorem \ref{thm_euler_sparsity}, \( \hat H \) satisfies
    \( 2f_2 +f_3 = 2 +f_5+2f_6 +\cdots \). As in the proof 
    of Lemma \ref{lem_iness_tight} we see that \( \hat H \) cannot 
    have a triangle or a digon with non degenerate boundary. So the 
    only possibility is that \( \hat H \) has a digon face with 
    degenerate boundary. As pointed out in Section 
    \ref{subsec_digon_contr},, there is only one 
    \( \mathbb S \)-graph with a degenerate digon face and its
    underlying graph is indeed \( K_2 \). 
\end{proof}

The case of a subgraph isomorphic to \( \cyc_4 \) is a little 
more involved.

\begin{lem}
    \label{lem_4cycle_is_face}
    Suppose that \( H \) is an inessential subgraph of \( G \) whose 
    underlying graph is isomorphic to \( \cyc_4 \). Then  
    \( H \) is the boundary of
    some quadrilateral face of \( G \). 
\end{lem}

\begin{proof}
    Suppose that \( U \) is an embedded disc containing \( \varphi(|\Lambda|) \)
    and let \( R \) be the face of \( H \) that is contained in \( U \).
    First observe that \( \gamma(H) = 4 \), so by Theorem 
    \ref{thm_hole_filling}, \( \gamma(\intt_G(R)) \leq 4 \).
    Now, by Lemma \ref{lem_bicycles_are_essential}, 
    \(\intt_G(R)\) has no digons or triangles and it follows 
    easily from Theorem \ref{thm_euler_sparsity} that 
    \( \gamma(\intt_G(R)) = 4 \) and that 
    all the 
    cellular faces of \( \intt_G(R) \) are quadrilaterals: that is
    to say that \( \intt_G(R) \) is in fact a quadrangulation of 
    \( \overline R \). 

    Now, let \( Q \) (with boundary vertices \( v_1,v_2,v_3,v_4\ \))
    be a quadrilateral face of \( \intt_G(R) \)
    that is contained in \( R \). 
    Since \( G \) is irreducible, we have blockers \( H_1 \) and 
    \( H_2 \) for the two possible contractions of \( Q \),
    as described in Lemma \ref{lem_blocked_quad}.
    Also we have simple loops \( \alpha_1 \) and \( \alpha_2 \) 
    as described in Section \ref{sec_inductive}. These loops intersect
    transversely at one point in \( Q \). 
    If \( w_1,w_2,w_3,w_3 \) are the vertices of \( \partial R \) in 
    cyclic order,
    it follows that one of the loops, say \( \alpha_1 \), contains
    \( w_1 \) and \( w_2 \) and that \( \alpha_2 \) contains 
    \( w_2 \) and \( w_4 \). Thus \( \alpha_2 \) divides \( R \) into 
    disjoint open subsets \( R_{1} \) and \( R_{3} \) (see Figure 
    \ref{fig_quadproof})
    where \( w_1,v_1 \in \overline{R_1} \) and \( w_3,v_3 \in \overline{R_3}\).
    Now we can decompose the blocker \( H_1 \) as 
    \( K_e \cup K_1 \cup K_3  \), where 
    \( K_e =\extt_G(R)\cap H_1 \), 
    \( K_1 \) is the part of \( H_1 \)
    contained in \( \overline R_1 \) and \( K_3 \) is the part of \( H_1 \)
    contained in \( \overline R_3 \). It is clear that 
    \( K_e \cap K_1 = \{w_1\} \) and \( K_e \cap K_3 = \{w_3\} \). 
    Therefore, by (\ref{lem_incl_excl}), 
    \[ 3 = \gamma(H_1) = \gamma(K_e) + \gamma(K_1) + \gamma(K_3) - 4. \]
    Using the sparsity of \( G \) it follows that at least 
    one of \( \gamma(K_1) \) or \( \gamma(K_3) \) is equal to \( 2 \).
    Now \( K_1 \) and \( K_3 \) are both inessential subgraphs
    of \( G \) since \( \overline R_1 \) and \( \overline R_3 \) 
    are both embedded closed discs in \( \Sigma \).  
    It follows from Lemma \ref{lem_iness_tight} that 
    at least one of \( K_1 \) or \( K_3 \) is a single vertex. 
    So either \( v_1 = w_1 \) or \( v_3 = w_3 \). 
    We have shown that at least one of \( v_1 \) or \( v_3 \) 
    actually lies in the boundary of \( R \). Similarly
    at least one of \( v_2 \) or \( v_4 \) lies in the boundary 
    of \( R \). 
    
    Thus we have shown that if \( Q \) is any quadrilateral 
    face of \( G \) contained in \( R \) then \( \partial Q \) 
    and \( \partial R \) share at least one edge. Now it is 
    an elementary exercise to show that in any  quadrangulation 
    of \( R \) that has this property, either there are no
    quadrilaterals properly contained in \( R \), or 
    some quadrilateral has a boundary vertex with degree 2. 
    Clearly, by Lemma \ref{lem_blocked_quad}, no quadrilateral 
    face of the irreducible graph \( G \) can have a boundary 
    vertex of degree 2. It follows that there are no quadrilateral
    faces of \( G \) that are properly contained in \( R \)
    and so \( R \) is itself a face of \( G \).
\end{proof}

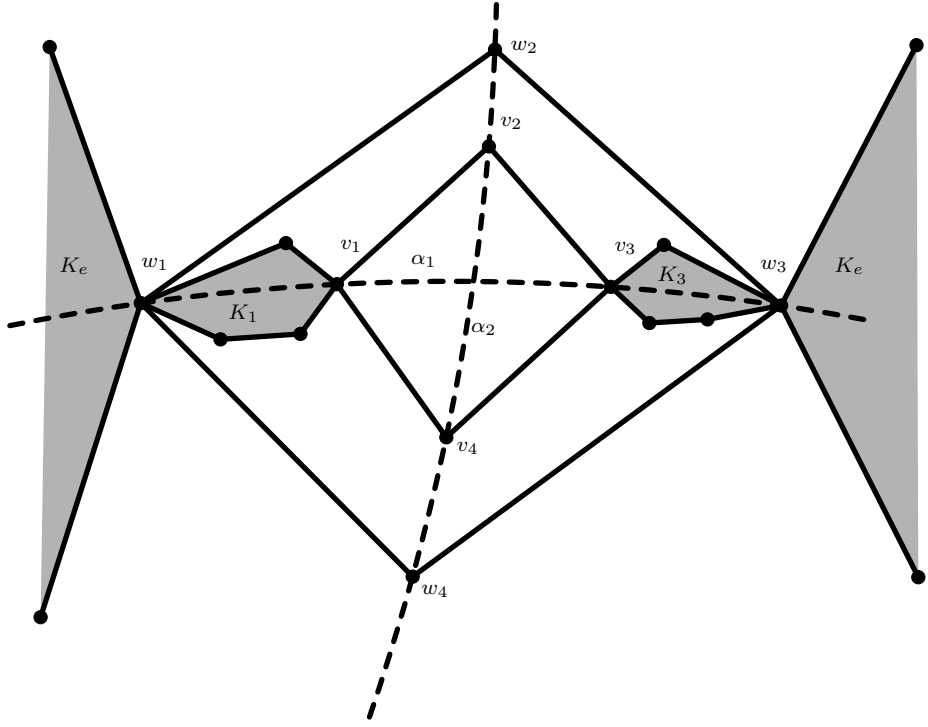
\begin{figure}
    \centering

\definecolor{zzttqq}{rgb}{0,0,0}
\definecolor{ududff}{rgb}{0,0,0}
\definecolor{xdxdff}{rgb}{0,0,0}
\def\myopac{0.3}
\begin{tikzpicture}[line cap=round,line join=round,>=triangle 45,x=1cm,y=1cm]
\fill[line width=2pt,color=zzttqq,fill=zzttqq,fill opacity=\myopac] (-3.766198255563519,1.9305569759110242) -- (-2.7066636415201377,1.4477118547929682) -- (-1.643675553034602,1.5201883153715274) -- (-1.159711850809439,2.1814620473837465) -- (-1.8369461145774268,2.7281293250141823) -- cycle;
\fill[line width=2pt,color=zzttqq,fill=zzttqq,fill opacity=\myopac] (2.4839209603739953,2.1465548802386145) -- (2.9948179239931902,1.665141236528646) -- (3.767900170164489,1.7134588769143522) -- (4.746353874433667,1.8981636352238027) -- (3.188088485536015,2.703970504821329) -- cycle;
\fill[line width=2pt,color=zzttqq,fill=zzttqq,fill opacity=\myopac] (-4.97759,5.33728) -- (-3.766198255563519,1.9305569759110242) -- (-5.098386840612593,-2.248587634713556) -- cycle;
\fill[line width=2pt,color=zzttqq,fill=zzttqq,fill opacity=\myopac] (6.5461644923425935,5.36144072603517) -- (4.746353874433667,1.8981636352238027) -- (6.5703233125354465,-1.7170935904707878) -- cycle;
\draw [dash pattern = on 5pt off 5pt, shift={(0.37829817411680644,-27.459886145466914)},line width=2pt]  plot[domain=1.3855391648521445:1.772966404807989,variable=\t]({1*29.681222978957333*cos(\t r)+0*29.681222978957333*sin(\t r)},{0*29.681222978957333*cos(\t r)+1*29.681222978957333*sin(\t r)});
\draw [dash pattern = on 5pt off 5pt, shift={(-30.71956369149556,6.655895105474076)},line width=2pt]  plot[domain=5.954392706513682:6.262920907441633,variable=\t]({1*31.691547498590143*cos(\t r)+0*31.691547498590143*sin(\t r)},{0*31.691547498590143*cos(\t r)+1*31.691547498590143*sin(\t r)});
\draw [line width=2pt] (-1.159711850809439,2.1814620473837465)-- (0.8617051700591674,4.014372826637693);
\draw [line width=2pt] (0.8617051700591674,4.014372826637693)-- (2.4839209603739953,2.1465548802386145);
\draw [line width=2pt] (2.4839209603739953,2.1465548802386145)-- (0.2958840648098757,0.14465994374911073);
\draw [line width=2pt] (0.2958840648098757,0.14465994374911073)-- (-1.159711850809439,2.1814620473837465);
\draw [line width=2pt] (-3.766198255563519,1.9305569759110242)-- (0.9430576390284813,5.30216199468775);
\draw [line width=2pt] (0.9430576390284813,5.30216199468775)-- (4.746353874433667,1.8981636352238027);
\draw [line width=2pt] (4.746353874433667,1.8981636352238027)-- (-0.15196903877681578,-1.7092910850703875);
\draw [line width=2pt] (-0.15196903877681578,-1.7092910850703875)-- (-3.766198255563519,1.9305569759110242);
\draw [line width=2pt,color=zzttqq] (-3.766198255563519,1.9305569759110242)-- (-2.7066636415201377,1.4477118547929682);
\draw [line width=2pt,color=zzttqq] (-2.7066636415201377,1.4477118547929682)-- (-1.643675553034602,1.5201883153715274);
\draw [line width=2pt,color=zzttqq] (-1.643675553034602,1.5201883153715274)-- (-1.159711850809439,2.1814620473837465);
\draw [line width=2pt,color=zzttqq] (-1.159711850809439,2.1814620473837465)-- (-1.8369461145774268,2.7281293250141823);
\draw [line width=2pt,color=zzttqq] (-1.8369461145774268,2.7281293250141823)-- (-3.766198255563519,1.9305569759110242);
\draw [line width=2pt,color=zzttqq] (2.4839209603739953,2.1465548802386145)-- (2.9948179239931902,1.665141236528646);
\draw [line width=2pt,color=zzttqq] (2.9948179239931902,1.665141236528646)-- (3.767900170164489,1.7134588769143522);
\draw [line width=2pt,color=zzttqq] (3.767900170164489,1.7134588769143522)-- (4.746353874433667,1.8981636352238027);
\draw [line width=2pt,color=zzttqq] (4.746353874433667,1.8981636352238027)-- (3.188088485536015,2.703970504821329);
\draw [line width=2pt,color=zzttqq] (3.188088485536015,2.703970504821329)-- (2.4839209603739953,2.1465548802386145);
\draw [line width=2pt,color=zzttqq] (-4.97759,5.33728)-- (-3.766198255563519,1.9305569759110242);
\draw [line width=2pt,color=zzttqq] (-3.766198255563519,1.9305569759110242)-- (-5.098386840612593,-2.248587634713556);
\draw [line width=2pt,color=zzttqq] (6.5461644923425935,5.36144072603517)-- (4.746353874433667,1.8981636352238027);
\draw [line width=2pt,color=zzttqq] (4.746353874433667,1.8981636352238027)-- (6.5703233125354465,-1.7170935904707878);
\begin{scriptsize}
\draw[color=black] (-0.0008757799205922762,2.4918910947249026) node {$\alpha_1$};
\draw[color=black] (0.7938888258650004,1.5805853658536602) node {$\alpha_2$};
\draw [fill=xdxdff] (-1.159711850809439,2.1814620473837465) circle (2.5pt);
\draw[color=xdxdff] (-0.9672285876347158,2.6918910947249026) node {$v_1$};
\draw [fill=xdxdff] (0.8617051700591674,4.014372826637693) circle (2.5pt);
\draw[color=xdxdff] (1.1621123085649435,4.3279614293817385) node {$v_2$};
\draw [fill=xdxdff] (2.4839209603739953,2.1465548802386145) circle (2.5pt);
\draw[color=xdxdff] (2.6807532614861,2.6677322745320495) node {$v_3$};
\draw [fill=xdxdff] (0.2958840648098757,0.14465994374911073) circle (2.5pt);
\draw[color=xdxdff] (0.58230062393646944,0.0025501985252424) node {$v_4$};
\draw [fill=xdxdff] (-3.766198255563519,1.9305569759110242) circle (2.5pt);
\draw[color=xdxdff] (-3.576381168462849,2.4503028927963717) node {$w_1$};
\draw [fill=xdxdff] (0.9430576390284813,5.30216199468775) circle (2.5pt);
\draw[color=xdxdff] (1.3345887691435028,5.332537719795806) node {$w_2$};
\draw [fill=xdxdff] (4.746353874433667,1.8981636352238027) circle (2.5pt);
\draw[color=xdxdff] (4.65168235961429,2.4261440726035186) node {$w_3$};
\draw [fill=xdxdff] (-0.15196903877681578,-1.7092910850703875) circle (2.5pt);
\draw[color=xdxdff] (0.147441860465113894,-1.8976789563244463) node {$w_4$};
\draw [fill=ududff] (-2.7066636415201377,1.4477118547929682) circle (2.5pt);
\draw [fill=ududff] (-1.643675553034602,1.5201883153715274) circle (2.5pt);
\draw [fill=ududff] (-1.8369461145774268,2.7281293250141823) circle (2.5pt);
\draw [fill=ududff] (2.9948179239931902,1.665141236528646) circle (2.5pt);
\draw [fill=ududff] (3.767900170164489,1.7134588769143522) circle (2.5pt);
\draw [fill=ududff] (3.188088485536015,2.703970504821329) circle (2.5pt);
\draw [fill=ududff] (-4.97759,5.33728) circle (2.5pt);
\draw [fill=ududff] (-5.098386840612593,-2.248587634713556) circle (2.5pt);
\draw [fill=ududff] (6.5461644923425935,5.36144072603517) circle (2.5pt);
\draw [fill=ududff] (6.5703233125354465,-1.7170935904707878) circle (2.5pt);
\draw[color=xdxdff] (-4.65168235961429,2.4261440726035186) node {$K_e$};
\draw[color=xdxdff] (5.65168235961429,2.4261440726035186) node {$K_e$};
\draw[color=xdxdff] (-2.4,1.8) node {$K_1$};
\draw[color=xdxdff] (3.3,2.3) node {$K_3$};
\end{scriptsize}
\end{tikzpicture}

\caption{From the proof of Lemma \ref{lem_4cycle_is_face}: the shaded 
region represents the blocker for the contraction \(G_{Q,v_1,v_3}\).}
    \label{fig_quadproof}
\end{figure}

We say that a subgraph \( H = (\Lambda,\varphi|_{|\Lambda|}) \) 
of \( G \) is annular if it is essential 
and \( \varphi(|\Lambda|) \) is contained in
some embedded open annulus of \( \Sigma\). 
Let \( \mathfrak B\) be the (unique) $(2,2)$-tight graph with 
3 vertices, one of which has degree 4.

\begin{lem}
    \label{lem_annular_quad}
    Suppose that \( H \) is a subgraph of \( G \) whose 
    underlying graph is isomorphic to \( \mathfrak B\). Then  
    \( H \) is not annular.
\end{lem}

\begin{proof}
    Suppose, seeking a contradiction, that \( H \) is annular. 
    Let \( U \) be an open annulus containing \( \varphi(|\Lambda|) \) and 
    let \( R \) be the face of \( H \) that is contained in 
    \( U \). Observe that \( \gamma(H) = 2 \), so by 
    Theorem \ref{thm_hole_filling}, \( \gamma(\intt_G(R)) = 2 \). 
    Let \( K \) be the \( \mathbb S \)-graph obtained by 
    cutting and capping the external faces of \( \intt_G(R) \) (there could 
    be more than one in this case). Now \( K \) is a \( (2,2) \)-tight
    \( \mathbb S \)-graph with two digon faces. Since all other faces
    of \( K \) are also faces of the irreducible \( G \), it follows 
    easily from Theorem \ref{thm_euler_sparsity} that all 
    other faces of \( K \) are quadrilaterals. Thus, all faces of 
    \( G \) that are contained in \( R \) are in fact quadrilaterals.

    Now we can argue, using a straightforward
    modification of the argument from 
    the proof of Lemma \ref{lem_4cycle_is_face},
    that any quadrilateral face of \( G \) that is contained in 
    \( R \) must in fact share a boundary edge with \( R \).
    Again, following the proof of Lemma \ref{lem_4cycle_is_face}
    it follows 
    that \( R \) itself must be a face of \( G \).
    However this contradicts Lemma \ref{lem_blocked_quad} 
    where we showed that any quadrilateral face of an irreducible 
    has a non degenerate boundary.
\end{proof}

Now the main result of this section: a tight subgraph of an irreducible
is also irreducible. 

\begin{thm}
    \label{thm_tight_sub_irred}
    Suppose that \( G = (\Gamma,\varphi)\) 
    is an irreducible \( (2,2) \)-tight 
    \( \Sigma \)-graph and \( \Lambda \) is a \( (2,2) \)-tight
    subgraph of \( \Gamma \). Then \( H = (\Lambda,\varphi|_{|\Lambda|}) \)
    is an irreducible \( \Sigma \)-graph.
\end{thm}

\begin{proof}
    We see that 
    \( H \) cannot have any triangle or digon, since
    the boundary of such a face would contradict Lemma 
    \ref{lem_bicycles_are_essential}.
    Now suppose that \( Q \) is a quadrilateral face 
    of \( H \). 
    It is not clear, a priori, that 
    the boundary of \( Q \) is non degenerate, so we must 
    prove that before proceeding.
    
    Applying Lemma \ref{lem_orient_no_repeat_edges}
    to \( H \), we see that there are no repeated edges
    in the boundary of \( Q \). Thus the only possibility for 
    a degenerate boundary is that one vertex is repeated and 
    that \( \partial Q \) has underlying graph isomorphic to 
    \( \mathfrak B \). 
    If \( \partial Q \) is inessential then, since \( \mathfrak B \)
    contains a copy of \( \cyc_2 \), this contradicts Lemma
    \ref{lem_bicycles_are_essential}. On the other hand, if \( \partial Q \)
    is inessential then it must be annular and this contradicts Lemma
    \ref{lem_annular_quad}.
    Thus we see that in fact 
    \( Q \) must have a non degenerate boundary.

    By Lemma \ref{lem_4cycle_is_face} this means that \( Q \) is
    also a face of \( G \) and so there are blockers \( H_1,H_2 \)
    as described by Lemma \ref{lem_blocked_quad}.
    Now consider the \( \Sigma \)-graph 
    \( K = H_1 \cup H_2 \cup \partial Q\). This is \( (2,2) \)-tight, 
    so, by Lemma \ref{lem_int_unions}, \( K\cap H \) is also 
    \( (2,2) \)-tight. Now, \( K \cap H = (H_1 \cap H) \cup (H_2 \cap H) 
    \cup \partial Q\). Using (\ref{lem_incl_excl}),
    \( H_1 \cap H_2 = \emptyset \), \( H_1 \cap H \cap \partial Q 
    = \{v_1,v_3\}\) and \( H_2 \cap H \cap \partial Q = \{v_2,v_4\} \),
    we have
    \begin{eqnarray*}
        2 &=& \gamma(K\cap H) \\
        &=& \gamma(\partial Q) +\gamma(H_1 \cap H) +\gamma(H_2 \cap H) - 
        \gamma(H_1\cap H \cap \partial Q) - \gamma(H_2\cap H \cap \partial Q)\\
        &=& 4+ \gamma(H_1\cap H) + \gamma(H_2 \cap H) -4 - 4. \\
    \end{eqnarray*}
    Thus \( \gamma(H_1 \cap H) +\gamma(H_2 \cap H) = 6 \). 
    If \( \gamma(H_1 \cap H) =2 \) then \( (H_1\cap H) \cup
    \{v_2\}\cup\{e_1,e_2\}\) would be a type 1 blocker for
    the contraction \( G_{Q,v_1,v_3} \), contradicting 
    Lemma \ref{lem_blocked_quad}. So \( \gamma(H_1 \cap H) \geq 3 \)
    and similarly \( \gamma(H_2 \cap H) \geq 3 \). It follows that 
    \( \gamma(H_1\cap H) = \gamma(H_2 \cap H) = 3\) and that 
    \( H_1 \cap H \) and \( H_2 \cap H \) are blockers for 
    the contractions \( H_{Q,v_1,v_3} \) and \( H_{Q,v_1,v_3} \)
    respectively. Thus both possible contractions of \( Q \) 
    are blocked in \( H \) as required.
\end{proof}

For example, suppose that \( \Gamma \) is the simple \( (2,2) \)-tight 
graph obtained 
by adding a vertex of degree two to \( K_4 \). Is it 
possible to embed \( \Gamma \) into the torus to create
an irreducible torus graph? If so, how many nonisomorphic embeddings exist?
There are several possible embeddings of \( \Gamma \)to consider, however we can 
significantly narrow the search space by observing that since \( K_4 \)
is tight, by Theorem \ref{thm_tight_sub_irred}, any irreducible 
embedding of \( \Gamma \) must extend an irreducible embedding of 
\( K_4 \). 
It is not difficult to show that, up to isomorphism there is a unique
irreducible embedding of \( K_4 \) in the torus  (see figure \ref{fig_torus4}).
Thus any irreducible embedding of $\Gamma$ must restrict to this 
embedding of $K_4$. Using this observation, it is not difficult to show that, up to 
isomorphism there are exactly two distinct irreducible torus embeddings of 
$\Gamma$.

\section{Irreducible torus graphs}
\label{sec_irred_torus}

Let \( \mathbb T = S^1 \times S^1 \) be the torus. 
Throughout this 
section let \( G = (\Gamma,\varphi) \) be an irreducible \( (2,2) \)-tight 
\( \mathbb T \)-graph. 
Our goal in this section is to show that there are only 
finitely many isomorphism classes of such graphs by 
establishing an upper bound for the number of vertices of \( G \).

In the case that \( G \) is not cellular we will see that we 
can essentially reduce the problem to the sphere or the annulus.
If \( G \) is cellular then 
using Theorem \ref{thm_euler_sparsity} and \( f_2 = f_3 = 0 \) 
we see that \( G \)
satisfies \( f_5 + 2f_6 + 3f_7 +4f_8 = 4 \) and \( f_i = 0\) for \( 
i\geq 9\). Since \( |V| = 2+\sum_{i \geq 2} f_i  \), the problem reduces 
to establishing a bound for the 
number of quadrilateral faces that an irreducible 
\( \mathbb T \)-graph can have. 

First we deal with the non cellular case.

\begin{lem}
    \label{lem_non_cell}
    Suppose that \( G \) is not cellular. 
    Then \( \Gamma  \) is either isomorphic to \( K_1 \) or to \( \cyc_2 \). Furthermore, in the latter case, \( G \) is annular.
\end{lem}

\begin{proof}
    Since \( \Gamma \) is connected
    it is clear \( G \) has a single non cellular face. 
    By cutting along a non separating loop in this face we obtain 
    an \( \mathbb A \)-graph \( \hat G \). Observe that 
    any face of \( \hat G \) that is not also a face of \( G \) is 
    non cellular. It follows that \( \hat G \) is an 
    irreducible \( \mathbb A \)-graph.
    Now the conclusion follows from Theorem \ref{thm_ann_irred}.
\end{proof}

For the remainder of the section, assume that \( G \) is cellular.
Let \( Q \) be a quadrilateral face of \( G \) with boundary 
walk \( v_1,e_1,v_2,e_2,v_3,e_3,v_4,e_4,v_1 \). As described in 
Section \ref{sec_inductive}
we have blockers \( H_1 \) and \( H_2 \) and simple loops
\( \alpha_1 \) and \( \alpha_2 \) that intersect transversely at 
one point. 

\begin{lem}
    \label{lem_one_blocker_inessential}
    At least one of \( H_1 \) or \( H_2 \) 
    is an inessential subgraph of \( G \).
\end{lem}

\begin{proof}
    Suppose that both are essential. Then 
    there are non separating simple loops \( \beta_1 \) contained in 
    \( H_1 \) and \( \beta_2 \) contained in \( H_2 \).
    Now \( H_1 \cap H_2 = \emptyset \), so 
    \( i(\beta_1,\beta_2) = 0 \). However, it is also clear that 
    \( i(\alpha_1,\beta_2) = i(\alpha_2,\beta_1) = 0 \). As 
    pointed out in Section \ref{subsec_simple_loops} this implies 
    that \( i(\alpha_1,\alpha_2) = 0 \),
    contradicting the fact that 
    these curves intersect transversely at one point.
\end{proof}

For the remainder of the section, suppose that \( H_1 \) is an 
inessential blocker. By Lemma \ref{lem_iness_3}, the graph 
of \( H_1 \) is \( K_2 \).
Furthermore we will assume that 
\( H_2 \) is a maximal blocker with respect to inclusion and let
\( J \) be the face of \( H_2 \),
that contains \( v_1,v_3 \).
See Figure \ref{fig_quadstructure} 
for an illustration of these assumptions in the case where \( H_2 \)
is an essential blocker.

\begin{figure}
    \centering
\newcommand{\ew}{2pt}
\newcommand{\ra}{5pt}
\begin{tikzpicture}[x=0.6cm,y=0.6cm]
    \coordinate (v1) at (3,5);
    \coordinate (v2) at (5,7.5);
    \coordinate (v3) at (7,5);
    \coordinate (v4) at (5,2.5);

    \coordinate (A1) at (5,0);
    \coordinate (A2) at (5,10);

    \coordinate (B1) at (1.5,0);
    \coordinate (B2) at (1.5,10);

    \coordinate (C1) at (8.5,0);
    \coordinate (C2) at (8.5,10);

    \draw [line width = \ew] (v1)--(v2)--(v3)--(v4)--(v1);

    \draw [fill=black] (v1) circle (\ra);
    \draw (v1)+(-0.1,0.6) node {\relax $v_4$};
    \draw [fill=black] (v2) circle (\ra);
    \draw (v2)+(0.6,0) node {\relax $v_1$};
    \draw [fill=black] (v3) circle (\ra);
    \draw (v3)+(-0.1,0.6) node {\relax $v_2$};
    \draw [fill=black] (v4) circle (\ra);
    \draw (v4)+(-0.6,0) node {\relax $v_3$};

    \path [fill=black!20!white] (v1) .. controls +(-1,0) and +(0,-1) .. (B2) -- (0,10) -- (0,0) -- (B1) .. controls +(0,1) and +(-1,0) .. (v1);
    \path [line width=\ew, draw] (B1) .. controls +(0,1) and +(-1,0) .. (v1) .. controls +(-1,0) and +(0,-1) .. (B2); 
    \draw (0.9,6) node {\large $H_2$};

    \path [fill=black!20!white] (v3) .. controls +(1,0) and +(0,-1) .. (C2) -- (10,10) -- (10,0) -- (C1) .. controls +(0,1) and +(1,0) .. (v3);
    \path [line width=\ew, draw] (C1) .. controls +(0,1) and +(1,0) .. (v3) .. controls +(1,0) and +(0,-1) .. (C2); 

    \path [line width = \ew,draw] (v1) .. controls +(-2,0) and +(1,0.2) .. (0,4);
    \path [line width = \ew,draw] (10,4) .. controls +(-1,-0.2) and +(1,0) .. (v3);
    \draw [line width = \ew, dash pattern=on 5pt off 5pt] (v1)--(v3);
    \draw (1,4) node {\relax $\alpha_2$};

    \draw [line width = \ew, dash pattern=on 5pt off 5pt] (v4)--(v2);
    \draw [line width = \ew] (v2)--(A2);
    \draw (5.6,9.4) node {\relax $H_1$};
    \draw [line width = \ew] (v4)--(A1);
    \draw (4.53,6.1) node {\tiny $\alpha_1$};

    \draw [line width=1pt,dash pattern=on 5pt off 5pt] (0,0) rectangle (10,10);
    \path (0,-0.5) rectangle (0,-0.5);
\end{tikzpicture}

\begin{caption}{A quadrilateral face with an essential blocker. The shaded region represents the essential blocker \( H_2 \).}
    \label{fig_quadstructure}
\end{caption}
\end{figure}

\begin{lem}
    \label{lem_maximal_blocker}
    Any face of \( H_2 \) that is not \( J \) is also a face of \( G \).
\end{lem}

\begin{proof}
    Suppose that \( F \neq J \) is a face of \( H_2 \). Then 
    \( \gamma(H_2 \cup \intt_G(F)) \leq \gamma(H_2) = 3 \), by Theorem
    \ref{thm_hole_filling}. Also \( v_1,v_3 \not\in \intt_G(F) \), 
    since \( F \neq J \). If follows that \( H_2 \cup \intt_G(F) \) is 
    a blocker for \( G_{Q,v_2,v_4} \) and so by the maximality of 
    \( H_2 \), \( \intt_G(F) \subset H_2 \) as required.
\end{proof}

Next we want to examine the structure of \( H_2 \). It turns out 
that there are exactly ten distinct possibilities.
If \( H_2 \) is inessential then, by Lemma \ref{lem_iness_3}
it has graph \( K_2 \) (Figure \ref{fig_iness_blocker}). 
On the other hand, if \( H_2 \) is essential 
we have the following.

\begin{lem}
    Suppose that \( H_2 \) is essential. Then it is isomorphic to 
    one of the nine torus graphs shown in Figures \ref{fig_ess_no_quad}
    and \ref{fig_ess_quad}.
\end{lem}

\begin{proof}
    Since \( \gamma(H_2) = 3 \), it is connected by Lemma \ref{lem_connected}.
    Let \( K \) be the \( \mathbb S \)-graph obtained by cutting and 
    capping \( H_2 \) along a non separating loop in \( J \). 
    Clearly \( K \)
    has two exceptional faces \( J^+ \) and \( J^- \) such that 
    all other faces of \( K \) are
    faces of \( G \) (using Lemma \ref{lem_maximal_blocker}). 
    Now, since \( J^+ \) and \( J^- \) are the only faces of 
    \( K  \) that could have degree less than 4, Theorem \ref{thm_euler_sparsity} implies that
    \( K \) satisfies 
    \begin{equation}
        \label{eq_ess_blocker}
        2f_2+f_3 = 2 +f_5+2f_6 
    \end{equation}
    and \( f_i = 0 \) for \( i \geq 7 \).
    There are two cases to consider.

    (a) There is no quadrilateral face of \( G \) in \( H_2 \). There 
    are various subcases:
    \begin{enumerate}
        \item \( |J^+| = |J^-| = 2 \). Then, from Equation \ref{eq_ess_blocker}
            we get \( f_5+2f_6 = 2 \). So either
            \( f_5=0 \) and \( f_6 = 1 \) and we have the example shown in Figure
            \ref{fig_ess_no_quad} (a), or,
            \( f_5 = 2 \) and \( f_6 = 0 \) and we have one of the examples 
            shown in Figure \ref{fig_ess_no_quad} (b) or (c).
        \item \( |J^+| = 2 \) and \( |J^-| = 3 \). Then we have \( f_5 = 1 \). There is
            one possibility: Figure \ref{fig_ess_no_quad} (d).
        \item \( |J^+| = |J^-| = 3 \). In this case, Equation \ref{eq_ess_blocker}
            implies that \( J^+ \) and \( J^- \) are the only faces of \( K \). So we have the example shown in Figure \ref{fig_ess_no_quad} (e).
        \item \( |J^+| = 2\) and \(|J^-| = 4 \). In this case, Equation \ref{eq_ess_blocker}
            implies that \( J^+ \) and \( J^- \) are the only faces of \( K \) and we have the example shown in Figure \ref{fig_ess_no_quad} (f).
    \end{enumerate}

    (b) There is some quadrilateral face of \( G \) in \( H_2 \).
    This case requires a little more effort as we must first establish that 
    there is no more than one such face. 
    Let \( G' = \partial Q \cup H_1 \cup H_2 \). Clearly \( G' \) is 
    \( (2,2) \)-tight and so by Theorem \ref{thm_tight_sub_irred} it is also 
    irreducible.

    Suppose that \( R \) is a quadrilateral face of \( G \), with 
    boundary vertices \( w_1,w_2,w_3,w_4 \), that is contained in 
    \( H_2 \) (and so is also a face of \( G' \)). 
    By Lemma \ref{lem_one_blocker_inessential} we know that 
    there is a blocker for one of the contractions of \( R \) in \( G' \)
    whose 
    graph is \( K_2 \). Without loss of generality assume that a blocker 
    \( L_1 \) for the contraction \( G'_{R,w_1,w_3} \) has graph \( K_2 \).
    Now we claim that 
    \( L_1 \subset H_2 \).
    If not then it is clear that \( L_1 \) must intersect \( H_1 \). 
    Since the vertices of \( L_1 \) are both in \( H_2 \) this contradicts
    \( H_1 \cap H_2 = \emptyset \), thus establishing our claim.

    Now consider a maximal blocker, \( L_2 \), for the contraction \( G'_{R,w_2,w_4} \). 
    We have 
    \begin{eqnarray*}
        3 &= & \gamma(L_2) \\
        &= & \gamma(L_2\cap (\partial Q \cup H_1))+ \gamma(L_2\cap H_2) - \gamma(L_2\cap (\partial Q \cup H_1) \cap H_2) \\
        &= & \gamma(L_2\cap (\partial Q \cup H_1))+ \gamma(L_2\cap H_2) - \gamma(L_2\cap \{v_2,v_4\}) \\
    \end{eqnarray*}

    Now it is clear that \( \{v_2 ,v_4\} \subset L_2 \) since 
    \( L_2 \) is connected, so we have 
    \begin{equation}
        \label{eq_caseb1}
        \gamma(L_2 \cap H_2) = 7 - \gamma(L_2\cap(\partial Q \cup H_1))
    \end{equation}
    Furthermore, it is also clear that \( L_1 \) separates \( v_2 \) from 
    \( v_4 \) in \( H_2 \), so \( L_2 \cap H_2 \) has at least two components. 
    Also \( L_2\cap (\partial Q \cup H_1) \) is a subgraph of \( \partial Q \cup H_1 \)
    that contains the vertices \( v_2,v_4 \). It follows easily that 
    \( \gamma(L_2 \cap (\partial Q \cup H_1)) \geq 3  \) with 
    equality only if \( L_2 \cap (\partial Q \cup H_1) = \partial Q \cup H_1 \).
    Therefore the only way that (\ref{eq_caseb1}) can be satisfied is that 
    \( \partial Q \cup H_1 \subset L_2 \) and \(L_2\cap H_2\) has exactly two components
    \( X_2 \ni v_2  \) and \( X_4 \ni v_4 \) such that \( \gamma(X_2) = \gamma(X_4) = 2 \).
    In particular it follows from Theorem \ref{thm_tight_sub_irred} and Lemma 
    \ref{lem_non_cell} that the underlying graph of \( X_2 \), and also of \( X_4 \), is isomorphic to 
    \( K_1 \) or \( \cyc_2 \). Now since \( L_1 \) also separates \( w_2 \)
    and \( w_4 \) in \( H_1 \) we can, without loss of generality, assume that 
    \( v_2,w_2 \in X_2 \) and \( v_4,w_4 \in X_4 \). 

    Let \( Z_2 \), respectively \( Z_4 \),  
    be the maximal \( (2,2) \)-tight subgraph of \( H_2 \) that contains 
    \( v_2 \), respectively \( v_4 \). By Lemma \ref{lem_int_unions} we 
    see that \( X_2 \subset Z_2 \) and \( X_4 \subset Z_4 \). Furthermore 
    we see that since \( Z_2 \) and \( Z_4 \) are both disjoint from \( \alpha_1 \),
    they are either annular or inessential. By Lemma \ref{lem_non_cell},  \( Z_2 \) has 
    graph \( K_1 \) (inessential case) or \( \cyc_2 \) (annular case). 
    Similar comments apply to \( Z_4 \). Now the argument in the paragraph 
    above shows that every quadrilateral face of \( H_2 \) has a boundary vertex in 
    \( Z_2 \) and a diagonally opposite vertex in \( Z_4 \). 
    It follows easily that there is at most one such quadrilateral face in \( H_2 \).

    Now we can argue as in case (a) but with the 
    proviso that there is exactly one quadrilateral face, \( R \),  of \( H_2 \) that is 
    also a face of \( G \). 
    We observe that there is a cycle of length \( 3 \) in \( H_2 \) (formed by two edges of \( \partial R \) and the inessential blocker for \( R \)) and so also 
    in \( K \). It is not hard to see that it follows that \( K \) must have at least two
    faces of odd degree: at least one on either `side' of the cycle of length 3.
    We find the following subcases.
    \begin{enumerate}
        \item \( |J^+| = |J^-| = 2 \). From Equation (\ref{eq_ess_blocker}) we 
            have \( f_5 +2f_6 = 2 \). 
            Since \( K \) has some face of odd degree
            we can rule out the possibility \( f_5 = 0, f_6 = 1 \). Therefore 
            \( f_5 =2 \) and \( f_6 = 0 \). There is only one possibility for \( H_2 \): Figure \ref{fig_ess_quad} (a).
        \item \( |J^+| =2 \) and \( |J^-| =3 \). Then, as in case (a) we have \( f_5=1 \)
            and there is one possibility: Figure \ref{fig_ess_quad} (b).
        \item \( |J^+| = |J^-| = 3 \). In this case, Equation \ref{eq_ess_blocker}
            implies that \( R\), \(J^+ \) and \( J^- \) are the only faces of \( K \): Figure \ref{fig_ess_quad} (c).
        \item \( |J^+| = 4 \) and \( |J^-|=2 \). Since \( K \) must have at least 
            two faces of odd degree, Theorem \ref{thm_euler_sparsity} would 
            imply that there is a triangle or digon in \( K \) that is also
            a face of \( G \), contradicting its irreducibility. Thus this 
            subcase cannot arise.
    \end{enumerate}
\end{proof}

\newcommand{\blockersize}{1.0in}
    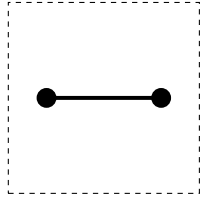
\begin{figure}
        \centering
\begin{tabular}{ccc}
\begin{tikzpicture}[x=\blockersize,y=\blockersize]
\draw[dash pattern=on 2pt off 2pt] (0,0) rectangle (1,1);
\draw [fill=black,line width=1.5pt] (0.2,0.5) circle (3pt);
\draw [fill=black,line width=1.5pt] (0.8,0.5) circle (3pt);
\draw [line width=1.5pt] (0.2,0.5)--(0.8,0.5);

\draw (0.5,-0.15) node {$$};

\end{tikzpicture}
&
\end{tabular}

        \caption{The unique inessential blocker}
        \label{fig_iness_blocker}
    \end{figure}

    \begin{figure}
        \centering
\begin{tabular}{ccc}
\begin{tikzpicture}[x=\blockersize,y=\blockersize]
\draw[dash pattern=on 2pt off 2pt] (0,0) rectangle (1,1);
\draw [fill=black,line width=1.5pt] (0.2,0.5) circle (3pt);
\draw [fill=black,line width=1.5pt] (0.8,0.5) circle (3pt);
\draw [fill=black,line width=1.5pt] (0.2,0.75) circle (3pt);
\draw [fill=black,line width=1.5pt] (0.8,0.75) circle (3pt);
\draw [line width=1.5pt] (0.2,0.5)--(0.8,0.5);
\draw [line width=1.5pt] (0.2,0.5)--(0.2,0.75);
\draw [line width=1.5pt] (0.2,0.5)--(0.2,0.0);
\draw [line width=1.5pt] (0.2,1.0)--(0.2,0.75);
\draw [line width=1.5pt] (0.8,0.5)--(0.8,0.75);
\draw [line width=1.5pt] (0.8,0.5)--(0.8,0.0);
\draw [line width=1.5pt] (0.8,1.0)--(0.8,0.75);

\draw (0.5,-0.15) node {(a)};

\end{tikzpicture}
&
\begin{tikzpicture}[x=\blockersize,y=\blockersize]
\draw[dash pattern=on 2pt off 2pt] (0,0) rectangle (1,1);
\draw [fill=black,line width=1.5pt] (0.2,0.5) circle (3pt);
\draw [fill=black,line width=1.5pt] (0.8,0.5) circle (3pt);
\draw [fill=black,line width=1.5pt] (0.2,0.75) circle (3pt);
\draw [fill=black,line width=1.5pt] (0.8,0.75) circle (3pt);
\draw [fill=black,line width=1.5pt] (0.5,0.75) circle (3pt);
\draw [line width=1.5pt] (0.2,0.5)--(0.8,0.5);
\draw [line width=1.5pt] (0.2,0.5)--(0.2,0.75);
\draw [line width=1.5pt] (0.2,0.5)--(0.2,0.0);
\draw [line width=1.5pt] (0.2,1.0)--(0.2,0.75);
\draw [line width=1.5pt] (0.8,0.5)--(0.8,0.75);
\draw [line width=1.5pt] (0.8,0.5)--(0.8,0.0);
\draw [line width=1.5pt] (0.8,1.0)--(0.8,0.75);
\draw [line width=1.5pt] (0.2,0.5)--(0.5,0.75);
\draw [line width=1.5pt] (0.5,0.75)--(0.6,1.0);
\draw [line width=1.5pt] (0.6,0.0)--(0.8,0.5);

\draw (0.5,-0.15) node {(b)};

\end{tikzpicture}
&
\begin{tikzpicture}[x=\blockersize,y=\blockersize]
\draw[dash pattern=on 2pt off 2pt] (0,0) rectangle (1,1);
\draw [fill=black,line width=1.5pt] (0.2,0.5) circle (3pt);
\draw [fill=black,line width=1.5pt] (0.8,0.5) circle (3pt);
\draw [fill=black,line width=1.5pt] (0.2,0.75) circle (3pt);
\draw [fill=black,line width=1.5pt] (0.8,0.75) circle (3pt);
\draw [fill=black,line width=1.5pt] (0.5,0.75) circle (3pt);
\draw [line width=1.5pt] (0.2,0.5)--(0.8,0.5);
\draw [line width=1.5pt] (0.2,0.5)--(0.2,0.75);
\draw [line width=1.5pt] (0.2,0.5)--(0.2,0.0);
\draw [line width=1.5pt] (0.2,1.0)--(0.2,0.75);
\draw [line width=1.5pt] (0.8,0.5)--(0.8,0.75);
\draw [line width=1.5pt] (0.8,0.5)--(0.8,0.0);
\draw [line width=1.5pt] (0.8,1.0)--(0.8,0.75);
\draw [line width=1.5pt] (0.2,0.75)--(0.5,0.75);
\draw [line width=1.5pt] (0.5,0.75)--(0.8,0.75);

\draw (0.5,-0.15) node {(c)};

\end{tikzpicture}
\\
\begin{tikzpicture}[x=\blockersize,y=\blockersize]
\draw[dash pattern=on 2pt off 2pt] (0,0) rectangle (1,1);
\draw [fill=black,line width=1.5pt] (0.2,0.5) circle (3pt);
\draw [fill=black,line width=1.5pt] (0.8,0.5) circle (3pt);
\draw [fill=black,line width=1.5pt] (0.2,0.75) circle (3pt);
\draw [fill=black,line width=1.5pt] (0.5,0.75) circle (3pt);
\draw [line width=1.5pt] (0.2,0.5)--(0.8,0.5);
\draw [line width=1.5pt] (0.2,0.5)--(0.2,0.75);
\draw [line width=1.5pt] (0.2,0.5)--(0.2,0.0);
\draw [line width=1.5pt] (0.2,1.0)--(0.2,0.75);
\draw [line width=1.5pt] (0.8,0.5)--(0.5,0.75);
\draw [line width=1.5pt] (0.2,0.5)--(0.4,0.0);
\draw [line width=1.5pt] (0.4,1.0)--(0.5,0.75);

\draw (0.5,-0.15) node {(d)};

\end{tikzpicture}
&
\begin{tikzpicture}[x=\blockersize,y=\blockersize]
\draw[dash pattern=on 2pt off 2pt] (0,0) rectangle (1,1);
\draw [fill=black,line width=1.5pt] (0.2,0.5) circle (3pt);
\draw [fill=black,line width=1.5pt] (0.8,0.5) circle (3pt);
\draw [fill=black,line width=1.5pt] (0.5,0.75) circle (3pt);
\draw [line width=1.5pt] (0.2,0.5)--(0.8,0.5);
\draw [line width=1.5pt] (0.8,0.5)--(0.5,0.75);
\draw [line width=1.5pt] (0.2,0.5)--(0.4,0.0);
\draw [line width=1.5pt] (0.4,1.0)--(0.5,0.75);

\draw (0.5,-0.15) node {(e)};

\end{tikzpicture}
&
\begin{tikzpicture}[x=\blockersize,y=\blockersize]
\draw[dash pattern=on 2pt off 2pt] (0,0) rectangle (1,1);
\draw [fill=black,line width=1.5pt] (0.2,0.5) circle (3pt);
\draw [fill=black,line width=1.5pt] (0.2,0.75) circle (3pt);
\draw [fill=black,line width=1.5pt] (0.8,0.5) circle (3pt);
\draw [line width=1.5pt] (0.2,0.5)--(0.2,0.75);
\draw [line width=1.5pt] (0.2,0.75)--(0.2,1.0);
\draw [line width=1.5pt] (0.2,0.0)--(0.2,0.5);
\draw [line width=1.5pt] (0.2,0.5)--(0.8,0.5);

\draw (0.5,-0.15) node {(f)};

\end{tikzpicture}
\\
\end{tabular}
        
        \caption{Essential blockers with no quadrilateral face}
        \label{fig_ess_no_quad}
    \end{figure}

    \begin{figure}
        \centering

\begin{tabular}{ccc}
\begin{tikzpicture}[x=\blockersize,y=\blockersize]
\draw[dash pattern=on 2pt off 2pt] (0,0) rectangle (1,1);
\draw [fill=black,line width=1.5pt] (0.25,0.5) circle (3pt);
\draw [fill=black,line width=1.5pt] (0.5,0.75) circle (3pt);
\draw [fill=black,line width=1.5pt] (0.75,0.5) circle (3pt);
\draw [fill=black,line width=1.5pt] (0.5,0.25) circle (3pt);
\draw [fill=black,line width=1.5pt] (0.25,0.75) circle (3pt);
\draw [fill=black,line width=1.5pt] (0.75,0.75) circle (3pt);
\draw [line width=1.5pt] (0.25,0.5)--(0.5,0.75);
\draw [line width=1.5pt] (0.5,0.75)--(0.75,0.5);
\draw [line width=1.5pt] (0.75,0.5)--(0.5,0.25);
\draw [line width=1.5pt] (0.5,0.25)--(0.25,0.5);
\draw [line width=1.5pt] (0.5,0.75)--(0.5,1.0);
\draw [line width=1.5pt] (0.5,0.0)--(0.5,0.25);
\draw [line width=1.5pt] (0.25,0.5)--(0.25,0.75);
\draw [line width=1.5pt] (0.25,0.5)--(0.25,0.0);
\draw [line width=1.5pt] (0.25,1.0)--(0.25,0.75);
\draw [line width=1.5pt] (0.75,0.5)--(0.75,0.75);
\draw [line width=1.5pt] (0.75,0.5)--(0.75,0.0);
\draw [line width=1.5pt] (0.75,1.0)--(0.75,0.75);

\draw (0.5,-0.15) node {(a)};

\end{tikzpicture}
&
\begin{tikzpicture}[x=\blockersize,y=\blockersize]
\draw[dash pattern=on 2pt off 2pt] (0,0) rectangle (1,1);
\draw [fill=black,line width=1.5pt] (0.25,0.5) circle (3pt);
\draw [fill=black,line width=1.5pt] (0.5,0.75) circle (3pt);
\draw [fill=black,line width=1.5pt] (0.75,0.5) circle (3pt);
\draw [fill=black,line width=1.5pt] (0.5,0.25) circle (3pt);
\draw [fill=black,line width=1.5pt] (0.25,0.75) circle (3pt);
\draw [line width=1.5pt] (0.25,0.5)--(0.5,0.75);
\draw [line width=1.5pt] (0.5,0.75)--(0.75,0.5);
\draw [line width=1.5pt] (0.75,0.5)--(0.5,0.25);
\draw [line width=1.5pt] (0.5,0.25)--(0.25,0.5);
\draw [line width=1.5pt] (0.5,0.75)--(0.5,1.0);
\draw [line width=1.5pt] (0.5,0.0)--(0.5,0.25);
\draw [line width=1.5pt] (0.25,0.5)--(0.25,0.75);
\draw [line width=1.5pt] (0.25,0.5)--(0.25,0.0);
\draw [line width=1.5pt] (0.25,1.0)--(0.25,0.75);

\draw (0.5,-0.15) node {(b)};

\end{tikzpicture}
&
\begin{tikzpicture}[x=\blockersize,y=\blockersize]
\draw[dash pattern=on 2pt off 2pt] (0,0) rectangle (1,1);
\draw [fill=black,line width=1.5pt] (0.25,0.5) circle (3pt);
\draw [fill=black,line width=1.5pt] (0.5,0.75) circle (3pt);
\draw [fill=black,line width=1.5pt] (0.75,0.5) circle (3pt);
\draw [fill=black,line width=1.5pt] (0.5,0.25) circle (3pt);
\draw [line width=1.5pt] (0.25,0.5)--(0.5,0.75);
\draw [line width=1.5pt] (0.5,0.75)--(0.75,0.5);
\draw [line width=1.5pt] (0.75,0.5)--(0.5,0.25);
\draw [line width=1.5pt] (0.5,0.25)--(0.25,0.5);
\draw [line width=1.5pt] (0.5,0.75)--(0.5,1.0);
\draw [line width=1.5pt] (0.5,0.0)--(0.5,0.25);

\draw (0.5,-0.15) node {(c)};

\end{tikzpicture}
\\
\end{tabular}

    \caption{Essential blockers with a quadrilateral face}
    \label{fig_ess_quad}
\end{figure}
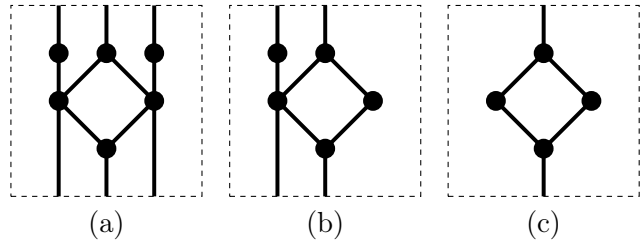

It remains to rule out the possibility of a quadrilateral 
face that is neither \( Q \) nor a face of \( H_2  \). In fact 
we can prove something a little more general than that.

\begin{lem}
    \label{lem_noquadincell}
    Let \( K \) be a \( (2,2) \)-tight subgraph of \( G \) and
    suppose that \( F \) is a cellular face of \( K \). There is 
    no quadrilateral face of \( G \) properly contained within
    \( F \).
\end{lem}

\begin{proof}
    Suppose that \( R \), with vertices 
    \( w_1,w_2,w_3,w_4 \), is a quadrilateral
    face of \( G \) properly contained within \( F \) and let 
    \( B_1 \) and \( B_2 \) be blockers for contractions of \( R \) in \( G \).
    If \( B_1 \subset F \), then since \( F \) is cellular, \( B_1 \)
    would separate \( w_2 \) from \( w_4 \) which contradicts \( B_1 
    \cap B_2 = \emptyset\). Therefore \( B_1 \ \) is not contained in 
    \( F \) or equivalently, since \( B_1 \) is connected, \( B_1
    \cap K \neq \emptyset\). Similarly \( B_2 \cap K \neq \emptyset \).

    Now, let \( M = \partial R \cup B_1 \cup B_2 \) and observe that 
    \( M \) is \( (2,2) \)-tight and therefore, by Lemma \ref{lem_int_unions},
    \( M \cap K \) is also \( (2,2) \)-tight. 
    Now it is clear that \( M \cap K = (B_1 \cap K) \cup (B_2 \cap K) \cup
    E(\partial R \cap K)\). Therefore 
    \begin{equation}
        2 = \gamma(M \cap K) = \gamma(B_1 \cap K) + \gamma(B_2 \cap K) - |E(\partial R \cap K)|
        \label{eq_noquadsincell}
    \end{equation}
    
    Now, we observe that \( |E(\partial R \cap K)| \in \{0,1,2,4\} \) since 
    \( K  \) is an induced subgraph of \( G \). 
    If \( |E(\partial R \cap K)| =4\) then clearly \( R \) must be a 
    face of \( K \) which contradicts our assumption that \( R \) is properly 
    contained within \( F \).
    On the other hand if  \( |E(\partial R \cap K)|
    \leq 1\), then (\ref{eq_noquadsincell}) yields 
    \(  \gamma(B_1 \cap K) + \gamma(B_2 \cap K)\leq 3  \) which contradicts
    the fact that both \( B_1 \cap K \) and \( B_2 \cap K \) are nonempty.
    Finally if \( |E(\partial R \cap K)| = 2 \) then it is clear that 
    \( K \) contains exactly three of the vertices \( w_1,w_2,w_3,w_4 \). 
    However in this case 
    (\ref{eq_noquadsincell}) implies that \( \gamma(B_1 \cap K)
    = \gamma(B_2 \cap K ) = 2\). It follows that \( K \) contains at most 
    one of the vertices \( w_1,w_3 \), otherwise \( (B_1 \cap K) \cup \{w_2\} \)
    would span a type 
    1 blocker for \( G_{R,w_1,w_3} \), contradicting Lemma \ref{lem_blocked_quad}.
    Similarly \( K \) contains at most one of the vertices
    \( w_2,w_4 \). Thus \( K \) contains at most two of the vertices
    \( w_1,w_2,w_3,w_4 \) yielding the required contradiction.
\end{proof}

Finally we have our main theorem about torus graphs.

\begin{thm}
    \label{thm_main_torus}
    Suppose that \( G \) is an irreducible \( (2,2) \)-tight 
    \( \mathbb T \)-graph. Then \( G \) has at most two 
    quadrilateral faces. 
\end{thm}

\begin{proof}
    Suppose, as above,  that \( Q \) is a quadrilateral face of 
    \( G \), with maximal blockers \( H_1 \) and \( H_2 \). Also 
    assume that \( H_1 \) is inessential. We have seen that 
    there is at most one other quadrilateral face of \( G \) 
    contained among faces of \( H_2 \). Now let \( K = \partial Q
    \cup H_1 \cup H_2\). Clearly \( K \) is a \( (2,2) \)-tight 
    subgraph of \( G \). 
    Now we consider the faces of \( K \) that are not also faces of 
    \( G \). If \( H_2 \) is also inessential then there is 
    at most one such face and this face is cellular of degree \(8\).
    If \( H_2 \) is essential then there at most two such faces 
    and each such face is cellular and has degree at least \(5\). 
    So by Lemma 
    \ref{lem_noquadincell}, there is no quadrilateral face of 
    \( G \) that is not also a face of \( K \).
\end{proof}

\begin{cor}
    \label{cor_main_finite}
    There are finitely many distinct isomorphism classes of irreducible \( (2,2) \)-tight 
    torus graphs. In particular any such irreducible torus graph has at most eight vertices.
\end{cor}

\begin{proof}
    We may as well assume that \( G \) is cellular, since in the 
    non cellular case we know that \( G \) has at most two vertices.
    Since \( \gamma(G) = 2 \) we have \( |V| = 1+\frac14\sum if_i \), so we 
    must maximise \( \sum if_i \).
    Since \( G \) is irreducible, 
    \( f_i = 0  \) for \( i = 0,1,2,3 \) and \( f_4 \leq 2 \).
    From Theorem \ref{thm_euler_sparsity} we have \( 
    f_5+2f_6+3f_7+4f_8 = 4\) and \( f_i = 0\) for \( i \geq 9 \). 
    Clearly the maximum value for \( \sum  if_i \) is attained by 
    having \( f_4 = 2 \), \( f_5 = 4 \) and \( f_i = 0 \) for 
    \( i \neq 4,5 \).
    In that case \( |V|=8 \). Now there are finitely 
    many isomorphism classes of \( (2,2) \)-tight graphs with at most 
    eight vertices. Moreover, for each such graph, there are finitely 
    many isomorphism classes of torus graphs with that underlying
    graph.
\end{proof}

\subsection{Identifying irreducibles}
\label{sec_search}

Given Corollary \ref{cor_main_finite},
a naive algorithm to find all the irreducibles mentioned therein
would be 
\begin{enumerate}
    \item Find all \( (2,2) \)-tight graphs with at most 8 vertices.
    \item For each such graph, find all isomorphism classes of torus embeddings.
    \item Eliminate all embeddings that are not irreducible.
\end{enumerate}
It is impractical to carry out 
this procedure without the assistance of a computer as step (1) will already yield many thousands
of distinct graphs, each of which could have many different torus embeddings.

However, since we have 
a lot of structural information about irreducibles, we can narrow the search space 
significantly. For example, it is clear 
from the proof of Corollary \ref{cor_main_finite} that any 
irreducible with 8 vertices must have 2 quadrilateral faces, 4 faces of degree
5 and no other faces. Moreover, we know that 
each quadrilateral face has one essential blocker and one other blocker which must 
be one of the 10 graphs described in Section \ref{sec_irred_torus}. 
It is not too difficult to deduce that any 
8 vertex irreducible must be isomorphic to one of the examples shown in Figure \ref{fig_torus8}.

\begin{figure}

\begin{tabular}{ccccc}
\begin{tikzpicture}[x=2.0cm,y=2.0cm]
\draw[dash pattern=on 2pt off 2pt] (0,0) rectangle (1,1);
\draw [fill=black,line width=1.5pt] (0.1,0.5) circle (3pt);
\draw [fill=black,line width=1.5pt] (0.25,0.7) circle (3pt);
\draw [fill=black,line width=1.5pt] (0.4,0.5) circle (3pt);
\draw [fill=black,line width=1.5pt] (0.25,0.3) circle (3pt);
\draw [fill=black,line width=1.5pt] (0.75,0.3) circle (3pt);
\draw [fill=black,line width=1.5pt] (0.9,0.5) circle (3pt);
\draw [fill=black,line width=1.5pt] (0.75,0.7) circle (3pt);
\draw [fill=black,line width=1.5pt] (0.6,0.5) circle (3pt);
\draw [line width=1.5pt] (0.1,0.5)--(0.25,0.7);
\draw [line width=1.5pt] (0.25,0.7)--(0.4,0.5);
\draw [line width=1.5pt] (0.25,0.3)--(0.1,0.5);
\draw [line width=1.5pt] (0.25,0.7)--(0.25,1.0);
\draw [line width=1.5pt] (0.25,0.0)--(0.25,0.3);
\draw [line width=1.5pt] (0.4,0.5)--(0.25,0.3);
\draw [line width=1.5pt] (0.75,0.3)--(0.9,0.5);
\draw [line width=1.5pt] (0.75,0.7)--(0.75,1.0);
\draw [line width=1.5pt] (0.75,0.0)--(0.75,0.3);
\draw [line width=1.5pt] (0.75,0.7)--(0.9,0.5);
\draw [line width=1.5pt] (0.9,0.5)--(1.0,0.5);
\draw [line width=1.5pt] (0.0,0.5)--(0.1,0.5);
\draw [line width=1.5pt] (0.6,0.5)--(0.4,0.5);
\draw [line width=1.5pt] (0.6,0.5)--(0.75,0.3);
\draw [line width=1.5pt] (0.75,0.7)--(0.6,0.5);
\draw [line width=1.5pt] (0.4,0.5)--(0.5,1.0);
\draw [line width=1.5pt] (0.5,0.0)--(0.6,0.5);
\draw [line width=1.5pt] (0.9,0.5)--(1.0,1.0);
\draw [line width=1.5pt] (1.0,0.0)--(1.0,0.0);
\draw [line width=1.5pt] (0.0,0.0)--(0.1,0.5);

\end{tikzpicture}
&
\begin{tikzpicture}[x=2.0cm,y=2.0cm]
\draw[dash pattern=on 2pt off 2pt] (0,0) rectangle (1,1);
\draw [fill=black,line width=1.5pt] (0,0.5) circle (3pt);
\draw [fill=black,line width=1.5pt] (1,0.5) circle (3pt);
\draw [fill=black,line width=1.5pt] (0.5,0.5) circle (3pt);
\draw [fill=black,line width=1.5pt] (0.25,0.25) circle (3pt);
\draw [fill=black,line width=1.5pt] (0.25,0.75) circle (3pt);
\draw [fill=black,line width=1.5pt] (0.75,0.25) circle (3pt);
\draw [fill=black,line width=1.5pt] (0.75,0.75) circle (3pt);
\draw [fill=black,line width=1.5pt] (0.88,0.88) circle (3pt);
\draw [fill=black,line width=1.5pt] (0.5,0.75) circle (3pt);
\draw [line width=1.5pt] (0,0.5)--(0.25,0.25);
\draw [line width=1.5pt] (0,0.5)--(0.25,0.75);
\draw [line width=1.5pt] (0.5,0.5)--(0.25,0.25);
\draw [line width=1.5pt] (0.25,0.25)--(0.25,0.0);
\draw [line width=1.5pt] (0.25,1.0)--(0.25,0.75);
\draw [line width=1.5pt] (0.5,0.5)--(0.25,0.75);
\draw [line width=1.5pt] (0.5,0.5)--(0.75,0.25);
\draw [line width=1.5pt] (0.5,0.5)--(0.75,0.75);
\draw [line width=1.5pt] (0,0.5)--(0.0,0.5);
\draw [line width=1.5pt] (1.0,0.5)--(0.75,0.75);
\draw [line width=1.5pt] (0,0.5)--(0.0,0.5);
\draw [line width=1.5pt] (1.0,0.5)--(0.75,0.25);
\draw [line width=1.5pt] (0.75,0.25)--(0.75,0.0);
\draw [line width=1.5pt] (0.75,1.0)--(0.75,0.75);
\draw [line width=1.5pt] (0.88,0.88)--(0.75,0.75);
\draw [line width=1.5pt] (0.88,0.88)--(1.00,1.0);
\draw [line width=1.5pt] (1.00,0.0)--(1.0,0.00);
\draw [line width=1.5pt] (0.0,0.00)--(0.25,0.25);
\draw [line width=1.5pt] (0.5,0.75)--(0.25,0.75);
\draw [line width=1.5pt] (0.5,0.75)--(0.625,1.0);
\draw [line width=1.5pt] (0.625,0.0)--(0.75,0.25);

\end{tikzpicture}
&
\begin{tikzpicture}[x=2.0cm,y=2.0cm]
\draw[dash pattern=on 2pt off 2pt] (0,0) rectangle (1,1);
\draw [fill=black,line width=1.5pt] (0,0.5) circle (3pt);
\draw [fill=black,line width=1.5pt] (1,0.5) circle (3pt);
\draw [fill=black,line width=1.5pt] (0.5,0.5) circle (3pt);
\draw [fill=black,line width=1.5pt] (0.25,0.25) circle (3pt);
\draw [fill=black,line width=1.5pt] (0.25,0.75) circle (3pt);
\draw [fill=black,line width=1.5pt] (0.75,0.25) circle (3pt);
\draw [fill=black,line width=1.5pt] (0.75,0.75) circle (3pt);
\draw [fill=black,line width=1.5pt] (0.5,0.75) circle (3pt);
\draw [fill=black,line width=1.5pt] (0.88,0.88) circle (3pt);
\draw [line width=1.5pt] (0,0.5)--(0.25,0.25);
\draw [line width=1.5pt] (0,0.5)--(0.25,0.75);
\draw [line width=1.5pt] (0.5,0.5)--(0.25,0.25);
\draw [line width=1.5pt] (0.25,0.25)--(0.25,0.0);
\draw [line width=1.5pt] (0.25,1.0)--(0.25,0.75);
\draw [line width=1.5pt] (0.5,0.5)--(0.25,0.75);
\draw [line width=1.5pt] (0.5,0.5)--(0.75,0.25);
\draw [line width=1.5pt] (0.5,0.5)--(0.75,0.75);
\draw [line width=1.5pt] (0,0.5)--(0.0,0.5);
\draw [line width=1.5pt] (1.0,0.5)--(0.75,0.75);
\draw [line width=1.5pt] (0,0.5)--(0.0,0.5);
\draw [line width=1.5pt] (1.0,0.5)--(0.75,0.25);
\draw [line width=1.5pt] (0.75,0.25)--(0.75,0.0);
\draw [line width=1.5pt] (0.75,1.0)--(0.75,0.75);
\draw [line width=1.5pt] (0.5,0.75)--(0.5,0.5);
\draw [line width=1.5pt] (0.5,0.75)--(0.5,0.0);
\draw [line width=1.5pt] (0.5,1.0)--(0.5,0.5);
\draw [line width=1.5pt] (0.88,0.88)--(0.75,0.75);
\draw [line width=1.5pt] (0.88,0.88)--(1.00,1.0);
\draw [line width=1.5pt] (1.00,0.0)--(1.0,0.00);
\draw [line width=1.5pt] (0.0,0.00)--(0.25,0.25);

\end{tikzpicture}
&
\begin{tikzpicture}[x=2.0cm,y=2.0cm]
\draw[dash pattern=on 2pt off 2pt] (0,0) rectangle (1,1);
\draw [fill=black,line width=1.5pt] (0.1,0.5) circle (3pt);
\draw [fill=black,line width=1.5pt] (0.6,0.5) circle (3pt);
\draw [fill=black,line width=1.5pt] (0.35,0.25) circle (3pt);
\draw [fill=black,line width=1.5pt] (0.35,0.75) circle (3pt);
\draw [fill=black,line width=1.5pt] (0.85,0.25) circle (3pt);
\draw [fill=black,line width=1.5pt] (0.85,0.75) circle (3pt);
\draw [fill=black,line width=1.5pt] (0.6,0.75) circle (3pt);
\draw [fill=black,line width=1.5pt] (0.1,0.75) circle (3pt);
\draw [line width=1.5pt] (0.1,0.5)--(0.35,0.25);
\draw [line width=1.5pt] (0.1,0.5)--(0.35,0.75);
\draw [line width=1.5pt] (0.6,0.5)--(0.35,0.25);
\draw [line width=1.5pt] (0.35,0.25)--(0.35,0.0);
\draw [line width=1.5pt] (0.35,1.0)--(0.35,0.75);
\draw [line width=1.5pt] (0.6,0.5)--(0.35,0.75);
\draw [line width=1.5pt] (0.6,0.5)--(0.85,0.25);
\draw [line width=1.5pt] (0.6,0.5)--(0.85,0.75);
\draw [line width=1.5pt] (0.1,0.5)--(0.0,0.5999999999999999920000000000);
\draw [line width=1.5pt] (1.0,0.5999999999999999920000000000)--(0.84999999999999998,0.75);
\draw [line width=1.5pt] (0.1,0.5)--(0.0,0.4000000000000000080000000000);
\draw [line width=1.5pt] (1.0,0.4000000000000000080000000000)--(0.84999999999999998,0.25);
\draw [line width=1.5pt] (0.85,0.25)--(0.85,0.0);
\draw [line width=1.5pt] (0.85,1.0)--(0.85,0.75);
\draw [line width=1.5pt] (0.6,0.75)--(0.6,0.5);
\draw [line width=1.5pt] (0.6,0.75)--(0.6,0.0);
\draw [line width=1.5pt] (0.6,1.0)--(0.6,0.5);
\draw [line width=1.5pt] (0.1,0.75)--(0.1,0.5);
\draw [line width=1.5pt] (0.1,0.75)--(0.1,1.0);
\draw [line width=1.5pt] (0.1,0.0)--(0.1,0.5);

\end{tikzpicture}
&
\begin{tikzpicture}[x=2.0cm,y=2.0cm]
\draw[dash pattern=on 2pt off 2pt] (0,0) rectangle (1,1);
\draw [fill=black,line width=1.5pt] (0.1,0.5) circle (3pt);
\draw [fill=black,line width=1.5pt] (0.3,0.7) circle (3pt);
\draw [fill=black,line width=1.5pt] (0.5,0.5) circle (3pt);
\draw [fill=black,line width=1.5pt] (0.3,0.3) circle (3pt);
\draw [fill=black,line width=1.5pt] (0.7,0.3) circle (3pt);
\draw [fill=black,line width=1.5pt] (0.9,0.5) circle (3pt);
\draw [fill=black,line width=1.5pt] (0.7,0.7) circle (3pt);
\draw [fill=black,line width=1.5pt] (0.5,0.3) circle (3pt);
\draw [line width=1.5pt] (0.1,0.5)--(0.3,0.7);
\draw [line width=1.5pt] (0.3,0.7)--(0.5,0.5);
\draw [line width=1.5pt] (0.5,0.5)--(0.7,0.3);
\draw [line width=1.5pt] (0.3,0.3)--(0.1,0.5);
\draw [line width=1.5pt] (0.3,0.7)--(0.3,1.0);
\draw [line width=1.5pt] (0.3,0.0)--(0.3,0.3);
\draw [line width=1.5pt] (0.5,0.5)--(0.3,0.3);
\draw [line width=1.5pt] (0.7,0.3)--(0.9,0.5);
\draw [line width=1.5pt] (0.7,0.7)--(0.7,1.0);
\draw [line width=1.5pt] (0.7,0.0)--(0.7,0.3);
\draw [line width=1.5pt] (0.7,0.7)--(0.9,0.5);
\draw [line width=1.5pt] (0.7,0.7)--(0.5,0.5);
\draw [line width=1.5pt] (0.9,0.5)--(1.0,0.5);
\draw [line width=1.5pt] (0.0,0.5)--(0.1,0.5);
\draw [line width=1.5pt] (0.5,0.3)--(0.5,0.5);
\draw [line width=1.5pt] (0.5,0.3)--(0.5,0.0);
\draw [line width=1.5pt] (0.5,1.0)--(0.5,0.5);
\draw [line width=1.5pt] (0.1,0.5)--(1E-17,0.0);
\draw [line width=1.5pt] (1E-17,1.0)--(0.0,0.9999999999999999500000000000);
\draw [line width=1.5pt] (1.0,0.9999999999999999500000000000)--(0.90000000000000002,0.5);

\end{tikzpicture}
\\
\begin{tikzpicture}[x=2.0cm,y=2.0cm]
\draw[dash pattern=on 2pt off 2pt] (0,0) rectangle (1,1);
\draw [fill=black,line width=1.5pt] (0.0,0.5) circle (3pt);
\draw [fill=black,line width=1.5pt] (1.0,0.5) circle (3pt);
\draw [fill=black,line width=1.5pt] (0.2,0.7) circle (3pt);
\draw [fill=black,line width=1.5pt] (0.4,0.5) circle (3pt);
\draw [fill=black,line width=1.5pt] (0.2,0.3) circle (3pt);
\draw [fill=black,line width=1.5pt] (0.8,0.3) circle (3pt);
\draw [fill=black,line width=1.5pt] (0.6,0.5) circle (3pt);
\draw [fill=black,line width=1.5pt] (0.8,0.7) circle (3pt);
\draw [fill=black,line width=1.5pt] (0.9,0.85) circle (3pt);
\draw [line width=1.5pt] (0.0,0.5)--(0.2,0.7);
\draw [line width=1.5pt] (0.2,0.7)--(0.4,0.5);
\draw [line width=1.5pt] (0.4,0.5)--(0.2,0.3);
\draw [line width=1.5pt] (0.2,0.3)--(0.0,0.5);
\draw [line width=1.5pt] (0.2,0.7)--(0.2,1.0);
\draw [line width=1.5pt] (0.2,0.0)--(0.2,0.3);
\draw [line width=1.5pt] (0.0,0.5)--(0.0,0.5);
\draw [line width=1.5pt] (1.0,0.5)--(0.80000000000000004,0.3);
\draw [line width=1.5pt] (0.4,0.5)--(0.6,0.5);
\draw [line width=1.5pt] (0.4,0.5)--(0.5,1.0);
\draw [line width=1.5pt] (0.5,0.0)--(0.6,0.5);
\draw [line width=1.5pt] (0.8,0.3)--(0.6,0.5);
\draw [line width=1.5pt] (0.8,0.7)--(0.8,1.0);
\draw [line width=1.5pt] (0.8,0.0)--(0.8,0.3);
\draw [line width=1.5pt] (0.8,0.7)--(0.6,0.5);
\draw [line width=1.5pt] (0.8,0.7)--(1.0,0.5);
\draw [line width=1.5pt] (0.9,0.85)--(0.8,0.7);
\draw [line width=1.5pt] (0.9,0.85)--(1.0,1.0);
\draw [line width=1.5pt] (1.0,0.0)--(1.0,0.0);
\draw [line width=1.5pt] (0.0,0.0)--(0.2,0.3);

\end{tikzpicture}
&
\end{tabular}

\begin{caption}{Irreducible torus graphs with eight vertices}
    \label{fig_torus8}
\end{caption}
\end{figure}

Similarly for torus graphs with at most 4 vertices there are relatively few possibilities 
for the underlying graph: \( 13 \) in total. Now, using Lemmas \ref{lem_bicycles_are_essential},
\ref{lem_4cycle_is_face} and \ref{lem_annular_quad} we can easily deduce that 
an irreducible with at most 4 vertices is isomorphic to one of the examples 
shown in Figures \ref{fig_torusleq3} or \ref{fig_torus4}. 
For the cases of 5, 6 and 7 vertices the naive 
this approach yields a relatively manageable problem in computational 
graph theory. We have used the computer algebra system SageMath 
\cite{sagemath} to automate much of the search process in these cases. The interested 
reader can find full details of the search algorithm and its implementation at 
\cite{irreducibles}. As a result of this  computation we have the following.

\begin{thm}
    \label{thm_116}
    There are 116 distinct isomorphism classes of $(2,2)$-tight irreducible 
    torus graphs. \qed
\end{thm}

\begin{figure}

\begin{tabular}{ccccc}
\begin{tikzpicture}[x=2.0cm,y=2.0cm]
\draw[dash pattern=on 2pt off 2pt] (0,0) rectangle (1,1);
\draw [fill=black,line width=1.5pt] (0.5,0.5) circle (3pt);

\end{tikzpicture}
&
\begin{tikzpicture}[x=2.0cm,y=2.0cm]
\draw[dash pattern=on 2pt off 2pt] (0,0) rectangle (1,1);
\draw [fill=black,line width=1.5pt] (0.5,0.25) circle (3pt);
\draw [fill=black,line width=1.5pt] (0.5,0.75) circle (3pt);
\draw [line width=1.5pt] (0.5,0.25)--(0.5,0.75);
\draw [line width=1.5pt] (0.5,0.25)--(0.5,0.0);
\draw [line width=1.5pt] (0.5,1.0)--(0.5,0.75);

\end{tikzpicture}
&
\begin{tikzpicture}[x=2.0cm,y=2.0cm]
\draw[dash pattern=on 2pt off 2pt] (0,0) rectangle (1,1);
\draw [fill=black,line width=1.5pt] (0.333,0.333) circle (3pt);
\draw [fill=black,line width=1.5pt] (0.666,0.333) circle (3pt);
\draw [fill=black,line width=1.5pt] (0.333,0.666) circle (3pt);
\draw [line width=1.5pt] (0.333,0.333)--(0.666,0.333);
\draw [line width=1.5pt] (0.333,0.333)--(0.333,0.666);
\draw [line width=1.5pt] (0.333,0.333)--(0.0,0.333);
\draw [line width=1.5pt] (1.0,0.333)--(0.66600000000000004,0.333);
\draw [line width=1.5pt] (0.333,0.333)--(0.333,0.0);
\draw [line width=1.5pt] (0.333,1.0)--(0.333,0.66600000000000004);

\end{tikzpicture}
&
\begin{tikzpicture}[x=2.0cm,y=2.0cm]
\draw[dash pattern=on 2pt off 2pt] (0,0) rectangle (1,1);
\draw [fill=black,line width=1.5pt] (0.333,0.333) circle (3pt);
\draw [fill=black,line width=1.5pt] (0.333,0.666) circle (3pt);
\draw [fill=black,line width=1.5pt] (0.7,0.5) circle (3pt);
\draw [line width=1.5pt] (0.333,0.333)--(0.333,0.666);
\draw [line width=1.5pt] (0.333,0.333)--(0.7,0.5);
\draw [line width=1.5pt] (0.333,0.333)--(0.333,0.0);
\draw [line width=1.5pt] (0.333,1.0)--(0.333,0.66600000000000004);
\draw [line width=1.5pt] (0.333,0.666)--(0.0,0.5786729857819905268453089553);
\draw [line width=1.5pt] (1.0,0.5786729857819905268453089553)--(0.69999999999999996,0.5);

\end{tikzpicture}
&
\end{tabular}

\begin{caption}{Irreducible torus graphs with at most three vertices}
    \label{fig_torusleq3}
\end{caption}
\end{figure}

\begin{figure}

\begin{tabular}{ccccc}
\begin{tikzpicture}[x=2.0cm,y=2.0cm]
\draw[dash pattern=on 2pt off 2pt] (0,0) rectangle (1,1);
\draw [fill=black,line width=1.5pt] (0.25,0.5) circle (3pt);
\draw [fill=black,line width=1.5pt] (0.5,0.75) circle (3pt);
\draw [fill=black,line width=1.5pt] (0.75,0.5) circle (3pt);
\draw [fill=black,line width=1.5pt] (0.5,0.25) circle (3pt);
\draw [line width=1.5pt] (0.25,0.5)--(0.5,0.75);
\draw [line width=1.5pt] (0.5,0.75)--(0.75,0.5);
\draw [line width=1.5pt] (0.75,0.5)--(0.5,0.25);
\draw [line width=1.5pt] (0.5,0.25)--(0.25,0.5);
\draw [line width=1.5pt] (0.25,0.5)--(0.0,0.5);
\draw [line width=1.5pt] (1.0,0.5)--(0.75,0.5);
\draw [line width=1.5pt] (0.5,0.75)--(0.5,1.0);
\draw [line width=1.5pt] (0.5,0.0)--(0.5,0.25);

\end{tikzpicture}
&
\begin{tikzpicture}[x=2.0cm,y=2.0cm]
\draw[dash pattern=on 2pt off 2pt] (0,0) rectangle (1,1);
\draw [fill=black,line width=1.5pt] (0.25,0.25) circle (3pt);
\draw [fill=black,line width=1.5pt] (0.75,0.25) circle (3pt);
\draw [fill=black,line width=1.5pt] (0.25,0.75) circle (3pt);
\draw [fill=black,line width=1.5pt] (0.75,0.75) circle (3pt);
\draw [line width=1.5pt] (0.75,0.75)--(0.75,0.25);
\draw [line width=1.5pt] (0.25,0.25)--(0.25,0.75);
\draw [line width=1.5pt] (0.25,0.75)--(0.75,0.75);
\draw [line width=1.5pt] (0.25,0.25)--(0.0,0.25);
\draw [line width=1.5pt] (1.0,0.25)--(0.75,0.25);
\draw [line width=1.5pt] (0.75,0.25)--(0.75,0.0);
\draw [line width=1.5pt] (0.75,1.0)--(0.75,0.75);
\draw [line width=1.5pt] (0.25,0.25)--(0.25,0.0);
\draw [line width=1.5pt] (0.25,1.0)--(0.25,0.75);

\end{tikzpicture}
&
\begin{tikzpicture}[x=2.0cm,y=2.0cm]
\draw[dash pattern=on 2pt off 2pt] (0,0) rectangle (1,1);
\draw [fill=black,line width=1.5pt] (0.25,0.25) circle (3pt);
\draw [fill=black,line width=1.5pt] (0.75,0.25) circle (3pt);
\draw [fill=black,line width=1.5pt] (0.25,0.75) circle (3pt);
\draw [fill=black,line width=1.5pt] (0.75,0.75) circle (3pt);
\draw [line width=1.5pt] (0.25,0.25)--(0.25,0.75);
\draw [line width=1.5pt] (0.25,0.75)--(0.75,0.75);
\draw [line width=1.5pt] (0.25,0.25)--(0.75,0.25);
\draw [line width=1.5pt] (0.25,0.75)--(0.25,1.0);
\draw [line width=1.5pt] (0.25,0.0)--(0.25,0.25);
\draw [line width=1.5pt] (0.75,0.75)--(1.0,0.50);
\draw [line width=1.5pt] (0.0,0.50)--(0.25,0.25);
\draw [line width=1.5pt] (0.75,0.25)--(1.00,0.0);
\draw [line width=1.5pt] (1.00,1.0)--(1.0,1.00);
\draw [line width=1.5pt] (0.0,1.00)--(0.25,0.75);

\end{tikzpicture}
&
\begin{tikzpicture}[x=2.0cm,y=2.0cm]
\draw[dash pattern=on 2pt off 2pt] (0,0) rectangle (1,1);
\draw [fill=black,line width=1.5pt] (0.25,0.25) circle (3pt);
\draw [fill=black,line width=1.5pt] (0.75,0.25) circle (3pt);
\draw [fill=black,line width=1.5pt] (0.25,0.75) circle (3pt);
\draw [fill=black,line width=1.5pt] (0.75,0.75) circle (3pt);
\draw [line width=1.5pt] (0.25,0.25)--(0.25,0.75);
\draw [line width=1.5pt] (0.25,0.25)--(0.25,0.0);
\draw [line width=1.5pt] (0.25,1.0)--(0.25,0.75);
\draw [line width=1.5pt] (0.25,0.75)--(0.75,0.75);
\draw [line width=1.5pt] (0.75,0.25)--(0.25,0.75);
\draw [line width=1.5pt] (0.75,0.75)--(1.00,1.0);
\draw [line width=1.5pt] (1.00,0.0)--(1.0,0.00);
\draw [line width=1.5pt] (0.0,0.00)--(0.25,0.25);
\draw [line width=1.5pt] (0.75,0.25)--(1.0,0.25);
\draw [line width=1.5pt] (0.0,0.25)--(0.25,0.25);

\end{tikzpicture}
&
\begin{tikzpicture}[x=2.0cm,y=2.0cm]
\draw[dash pattern=on 2pt off 2pt] (0,0) rectangle (1,1);
\draw [fill=black,line width=1.5pt] (0.25,0.25) circle (3pt);
\draw [fill=black,line width=1.5pt] (0.75,0.25) circle (3pt);
\draw [fill=black,line width=1.5pt] (0.25,0.75) circle (3pt);
\draw [fill=black,line width=1.5pt] (0.75,0.75) circle (3pt);
\draw [line width=1.5pt] (0.25,0.25)--(0.25,0.75);
\draw [line width=1.5pt] (0.25,0.75)--(0.75,0.75);
\draw [line width=1.5pt] (0.75,0.25)--(0.75,0.75);
\draw [line width=1.5pt] (0.75,0.75)--(1.0,0.50);
\draw [line width=1.5pt] (0.0,0.50)--(0.25,0.25);
\draw [line width=1.5pt] (0.25,0.75)--(0.25,1.0);
\draw [line width=1.5pt] (0.25,0.0)--(0.25,0.25);
\draw [line width=1.5pt] (0.25,0.75)--(0.50,1.0);
\draw [line width=1.5pt] (0.50,0.0)--(0.75,0.25);

\end{tikzpicture}
\\
\begin{tikzpicture}[x=2.0cm,y=2.0cm]
\draw[dash pattern=on 2pt off 2pt] (0,0) rectangle (1,1);
\draw [fill=black,line width=1.5pt] (0.75,0.25) circle (3pt);
\draw [fill=black,line width=1.5pt] (0.5,0.25) circle (3pt);
\draw [fill=black,line width=1.5pt] (0.5,0.75) circle (3pt);
\draw [fill=black,line width=1.5pt] (0.75,0.75) circle (3pt);
\draw [line width=1.5pt] (0.75,0.25)--(0.5,0.75);
\draw [line width=1.5pt] (0.75,0.25)--(0.625,0.0);
\draw [line width=1.5pt] (0.625,1.0)--(0.5,0.75);
\draw [line width=1.5pt] (0.5,0.75)--(0.75,0.75);
\draw [line width=1.5pt] (0.5,0.75)--(0.0,0.75);
\draw [line width=1.5pt] (1.0,0.75)--(0.75,0.75);
\draw [line width=1.5pt] (0.5,0.25)--(0.75,0.25);
\draw [line width=1.5pt] (0.5,0.25)--(0.0,0.0);
\draw [line width=1.5pt] (0.0,1.0)--(0.0,1.00);
\draw [line width=1.5pt] (1.0,1.00)--(0.5,0.75);

\end{tikzpicture}
&
\begin{tikzpicture}[x=2.0cm,y=2.0cm]
\draw[dash pattern=on 2pt off 2pt] (0,0) rectangle (1,1);
\draw [fill=black,line width=1.5pt] (0.25,0.25) circle (3pt);
\draw [fill=black,line width=1.5pt] (0.75,0.25) circle (3pt);
\draw [fill=black,line width=1.5pt] (0.25,0.75) circle (3pt);
\draw [fill=black,line width=1.5pt] (0.75,0.75) circle (3pt);
\draw [line width=1.5pt] (0.25,0.25)--(0.25,0.75);
\draw [line width=1.5pt] (0.25,0.75)--(0.75,0.75);
\draw [line width=1.5pt] (0.75,0.25)--(0.75,0.75);
\draw [line width=1.5pt] (0.75,0.75)--(1.0,0.50);
\draw [line width=1.5pt] (0.0,0.50)--(0.25,0.25);
\draw [line width=1.5pt] (0.25,0.75)--(0.25,1.0);
\draw [line width=1.5pt] (0.25,0.0)--(0.25,0.25);
\draw [line width=1.5pt] (0.75,0.75)--(0.75,1.0);
\draw [line width=1.5pt] (0.75,0.0)--(0.75,0.25);

\end{tikzpicture}
&
\begin{tikzpicture}[x=2.0cm,y=2.0cm]
\draw[dash pattern=on 2pt off 2pt] (0,0) rectangle (1,1);
\draw [fill=black,line width=1.5pt] (0.25,0.25) circle (3pt);
\draw [fill=black,line width=1.5pt] (0.75,0.25) circle (3pt);
\draw [fill=black,line width=1.5pt] (0.25,0.75) circle (3pt);
\draw [fill=black,line width=1.5pt] (0.75,0.75) circle (3pt);
\draw [line width=1.5pt] (0.25,0.25)--(0.75,0.25);
\draw [line width=1.5pt] (0.25,0.25)--(0.25,0.75);
\draw [line width=1.5pt] (0.25,0.25)--(0.75,0.75);
\draw [line width=1.5pt] (0.25,0.25)--(0.0,0.25);
\draw [line width=1.5pt] (1.0,0.25)--(0.75,0.25);
\draw [line width=1.5pt] (0.25,0.25)--(0.00,0.0);
\draw [line width=1.5pt] (0.00,1.0)--(0.0,1.00);
\draw [line width=1.5pt] (1.0,1.00)--(0.75,0.75);
\draw [line width=1.5pt] (0.25,0.25)--(0.25,0.0);
\draw [line width=1.5pt] (0.25,1.0)--(0.25,0.75);

\end{tikzpicture}
&
\begin{tikzpicture}[x=2.0cm,y=2.0cm]
\draw[dash pattern=on 2pt off 2pt] (0,0) rectangle (1,1);
\draw [fill=black,line width=1.5pt] (0.25,0.25) circle (3pt);
\draw [fill=black,line width=1.5pt] (0.75,0.25) circle (3pt);
\draw [fill=black,line width=1.5pt] (0.25,0.75) circle (3pt);
\draw [fill=black,line width=1.5pt] (0.75,0.75) circle (3pt);
\draw [line width=1.5pt] (0.25,0.25)--(0.25,0.75);
\draw [line width=1.5pt] (0.25,0.75)--(0.75,0.75);
\draw [line width=1.5pt] (0.75,0.25)--(0.75,0.75);
\draw [line width=1.5pt] (0.75,0.75)--(1.0,0.75);
\draw [line width=1.5pt] (0.0,0.75)--(0.25,0.75);
\draw [line width=1.5pt] (0.25,0.75)--(0.25,1.0);
\draw [line width=1.5pt] (0.25,0.0)--(0.25,0.25);
\draw [line width=1.5pt] (0.75,0.75)--(0.75,1.0);
\draw [line width=1.5pt] (0.75,0.0)--(0.75,0.25);

\end{tikzpicture}
&
\end{tabular}

\begin{caption}{Irreducible torus graphs with four vertices}
    \label{fig_torus4}
\end{caption}
\end{figure}

\section{Application: contacts of circular arcs}
\label{sec_contacts}

In this section we describe an application to 
the study of contact graphs.
The foundational 
result in this area is the well known Koebe-Andreev-Thurston 
Circle Packing Theorem 
(\cite{Koebe}). More recently, contact graphs for 
many different classes of geometric objects have been studied,
with various restrictions placed on the allowed contacts.
See for example \cites{MR3677547,MR3775800,MR2304990}

We consider 
contact graphs arising from certain families of curves in surfaces of 
constant curvature. 
We begin by giving a model for a general class of contact problems 
and then specialise to a case of particular interest.

Let \( \alpha:[0,1] \rightarrow \Sigma \) be a curve. We say that 
\( \alpha \) is non selfoverlapping if it is injective on 
the open interval \( (0,1) \).
Now suppose that \( \alpha,\beta :[0,1]\rightarrow \Sigma \) are distinct 
curves in \( \Sigma \). We say that \( \alpha \) and \( \beta \)
are non overlapping if \( \alpha( (0,1) ) \cap \beta( (0,1)) = \emptyset \).
Let \( \mathcal C \) be a collection of
curves in \( \Sigma \) having the following properties
\begin{itemize}
    \item Every \( \alpha \in \mathcal C \) is non selfoverlapping.
    \item For every distinct  \( \alpha,\beta \in \mathcal C \), 
        \( \alpha \) and \( \beta \) are non overlapping.
\end{itemize}
We want to construct a combinatorial object that describes the
contact properties of such a collection. In order to do this we impose 
some further non degeneracy conditions on \( \mathcal C \)  as follows.

\begin{itemize}
    \item \( \alpha(0) \neq \alpha(1) \) for every \( \alpha\in \mathcal C \)
    \item For every distinct  \( \alpha,\beta \in \mathcal C \), 
        \( \{ \alpha(0),\alpha(1)\}\cap \{\beta(0),\beta(1)\} \) is empty.
\end{itemize}
In other words, we allow the end of 
one curve to touch another curve (or to touch itself), but
the point that it touches cannot be an endpoint of that curve.
We say that \( \mathcal C \) is a non degenerate collection of 
non overlapping 
curves. Note that if a collection fails the non degeneracy conditions, 
it can typically be made degenerate by an arbitrarily small perturbation. 
A contact of \( \mathcal C \) 
is a quadruple \( (\alpha,\beta,x,y) \) where \( \alpha,
\beta \in \mathcal C\), \( x \in \{ 0,1\} \) and \( \alpha(x) = \beta(y) 
\). 

Now we can define a graph
\( \Gamma_\mathcal C \) as follows. The vertex set is \( \mathcal C \)
and the edge set is \( \mathcal T \), 
the set of contacts of \( \mathcal C \). We define 
the incidence functions $s,t: \mathcal T \rightarrow \mathcal C$
by \( s(\alpha,\beta,x,y) = \alpha\) and
\( t(\alpha,\beta,x,y) = \beta \). The quadruple $(\mathcal C, \mathcal T, s,t)$
is a directed multigraph and we let $\Gamma_\mathcal C$ be the graph 
obtained by forgetting the edge orientations. 
We can construct an embedding \( |\Gamma_\mathcal C| \rightarrow 
\Sigma\) as follows. For \( \beta \in \mathcal C \), suppose that 
\( t^{-1}(\beta) = \{ (\alpha_1,\beta,x_1,y_1),\cdots,(\alpha_k,\beta,
x_k,y_k)\}\). Let \( J_\beta \) 
be a nonempty  
closed subinterval of \( [0,1] \) with the following properties.
\begin{enumerate}
    \item \( \{y_1,\cdots,y_k\} \subset J_\alpha \).
    \item \( 0 \in J_\beta \) if and only there is 
        no contact \( (\beta,\gamma,0,y) \) in 
        \( \mathcal T \).
    \item \( 1 \in J_\beta \) if and only there is 
        no contact \( (\beta,\gamma,1,y) \) in 
        \( \mathcal T \).
\end{enumerate}
In other words \( J_\beta \) is a subinterval that covers
all the `points of contact' in \( \beta \) together with 
any endpoints of \( \beta \) that do not touch a curve.
Now we observe that \( \beta:J_\beta \rightarrow \Sigma\) 
is a homeomorphism onto its 
image. So it follows from the Jordan-Schoenflies Theorem that 
\( \Sigma/\beta(J_\beta)  \) is homeomorphic to \( \Sigma \). 
Furthermore, since \( \beta(J_\beta) \cap \delta(J_\delta) = \emptyset\)
for \( \beta\neq \delta \) 
it follows that \( \Sigma \) is homeomorphic to \( \Sigma/\sim \)
where \( \sim \) is the equivalence relation that collapses each 
\( \beta(J_\beta)  \) to a point, for all \( \beta \in \mathcal C \). 
Using this homeomorphism we construct an embedding by mapping each vertex
of \( \Gamma_\mathcal C \)
(i.e. element of \( \mathcal C \)) to the corresponding point of 
\( \Sigma/\sim \). Since an edge of \( \Gamma_\mathcal C \)
is a contact \( (\alpha,\beta,x,y) \), we can construct the
corresponding edge embedding by using the restriction of 
\( \alpha \) to the component of \( [0,1]-J_\alpha \) that 
contains \( x \): see Figure \ref{fig:contactgraph}.
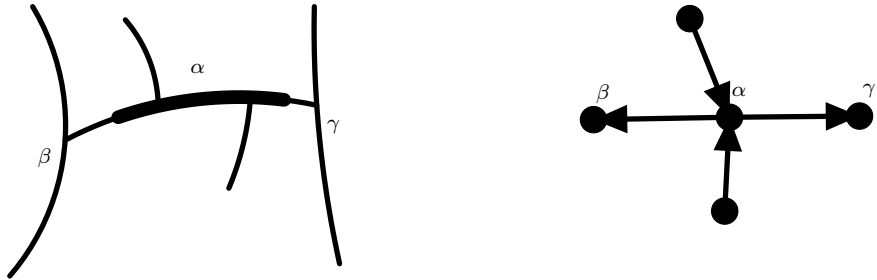
\begin{figure}
    \centering
    \begin{tikzpicture}[line cap=round,line join=round,>=triangle 45,x=0.8cm,y=0.8cm]
\draw [shift={(-14.313376241265173,3.633232806178742)},line width=2pt]  plot[domain=-0.709822785806681:0.540585023109418,variable=\t]({1*3.8410852119027363*cos(\t r)+0*3.8410852119027363*sin(\t r)},{0*3.8410852119027363*cos(\t r)+1*3.8410852119027363*sin(\t r)});
\draw [shift={(-7.588330241187384,-2.097551020408175)},line width=2pt]  plot[domain=1.3615724696804914:2.0557714410315295,variable=\t]({1*6.202819535954131*cos(\t r)+0*6.202819535954131*sin(\t r)},{0*6.202819535954131*cos(\t r)+1*6.202819535954131*sin(\t r)});
\draw [shift={(11.833431221020108,5.232766615146833)},line width=2pt]  plot[domain=3.1208382228641334:3.357982403706137,variable=\t]({1*18.177345993620637*cos(\t r)+0*18.177345993620637*sin(\t r)},{0*18.177345993620637*cos(\t r)+1*18.177345993620637*sin(\t r)});
\draw [shift={(-12.175668722024378,4.4413095947882155)},line width=2pt]  plot[domain=5.886187227121941:6.212343211515917,variable=\t]({1*4.78805570961944*cos(\t r)+0*4.78805570961944*sin(\t r)},{0*4.78805570961944*cos(\t r)+1*4.78805570961944*sin(\t r)});
\draw [shift={(-11.187308667220984,3.9171333975097165)},line width=2pt]  plot[domain=0.01803494063307368:0.7118114194542798,variable=\t]({1*2.2548256947930945*cos(\t r)+0*2.2548256947930945*sin(\t r)},{0*2.2548256947930945*cos(\t r)+1*2.2548256947930945*sin(\t r)});
\draw [shift={(-7.588330241187388,-2.09755102040814)},line width=5.2pt]  plot[domain=1.4500530883041112:1.9010842234735172,variable=\t]({1*6.202819535954096*cos(\t r)+0*6.202819535954096*sin(\t r)},{0*6.202819535954096*cos(\t r)+1*6.202819535954096*sin(\t r)});
\draw [line width=2pt,->] (0.56,3.77)-- (-1.7,3.73);
\draw [line width=2pt,->] (0.56,3.77)-- (2.72,3.79);
\draw [line width=2pt,<-] (0.56,3.77)-- (-0.1,5.41);
\draw [line width=2pt,<-] (0.56,3.77)-- (0.48,2.21);
\begin{scriptsize}
\draw[color=black] (-10.82,3.1) node {$\beta$};
\draw[color=black] (-8.28,4.58) node {$\alpha$};
\draw[color=black] (-6.02,3.62) node {$\gamma$};
\draw [fill=black] (0.56,3.77) circle (5pt);
\draw[color=black] (0.72,4.2) node {$\alpha$};
\draw [fill=black] (-1.7,3.73) circle (5pt);
\draw[color=black] (-1.54,4.16) node {$\beta$};
\draw [fill=black] (2.72,3.79) circle (5pt);
\draw[color=black] (2.88,4.22) node {$\gamma$};
\draw [fill=black] (-0.1,5.41) circle (5pt);
\draw [fill=black] (0.48,2.21) circle (5pt);
\end{scriptsize}
\end{tikzpicture}

    \caption{The construction of the contact graph associated to a 
    collection of curves. On the left we have a collection of curves. The
bold section of \( \alpha \) represents \( \alpha(J_\alpha) \). On the 
right is the corresponding graph with edge orientations as indicated.}
    \label{fig:contactgraph}
\end{figure}

We are interested in the recognition problem for contact graphs:
can we  find necessary and/or sufficient conditions for a 
surface graph to be the contact graph of a collection of curves?
Typically we are looking for conditions for which there are efficient 
algorithmic checks. As noted in the introduction, there are efficient
algorithms for checking whether or not a given graph is \( (2,l) \)-sparse.
See \cite{MR2392060} and \cite{MR1481894} for details. 

Hlin\v{e}n\'{y} (\cite{MR1644051}) has shown that a plane graph 
admits a representation by contacts of curves if and only if it is 
\( (2,0) \)-sparse. We note that he states the result 
only for simple planar graphs. It is easy to see however the the result applies to 
the multigraph model of contact graphs described above and to graphs 
embedded in arbitrary surfaces.
We include the statement to provide some context for our later result.

\begin{thm}
    \label{thm_curvesreplemma}
    Let \( G \) be a \( \Sigma \)-graph. Then \( G \cong G_\mathcal C \)
    for some \( \mathcal C \) as above if and only if \( G \) is 
    \( (2,0) \)-sparse. \qed
\end{thm}
It is worth noting here that graphs studied in \cite{MR1644051} and elsewhere 
are typically the intersection graphs of a collection of curves.
While this is usually appropriate in the plane, for non simply connected 
surfaces we propose that it is more natural to 
consider the contact graph as above. We observe that the intersection graph is 
obtained from the surface graph $G_\mathcal L$ by taking the underlying graph,
deleting all loops and replacing any 
sets of parallel edges by a single edge. 

Now we suppose that \( \Sigma \) is also equipped with a metric of constant 
curvature. In this context we can distinguish many interesting subclasses
of non selfoverlapping curves. For example, a circular arc is a 
curve of constant curvature and 
a line segment is a locally geodesic curve.
For collections of such curves the representability question 
can depend on the embedding of the graph and not just the graph itself
(in contrast to Theorem \ref{thm_curvesreplemma}).
For example, the graph \( C_2 \)  cannot 
be represented by a collection of line segments in the flat plane. However,
if \( C_2 \) is embedded as a non separating cycle in the  
torus, then it is easy to construct a representation of the 
resulting surface graph as a collection of line segments in the flat torus.

Given a \( \Sigma \)-graph \( G \) and 
a non degenerate non overlapping collection of circular arcs \( C \) 
such that \( G \cong G_\mathcal C \) we say that \( \mathcal C \)
is a CCA representation of \( \mathcal C \) (abbreviating Contacts of 
Circular Arcs).
See Figure \ref{fig:graphandCCArep} for an example in the torus.
Alam et al. (\cite{MR3677547}) have shown that any \( (2,2) \)-sparse
plane graph has a CCA representation in the flat plane.
We prove an analogous result for the flat torus.

\begin{figure}
    \centering
    \begin{tabular}[]{cc}

\begin{tikzpicture}[line width=1.5pt,vertex/.style = {inner sep=0mm,circle,draw,minimum size=7pt,fill=black}]
    \path (-0.2,-0.2) rectangle (4.2,4.2);
    \clip (0,0) rectangle (4,4);
    \draw [dash pattern = on 5pt off 5pt] (0,0) rectangle (4,4);
    \node (A) at (1,2) [vertex]{};
    \node (B) at (3,1.5) [vertex]{};
    \node (C) at (2,2.8) [vertex]{};
    \draw [->](A) -- (B);
    \draw [->]($(A)+(4,0)$) -- (B);
    \draw [->](C) -- (A);
    \draw ($(B)-(4,0)$)  -- (A);
    \draw [->]($(B)+(0,4)$)  -- (C);
    \draw (B) -- ($(C)-(0,4)$) ;

\end{tikzpicture}
&
\begin{tikzpicture}[line width=1.5pt]
    \path (-0.2,-0.2) rectangle (4.2,4.2);
    \clip (0,0) rectangle (4,4);
    \draw [dash pattern = on 5pt off 5pt] (0,0) rectangle (4,4);
    \draw
        (3,2) arc 
        (100:80:5) coordinate (A1) arc 
        (80:60.3:5) coordinate (B1);
    \draw
        (-1,2) arc 
        (100:80:5) coordinate (A2) arc 
        (80:60.3:5) coordinate (B2);
    \draw
        (A2) arc (-20:5:7); 
    \draw
        ($(A2)+(0,-4)$) arc (-20:3:7) coordinate (C) arc (3:5:7); 
    \draw
        (C) arc (-70:-35:4.58);
\end{tikzpicture}
\\
    \end{tabular}

    \caption{A torus graph and a corresponding CCA representation in the flat torus. The orientation of the graph edges is the orientation induced by the CCA representation.}
    \label{fig:graphandCCArep}
\end{figure}
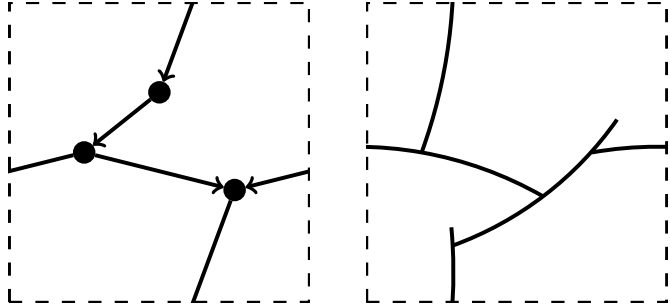

First we need a lemma to show that every sparse surface graph can be obtained by 
deleting only edges from a tight surface graph.

\begin{lem}
    \label{lem_sparse_to_tight}
    Suppose that \( \Sigma \) is a connected surface, \( l \leq 2 \) and
    \( G \) is a \( (2,l) \)-sparse \( \Sigma \)-graph. There is 
    some \( (2,2) \)-tight \( \Sigma \)-graph \( H \) such that \( V(H) = V(G) \) and 
    \( G \) is a subgraph of \( H \). 
\end{lem}

\begin{proof}
    Clearly it suffices to show that if \( \gamma(G) \geq l+1 \) then we 
    can add an edge \( e \) within some face of \( G \) 
    so that \( G \cup \{e\} \) is \( (2,l) \)-sparse.

    Now if \( G \) has no tight subgraph then we can add any edge without violating
    the sparsity count. So we assume that \( G \) has some nonempty tight subgraph. 
    Let \( L \) be a maximal tight subgraph of \( G \). If \( L \) spans all vertices of
    \( G \) then \( L = G \) and \( G \) is already tight, so we assume that 
    \( L \) is not spanning in \( G \). Since \( \Sigma \) is connected there is some 
    face \( F \) of \( G \) whose boundary contains vertices \( u \in L \) and \( v \not\in
    L\). Let \( e \) be a new edge that joins \( u \) and \( v \) through a path in 
    \( F \). We claim that \( G \cup \{e\} \) is \( (2,l) \)-sparse. If not then 
    there must be some tight subgraph \( K \) of \( G \) such that \( u,v \in K \). 
    But \( K \cap L \) is nonempty, so by Lemma \ref{lem_int_unions}, \( K\cup L \)
    is \( (2,l) \)-tight. This contradicts the maximality of \( L \).
\end{proof}

\begin{thm}
    \label{thm:torusCCA}
    Every \( (2,2) \)-sparse torus graph admits a CCA representation 
    in the flat torus.
\end{thm}

\begin{proof}
    First observe that edge deletion is CCA representable: just 
    shorten one of the arcs slightly. So by Lemma \ref{lem_sparse_to_tight}
    it suffices to prove the theorem for \( (2,2) \)-tight torus graphs.
    To that end we must show that
    \begin{enumerate}[(a)]
        \item each irreducible \( (2,2) \)-tight torus graph
            has a CCA representation.
        \item if \( G \rightarrow G'\) is a digon, triangle or quadrilateral
            contraction move and \( G' \) has a CCA representation, then 
            \( G \) also has a CCA representation. In other words 
            the relevant vertex splitting moves are CCA representable.
    \end{enumerate}

    For (a) it is possible to give an explicit CCA representation for each 
    of the 116  irreducibles listed in Appendix A. 
    We will not describe those
    here but below we shall explain a simple method to make 
    these constructions easily. Full details are given in 
    \cite{QSthesis}.
    
    \begin{figure}
        \begin{center}

\begin{tikzpicture}[line width=1.0pt,rotate = -90]
    \draw (0,0) coordinate (A) arc (110:98:5) coordinate (B) 
        arc (98:83:5) coordinate (C) arc (83:70:5) coordinate (D) 
        arc (70:60:5) coordinate (E) node {};
    \draw (A) arc (0:20:3) coordinate (A1) arc (20:30:3);
    \draw (A) arc (0:-15:3) coordinate (A2) arc (-15:-20:3);
    \draw (E) arc (-190:-210:3);
    \draw (E) arc (-190:-170:3);
    \draw (D) arc (0:20:6);
    \draw (B) arc (180:195:6);
    \draw (C) arc (0:10:10);

    \node (z) at ($(A)!0.5!(E)+(-.2,0.2)$) {$z$};
    \node (v2) at (A) [above]{$v_2$};
    \node (v4) at ($(D)+(0.3,0.5)$){$v_4$};

    \tikzset{shift={(0,6)}};
    \draw (0,0) coordinate (A) arc (110:98:5) coordinate (B) arc (98:83:5) coordinate (C) arc (83:70:5) coordinate (D) arc (70:60:5) coordinate (E) node {};
    \draw (A) arc (0:5:3) coordinate (S) arc (5:20:3) coordinate (A1) arc (20:30:3);
    \draw (A) arc (0:-15:3) coordinate (A2) arc (-15:-20:3);
    \draw (E) arc (-190:-210:3);
    \draw (E) arc (-190:-170:3);
    \draw (D) arc (0:20:6);
    \draw (B) arc (180:195:6);
    \path [name path={CP}] (C) arc (0:10:10);
    \draw [name path = {Z}] (S) arc (110:71:5.1);
    \draw [name intersections={of=Z and CP,by={I}}] (I) arc (2:10:10);

    \node (v1) at ($(A)!0.5!(E)+(-.2,0.2)$) {$v_1$};
    \node (v3) at ($(A)!0.5!(C)+(0,0.4)$) [right]{$v_3$};
    \node (v2) at (A) [above]{$v_2$};
    \node (v4) at ($(D)+(0.3,0.5)$){$v_4$};
\end{tikzpicture}

        \end{center}
        \caption{A CCA representation of a quadrilateral splitting move. The contact
        graph of the configuration on the left is a quadrilateral contraction of the contact
    graph of the configuration on the right.}
        \label{fig:CCAinversequad}
    \end{figure}
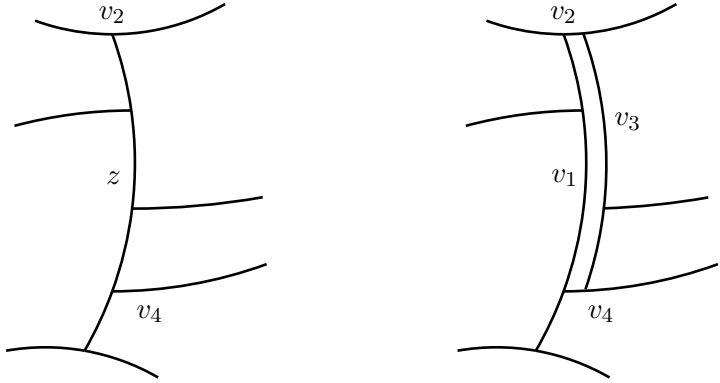

    For (b), see Figure \ref{fig:CCAinversequad} for an illustration of the CCA representation of the 
    quadrilateral split.  It is easily seen that any 
    quadrilateral split is similarly CCA representable. The 
    digon and triangle 
    vertex splits are also representable and indeed 
    have already been dealt with in the plane context in \cite{MR3677547}. 
    We observe that the constructions described there 
    work equally well for 
    torus graphs. 
\end{proof}

In order to construct the CCA representations of the irreducibles
mentioned in the proof we can make use of topological Henneberg moves.
We remind the reader that a Henneberg 
vertex addition move is the operation of adding a new vertex to 
a graph and two edges from that vertex to the existing graph. 
Note that we allow the two new edges to be parallel. Moreover 
in the context of surface graphs we insist that the new vertex is placed in
some face of the existing graph and the two edges are incident 
with vertices in the boundary of that face. We refer to such an operation 
as a topological Henneberg move. Clearly a Henneberg move
is the inverse operation to divalent vertex deletion. It is well known 
(and elementary) that divalent vertex deletions preserve \( (2,l) \)-sparseness
for all \( l \). On the other hand Henneberg moves preserve 
\( (2,l) \)-sparseness for \( l \leq 2 \), and for \( l = 3 \) if we also 
insist that the new edges are not parallel.

It turns out that there are just 12 irreducibles that have no vertices of
degree 2. In Figure \ref{fig_no_deg_two} we give diagrams of each of these torus 
graphs and in Figure \ref{fig_CCArepsnodegree2} we give  sample CCA representations 
of each of these in the flat torus.
We observe that each of the 116  irreducible graphs 
can be constructed
by a sequence topological Henneberg moves
from one of the torus graphs in Figure \ref{fig_no_deg_two}:
indeed one easily sees that at most five Henneberg moves are required.

It remains to show that the required topological Henneberg moves are 
CCA representable.
A CCA representation of a topological Henneberg  move is illustrated in 
Figure \ref{fig:CCAhenneberg}. 
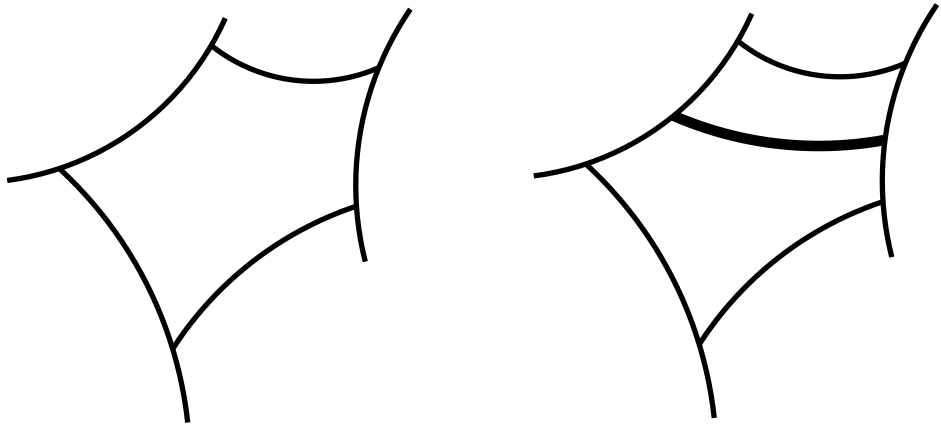
\begin{figure}
    \centering
\begin{tikzpicture}[line join=round,>=triangle 45,x=1cm,y=1cm]
\draw [shift={(-5.268479025529105,7.702508793584722)},line width=2pt]  plot[domain=4.037237719266882:5.1234322511114065,variable=\t]({1*2.1735563278407204*cos(\t r)+0*2.1735563278407204*sin(\t r)},{0*2.1735563278407204*cos(\t r)+1*2.1735563278407204*sin(\t r)});
\draw [shift={(-0.5609160305343568,4.1458778625954205)},line width=2pt]  plot[domain=2.540597002328743:3.38917059271077,variable=\t]({1*4.145484746724701*cos(\t r)+0*4.145484746724701*sin(\t r)},{0*4.145484746724701*cos(\t r)+1*4.145484746724701*sin(\t r)});
\draw [shift={(-3.2202102515328495,-0.5519987227543979)},line width=2pt]  plot[domain=1.89321994944731:2.5698700707776663,variable=\t]({1*4.661033064650258*cos(\t r)+0*4.661033064650258*sin(\t r)},{0*4.661033064650258*cos(\t r)+1*4.661033064650258*sin(\t r)});
\draw [shift={(-12.289632037947849,0.40864652286778974)},line width=2pt]  plot[domain=0.10824689414403717:0.8270296741331332,variable=\t]({1*5.381127651971449*cos(\t r)+0*5.381127651971449*sin(\t r)},{0*5.381127651971449*cos(\t r)+1*5.381127651971449*sin(\t r)});
\draw [shift={(-2.812827521206411,7.931573986804904)},line width=2pt]  plot[domain=4.840810624014497:5.864332410564741,variable=\t]({1*3.6919764253928484*cos(\t r)+0*3.6919764253928484*sin(\t r)},{0*3.6919764253928484*cos(\t r)+1*3.6919764253928484*sin(\t r)});
\draw [shift={(1.7315209744708944,7.762508793584735)},line width=2pt]  plot[domain=4.037237719266886:5.123432251111404,variable=\t]({1*2.1735563278407315*cos(\t r)+0*2.1735563278407315*sin(\t r)},{0*2.1735563278407315*cos(\t r)+1*2.1735563278407315*sin(\t r)});
\draw [shift={(6.4390839694656545,4.205877862595419)},line width=2pt]  plot[domain=2.5405970023287443:3.389170592710769,variable=\t]({1*4.14548474672471*cos(\t r)+0*4.14548474672471*sin(\t r)},{0*4.14548474672471*cos(\t r)+1*4.14548474672471*sin(\t r)});
\draw [shift={(3.779789748467141,-0.49199872275438783)},line width=2pt]  plot[domain=1.8932199494473083:2.5698700707776667,variable=\t]({1*4.661033064650245*cos(\t r)+0*4.661033064650245*sin(\t r)},{0*4.661033064650245*cos(\t r)+1*4.661033064650245*sin(\t r)});
\draw [shift={(-5.289632037947856,0.4686465228677842)},line width=2pt]  plot[domain=0.10824689414403807:0.8270296741331332,variable=\t]({1*5.381127651971457*cos(\t r)+0*5.381127651971457*sin(\t r)},{0*5.381127651971457*cos(\t r)+1*5.381127651971457*sin(\t r)});
\draw [shift={(1.4503866834932972,9.50720162603596)},line width=4pt]  plot[domain=4.30204017008921:4.895037345012064,variable=\t]({1*4.838920439303007*cos(\t r)+0*4.838920439303007*sin(\t r)},{0*4.838920439303007*cos(\t r)+1*4.838920439303007*sin(\t r)});
\draw [shift={(-9.812827521206408,7.8715739868049)},line width=2pt]  plot[domain=4.840810624014497:5.864332410564742,variable=\t]({1*3.691976425392844*cos(\t r)+0*3.691976425392844*sin(\t r)},{0*3.691976425392844*cos(\t r)+1*3.691976425392844*sin(\t r)});
\end{tikzpicture}

\caption{A CCA representation of a topological Henneberg move. The bold arc on the right represents the new vertex, that touches two of the initial arcs.}
    \label{fig:CCAhenneberg}
\end{figure}
In general of course, topological Henneberg moves can fail to 
be CCA representable given a fixed representation of the initial graph. 
It is easy to construct an CCA representation of a graph that has a highly 
nonconvex region and so may not admit the required circular arc to realise a 
topological Henneberg move.
However, it is readily verified that given the CCA representations in Figure
\ref{fig_CCArepsnodegree2}, it is possible to represent all the necessary 
Henneberg moves that are required to construct CCA representations of 
the full set of 116  irreducible graphs. See \cite{QSthesis} for complete 
details of this. 

\begin{figure}
    \centering

\begin{tabular}{ccccc}
\begin{tikzpicture}[x=2.0cm,y=2.0cm]
\draw[dash pattern=on 2pt off 2pt] (0,0) rectangle (1,1);
\draw [fill=black,line width=1.5pt] (0.5,0.5) circle (3pt);

\draw (0.5,-0.15) node {$G^1_1$};

\end{tikzpicture}
&
\begin{tikzpicture}[x=2.0cm,y=2.0cm]
\draw[dash pattern=on 2pt off 2pt] (0,0) rectangle (1,1);
\draw [fill=black,line width=1.5pt] (0.25,0.5) circle (3pt);
\draw [fill=black,line width=1.5pt] (0.5,0.75) circle (3pt);
\draw [fill=black,line width=1.5pt] (0.75,0.5) circle (3pt);
\draw [fill=black,line width=1.5pt] (0.5,0.25) circle (3pt);
\draw [line width=1.5pt] (0.25,0.5)--(0.5,0.75);
\draw [line width=1.5pt] (0.5,0.75)--(0.75,0.5);
\draw [line width=1.5pt] (0.75,0.5)--(0.5,0.25);
\draw [line width=1.5pt] (0.5,0.25)--(0.25,0.5);
\draw [line width=1.5pt] (0.25,0.5)--(0.0,0.5);
\draw [line width=1.5pt] (1.0,0.5)--(0.75,0.5);
\draw [line width=1.5pt] (0.5,0.75)--(0.5,1.0);
\draw [line width=1.5pt] (0.5,0.0)--(0.5,0.25);

\draw (0.5,-0.15) node {$G^4_1$};

\end{tikzpicture}
&
\begin{tikzpicture}[x=2.0cm,y=2.0cm]
\draw[dash pattern=on 2pt off 2pt] (0,0) rectangle (1,1);
\draw [fill=black,line width=1.5pt] (0.25,0.25) circle (3pt);
\draw [fill=black,line width=1.5pt] (0.75,0.25) circle (3pt);
\draw [fill=black,line width=1.5pt] (0.25,0.75) circle (3pt);
\draw [fill=black,line width=1.5pt] (0.75,0.75) circle (3pt);
\draw [line width=1.5pt] (0.75,0.75)--(0.75,0.25);
\draw [line width=1.5pt] (0.25,0.25)--(0.25,0.75);
\draw [line width=1.5pt] (0.25,0.75)--(0.75,0.75);
\draw [line width=1.5pt] (0.25,0.25)--(0.0,0.25);
\draw [line width=1.5pt] (1.0,0.25)--(0.75,0.25);
\draw [line width=1.5pt] (0.75,0.25)--(0.75,0.0);
\draw [line width=1.5pt] (0.75,1.0)--(0.75,0.75);
\draw [line width=1.5pt] (0.25,0.25)--(0.25,0.0);
\draw [line width=1.5pt] (0.25,1.0)--(0.25,0.75);

\draw (0.5,-0.15) node {$G^4_2$};

\end{tikzpicture}
&
\begin{tikzpicture}[x=2.0cm,y=2.0cm]
\draw[dash pattern=on 2pt off 2pt] (0,0) rectangle (1,1);
\draw [fill=black,line width=1.5pt] (0.25,0.25) circle (3pt);
\draw [fill=black,line width=1.5pt] (0.25,0.75) circle (3pt);
\draw [fill=black,line width=1.5pt] (0.5,0.75) circle (3pt);
\draw [fill=black,line width=1.5pt] (0.75,0.75) circle (3pt);
\draw [fill=black,line width=1.5pt] (0.75,0.25) circle (3pt);
\draw [line width=1.5pt] (0.25,0.25)--(0.25,0.75);
\draw [line width=1.5pt] (0.25,0.75)--(0.5,0.75);
\draw [line width=1.5pt] (0.5,0.75)--(0.75,0.75);
\draw [line width=1.5pt] (0.75,0.75)--(0.50,1.0);
\draw [line width=1.5pt] (0.50,0.0)--(0.25,0.25);
\draw [line width=1.5pt] (0.25,0.75)--(1.0,0.75);
\draw [line width=1.5pt] (0.0,0.75)--(0.75,0.75);
\draw [line width=1.5pt] (0.25,0.75)--(0.00,1.0);
\draw [line width=1.5pt] (0.00,0.0)--(0.0,0.00);
\draw [line width=1.5pt] (1.0,0.00)--(0.75,0.25);
\draw [line width=1.5pt] (0.5,0.75)--(0.75,0.25);
\draw [line width=1.5pt] (0.25,0.25)--(0.0,0.25);
\draw [line width=1.5pt] (1.0,0.25)--(0.75,0.25);

\draw (0.5,-0.15) node {$G^5_1$};

\end{tikzpicture}
&
\begin{tikzpicture}[x=2.0cm,y=2.0cm]
\draw[dash pattern=on 2pt off 2pt] (0,0) rectangle (1,1);
\draw [fill=black,line width=1.5pt] (0.25,0.5) circle (3pt);
\draw [fill=black,line width=1.5pt] (0.5,0.75) circle (3pt);
\draw [fill=black,line width=1.5pt] (0.75,0.5) circle (3pt);
\draw [fill=black,line width=1.5pt] (0.5,0.25) circle (3pt);
\draw [fill=black,line width=1.5pt] (0.75,0.25) circle (3pt);
\draw [line width=1.5pt] (0.25,0.5)--(0.5,0.75);
\draw [line width=1.5pt] (0.5,0.75)--(0.75,0.5);
\draw [line width=1.5pt] (0.75,0.5)--(0.5,0.25);
\draw [line width=1.5pt] (0.5,0.25)--(0.25,0.5);
\draw [line width=1.5pt] (0.5,0.75)--(0.5,1.0);
\draw [line width=1.5pt] (0.5,0.0)--(0.5,0.25);
\draw [line width=1.5pt] (0.75,0.5)--(0.75,0.25);
\draw [line width=1.5pt] (0.75,0.5)--(0.75,1.0);
\draw [line width=1.5pt] (0.75,0.0)--(0.75,0.25);
\draw [line width=1.5pt] (0.25,0.5)--(0.0,0.375);
\draw [line width=1.5pt] (1.0,0.375)--(0.75,0.25);

\draw (0.5,-0.15) node {$G^5_2$};

\end{tikzpicture}
\\
\begin{tikzpicture}[x=2.0cm,y=2.0cm]
\draw[dash pattern=on 2pt off 2pt] (0,0) rectangle (1,1);
\draw [fill=black,line width=1.5pt] (0,0.5) circle (3pt);
\draw [fill=black,line width=1.5pt] (1,0.5) circle (3pt);
\draw [fill=black,line width=1.5pt] (0.5,0.5) circle (3pt);
\draw [fill=black,line width=1.5pt] (0.25,0.25) circle (3pt);
\draw [fill=black,line width=1.5pt] (0.25,0.75) circle (3pt);
\draw [fill=black,line width=1.5pt] (0.75,0.25) circle (3pt);
\draw [fill=black,line width=1.5pt] (0.75,0.75) circle (3pt);
\draw [line width=1.5pt] (0,0.5)--(0.25,0.25);
\draw [line width=1.5pt] (0,0.5)--(0.25,0.75);
\draw [line width=1.5pt] (0.5,0.5)--(0.25,0.25);
\draw [line width=1.5pt] (0.25,0.25)--(0.25,0.0);
\draw [line width=1.5pt] (0.25,1.0)--(0.25,0.75);
\draw [line width=1.5pt] (0.5,0.5)--(0.25,0.75);
\draw [line width=1.5pt] (0.5,0.5)--(0.75,0.25);
\draw [line width=1.5pt] (0.5,0.5)--(0.75,0.75);
\draw [line width=1.5pt] (0,0.5)--(0.0,0.5);
\draw [line width=1.5pt] (1.0,0.5)--(0.75,0.75);
\draw [line width=1.5pt] (0,0.5)--(0.0,0.5);
\draw [line width=1.5pt] (1.0,0.5)--(0.75,0.25);
\draw [line width=1.5pt] (0.75,0.25)--(0.75,0.0);
\draw [line width=1.5pt] (0.75,1.0)--(0.75,0.75);

\draw (0.5,-0.15) node {$G^6_1$};

\end{tikzpicture}
&
\begin{tikzpicture}[x=2.0cm,y=2.0cm]
\draw[dash pattern=on 2pt off 2pt] (0,0) rectangle (1,1);
\draw [fill=black,line width=1.5pt] (0.2,0.2) circle (3pt);
\draw [fill=black,line width=1.5pt] (0.2,0.4) circle (3pt);
\draw [fill=black,line width=1.5pt] (0.2,0.6) circle (3pt);
\draw [fill=black,line width=1.5pt] (0.2,0.8) circle (3pt);
\draw [fill=black,line width=1.5pt] (0.8,0.8) circle (3pt);
\draw [fill=black,line width=1.5pt] (0.8,0.2) circle (3pt);
\draw [line width=1.5pt] (0.2,0.2)--(0.2,0.4);
\draw [line width=1.5pt] (0.2,0.6)--(0.2,0.4);
\draw [line width=1.5pt] (0.2,0.6)--(0.2,0.8);
\draw [line width=1.5pt] (0.2,0.8)--(0.2,1.0);
\draw [line width=1.5pt] (0.2,0.0)--(0.2,0.2);
\draw [line width=1.5pt] (0.8,0.2)--(1.0,2.00000000000E-17);
\draw [line width=1.5pt] (0.0,2.00000000000E-17)--(2.0000000000E-17,0.0);
\draw [line width=1.5pt] (2.0000000000E-17,1.0)--(0.2,0.80000000000000004);
\draw [line width=1.5pt] (0.8,0.8)--(0.2,0.8);
\draw [line width=1.5pt] (0.8,0.8)--(0.8,1.0);
\draw [line width=1.5pt] (0.8,0.0)--(0.8,0.2);
\draw [line width=1.5pt] (0.8,0.8)--(1.0,0.6);
\draw [line width=1.5pt] (0.0,0.6)--(0.2,0.4);
\draw [line width=1.5pt] (0.8,0.2)--(0.2,0.4);
\draw [line width=1.5pt] (0.2,0.6)--(1.0,0.28);
\draw [line width=1.5pt] (0.0,0.28)--(0.2,0.2);

\draw (0.5,-0.15) node {$G_{2}^6$};

\end{tikzpicture}
&
\begin{tikzpicture}[x=2.0cm,y=2.0cm]
\draw[dash pattern=on 2pt off 2pt] (0,0) rectangle (1,1);
\draw [fill=black,line width=1.5pt] (0.2,0.2) circle (3pt);
\draw [fill=black,line width=1.5pt] (0.2,0.4) circle (3pt);
\draw [fill=black,line width=1.5pt] (0.2,0.6) circle (3pt);
\draw [fill=black,line width=1.5pt] (0.2,0.8) circle (3pt);
\draw [fill=black,line width=1.5pt] (0.8,0.8) circle (3pt);
\draw [fill=black,line width=1.5pt] (0.8,0.4) circle (3pt);
\draw [line width=1.5pt] (0.2,0.2)--(0.2,0.4);
\draw [line width=1.5pt] (0.2,0.6)--(0.2,0.4);
\draw [line width=1.5pt] (0.2,0.6)--(0.2,0.8);
\draw [line width=1.5pt] (0.2,0.8)--(0.2,1.0);
\draw [line width=1.5pt] (0.2,0.0)--(0.2,0.2);
\draw [line width=1.5pt] (0.8,0.8)--(0.2,0.8);
\draw [line width=1.5pt] (0.8,0.8)--(0.8,0.4);
\draw [line width=1.5pt] (0.2,0.4)--(0.8,0.4);
\draw [line width=1.5pt] (0.2,0.2)--(0.0,0.24);
\draw [line width=1.5pt] (1.0,0.24)--(0.2,0.4);
\draw [line width=1.5pt] (0.8,0.4)--(1.0,0.5);
\draw [line width=1.5pt] (0.0,0.5)--(0.2,0.6);
\draw [line width=1.5pt] (0.8,0.8)--(1.0,1.0);
\draw [line width=1.5pt] (1.0,0.0)--(1.0,0.0);
\draw [line width=1.5pt] (0.0,0.0)--(0.2,0.2);

\draw (0.5,-0.15) node {$G_{3}^6$};

\end{tikzpicture}
&
\begin{tikzpicture}[x=2.0cm,y=2.0cm]
\draw[dash pattern=on 2pt off 2pt] (0,0) rectangle (1,1);
\draw [fill=black,line width=1.5pt] (0.2,0.2) circle (3pt);
\draw [fill=black,line width=1.5pt] (0.2,0.4) circle (3pt);
\draw [fill=black,line width=1.5pt] (0.2,0.6) circle (3pt);
\draw [fill=black,line width=1.5pt] (0.2,0.8) circle (3pt);
\draw [fill=black,line width=1.5pt] (0.8,0.8) circle (3pt);
\draw [fill=black,line width=1.5pt] (0.8,0.2) circle (3pt);
\draw [line width=1.5pt] (0.2,0.2)--(0.2,0.4);
\draw [line width=1.5pt] (0.2,0.6)--(0.2,0.4);
\draw [line width=1.5pt] (0.2,0.6)--(0.2,0.8);
\draw [line width=1.5pt] (0.2,0.8)--(0.2,1.0);
\draw [line width=1.5pt] (0.2,0.0)--(0.2,0.2);
\draw [line width=1.5pt] (0.8,0.8)--(0.2,0.6);
\draw [line width=1.5pt] (0.8,0.8)--(0.8,0.2);
\draw [line width=1.5pt] (0.2,0.2)--(0.8,0.2);
\draw [line width=1.5pt] (0.8,0.8)--(1.0,0.6);
\draw [line width=1.5pt] (0.0,0.6)--(0.2,0.4);
\draw [line width=1.5pt] (0.8,0.2)--(1.0,2.00000000000E-17);
\draw [line width=1.5pt] (0.0,2.00000000000E-17)--(2.0000000000E-17,0.0);
\draw [line width=1.5pt] (2.0000000000E-17,1.0)--(0.2,0.80000000000000004);
\draw [line width=1.5pt] (0.8,0.8)--(0.8,1.0);
\draw [line width=1.5pt] (0.8,0.0)--(0.8,0.2);

\draw (0.5,-0.15) node {$G_{4}^6$};

\end{tikzpicture}
&
\begin{tikzpicture}[x=2.0cm,y=2.0cm]
\draw[dash pattern=on 2pt off 2pt] (0,0) rectangle (1,1);
\draw [fill=black,line width=1.5pt] (0.25,0.5) circle (3pt);
\draw [fill=black,line width=1.5pt] (0.5,0.75) circle (3pt);
\draw [fill=black,line width=1.5pt] (0.75,0.5) circle (3pt);
\draw [fill=black,line width=1.5pt] (0.5,0.25) circle (3pt);
\draw [fill=black,line width=1.5pt] (0.25,0.25) circle (3pt);
\draw [fill=black,line width=1.5pt] (0.75,0.25) circle (3pt);
\draw [line width=1.5pt] (0.25,0.5)--(0.5,0.75);
\draw [line width=1.5pt] (0.5,0.75)--(0.75,0.5);
\draw [line width=1.5pt] (0.75,0.5)--(0.5,0.25);
\draw [line width=1.5pt] (0.5,0.25)--(0.25,0.5);
\draw [line width=1.5pt] (0.5,0.75)--(0.5,1.0);
\draw [line width=1.5pt] (0.5,0.0)--(0.5,0.25);
\draw [line width=1.5pt] (0.25,0.25)--(0.0,0.25);
\draw [line width=1.5pt] (1.0,0.25)--(0.75,0.25);
\draw [line width=1.5pt] (0.25,0.25)--(0.25,1.0);
\draw [line width=1.5pt] (0.25,0.0)--(0.25,0.5);
\draw [line width=1.5pt] (0.75,0.5)--(0.75,0.25);
\draw [line width=1.5pt] (0.75,0.5)--(0.75,1.0);
\draw [line width=1.5pt] (0.75,0.0)--(0.75,0.25);

\draw (0.5,-0.15) node {$G_{5}^6$};

\end{tikzpicture}
\\
\begin{tikzpicture}[x=2.0cm,y=2.0cm]
\draw[dash pattern=on 2pt off 2pt] (0,0) rectangle (1,1);
\draw [fill=black,line width=1.5pt] (0.0,0.5) circle (3pt);
\draw [fill=black,line width=1.5pt] (1.0,0.5) circle (3pt);
\draw [fill=black,line width=1.5pt] (0.2,0.7) circle (3pt);
\draw [fill=black,line width=1.5pt] (0.4,0.5) circle (3pt);
\draw [fill=black,line width=1.5pt] (0.2,0.3) circle (3pt);
\draw [fill=black,line width=1.5pt] (0.8,0.3) circle (3pt);
\draw [fill=black,line width=1.5pt] (0.6,0.5) circle (3pt);
\draw [fill=black,line width=1.5pt] (0.8,0.7) circle (3pt);
\draw [line width=1.5pt] (0.0,0.5)--(0.2,0.7);
\draw [line width=1.5pt] (0.2,0.7)--(0.4,0.5);
\draw [line width=1.5pt] (0.4,0.5)--(0.2,0.3);
\draw [line width=1.5pt] (0.2,0.3)--(0.0,0.5);
\draw [line width=1.5pt] (0.2,0.7)--(0.2,1.0);
\draw [line width=1.5pt] (0.2,0.0)--(0.2,0.3);
\draw [line width=1.5pt] (0.0,0.5)--(0.0,0.5);
\draw [line width=1.5pt] (1.0,0.5)--(0.80000000000000004,0.3);
\draw [line width=1.5pt] (0.4,0.5)--(0.6,0.5);
\draw [line width=1.5pt] (0.4,0.5)--(0.5,1.0);
\draw [line width=1.5pt] (0.5,0.0)--(0.6,0.5);
\draw [line width=1.5pt] (0.8,0.3)--(0.6,0.5);
\draw [line width=1.5pt] (0.8,0.7)--(0.8,1.0);
\draw [line width=1.5pt] (0.8,0.0)--(0.8,0.3);
\draw [line width=1.5pt] (0.8,0.7)--(0.6,0.5);
\draw [line width=1.5pt] (0.8,0.7)--(1.0,0.5);

\draw (0.5,-0.15) node {$G_{1}^7$};

\end{tikzpicture}
&
\begin{tikzpicture}[x=2.0cm,y=2.0cm]
\draw[dash pattern=on 2pt off 2pt] (0,0) rectangle (1,1);
\draw [fill=black,line width=1.5pt] (0.1,0.5) circle (3pt);
\draw [fill=black,line width=1.5pt] (0.25,0.7) circle (3pt);
\draw [fill=black,line width=1.5pt] (0.4,0.5) circle (3pt);
\draw [fill=black,line width=1.5pt] (0.25,0.3) circle (3pt);
\draw [fill=black,line width=1.5pt] (0.75,0.3) circle (3pt);
\draw [fill=black,line width=1.5pt] (0.9,0.5) circle (3pt);
\draw [fill=black,line width=1.5pt] (0.75,0.7) circle (3pt);
\draw [fill=black,line width=1.5pt] (0.6,0.5) circle (3pt);
\draw [line width=1.5pt] (0.1,0.5)--(0.25,0.7);
\draw [line width=1.5pt] (0.25,0.7)--(0.4,0.5);
\draw [line width=1.5pt] (0.25,0.3)--(0.1,0.5);
\draw [line width=1.5pt] (0.25,0.7)--(0.25,1.0);
\draw [line width=1.5pt] (0.25,0.0)--(0.25,0.3);
\draw [line width=1.5pt] (0.4,0.5)--(0.25,0.3);
\draw [line width=1.5pt] (0.75,0.3)--(0.9,0.5);
\draw [line width=1.5pt] (0.75,0.7)--(0.75,1.0);
\draw [line width=1.5pt] (0.75,0.0)--(0.75,0.3);
\draw [line width=1.5pt] (0.75,0.7)--(0.9,0.5);
\draw [line width=1.5pt] (0.9,0.5)--(1.0,0.5);
\draw [line width=1.5pt] (0.0,0.5)--(0.1,0.5);
\draw [line width=1.5pt] (0.6,0.5)--(0.4,0.5);
\draw [line width=1.5pt] (0.6,0.5)--(0.75,0.3);
\draw [line width=1.5pt] (0.75,0.7)--(0.6,0.5);
\draw [line width=1.5pt] (0.4,0.5)--(0.5,1.0);
\draw [line width=1.5pt] (0.5,0.0)--(0.6,0.5);
\draw [line width=1.5pt] (0.9,0.5)--(1.0,1.0);
\draw [line width=1.5pt] (1.0,0.0)--(1.0,0.0);
\draw [line width=1.5pt] (0.0,0.0)--(0.1,0.5);

\draw (0.5,-0.15) node {$G_{1}^8$};

\end{tikzpicture}
&
\end{tabular}

\caption{The irreducibles that have no vertex of degree 2.}
    \label{fig_no_deg_two}
\end{figure}
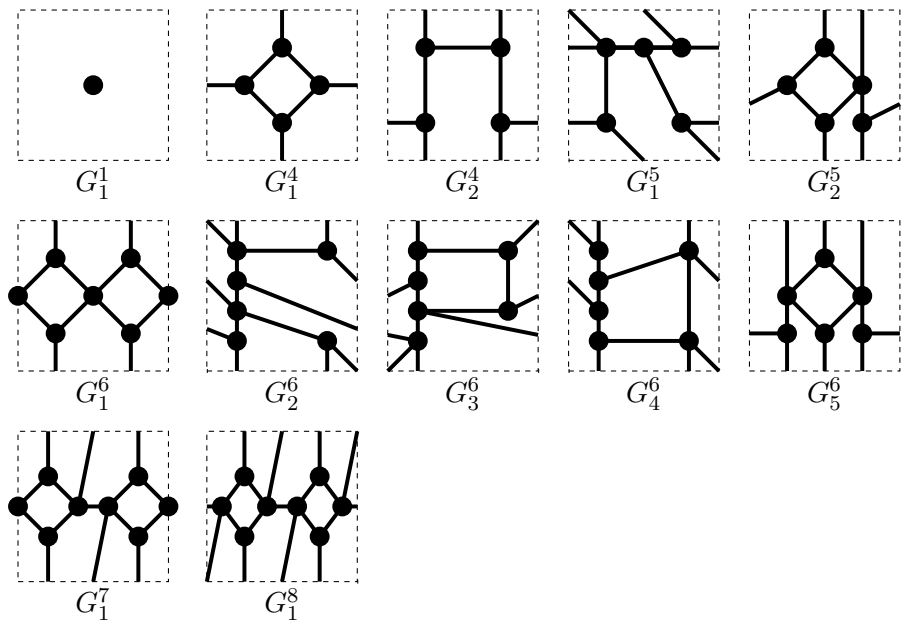

\newcommand{\bwt}{4.2cm}
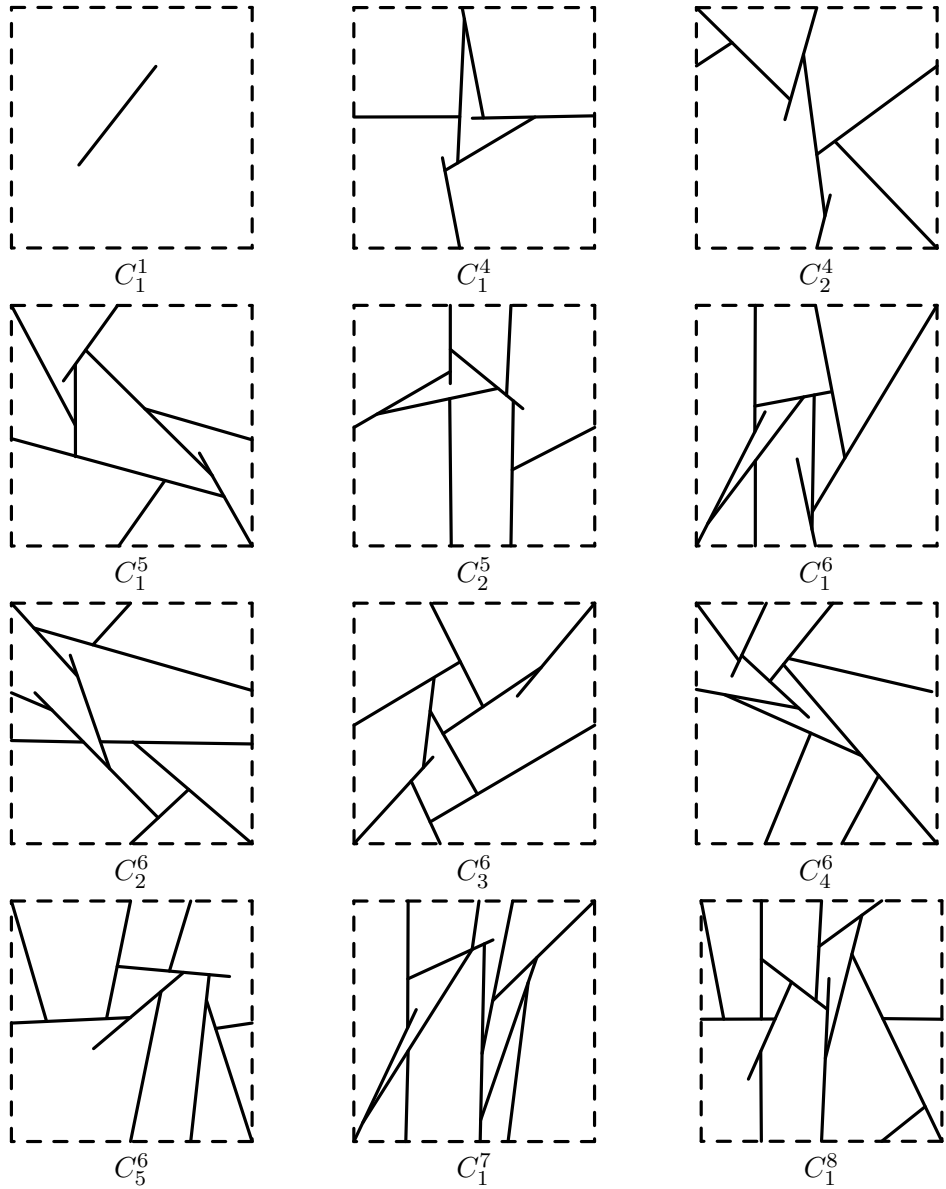
\begin{figure}
    \centering
    \begin{tabular}{ccc}
\begin{minipage}{\bwt}
    \centering
\begin{tikzpicture}[line cap=round,line join=round,>=triangle 45,x=0.8cm,y=0.8cm]
\draw [line width=1.2pt,dash pattern=on 5pt off 5pt] (-4,6)-- (0,6);
\draw [line width=1.2pt,dash pattern=on 5pt off 5pt] (0,6)-- (0,2);
\draw [line width=1.2pt,dash pattern=on 5pt off 5pt] (0,2)-- (-4,2);
\draw [line width=1.2pt,dash pattern=on 5pt off 5pt] (-4,2)-- (-4,6);
\draw [line width=1.2pt] (-1.6,5.02)-- (-2.88,3.38);
\end{tikzpicture}\\
$C^1_1$
\end{minipage}
&
\begin{minipage}{\bwt}
    \centering
\begin{tikzpicture}[line cap=round,line join=round,>=triangle 45,x=0.8cm,y=0.8cm]
\draw [line width=1.2pt,dash pattern=on 5pt off 5pt] (-4.2,6.22)-- (-0.2,6.22);
\draw [line width=1.2pt,dash pattern=on 5pt off 5pt] (-0.2,6.22)-- (-0.2,2.22);
\draw [line width=1.2pt,dash pattern=on 5pt off 5pt] (-0.2,2.22)-- (-4.2,2.22);
\draw [line width=1.2pt,dash pattern=on 5pt off 5pt] (-4.2,2.22)-- (-4.2,6.22);
\draw [line width=1.2pt] (-0.2,4.42)-- (-2.2309580225002827,4.379657659768442);
\draw [line width=1.2pt] (-2.44,2.22)-- (-2.7276695229750523,3.722245379728306);
\draw [line width=1.2pt] (-1.1882140953316445,4.40037041198596)-- (-2.6869531064288505,3.509619280157343);
\draw [line width=1.2pt] (-2.4,6.22)-- (-2.0464334461235394,4.38332300053807);
\draw [line width=1.2pt] (-2.4749220115093644,3.6356365095641667)-- (-2.3633609788729055,6.029670777257831);
\draw [line width=1.2pt] (-4.2,4.4)-- (-2.4391321828161234,4.403665222651778);
\end{tikzpicture}
\\
$C^4_1$
\end{minipage}
&

\begin{minipage}{\bwt}
    \centering
\begin{tikzpicture}[line cap=round,line join=round,>=triangle 45,x=0.8cm,y=0.8cm]
\draw [line width=1.2pt,dash pattern=on 5pt off 5pt] (-4,3.97)-- (0,3.97);
\draw [line width=1.2pt,dash pattern=on 5pt off 5pt] (0,3.97)-- (0,-0.03);
\draw [line width=1.2pt,dash pattern=on 5pt off 5pt] (0,-0.03)-- (-4,-0.03);
\draw [line width=1.2pt,dash pattern=on 5pt off 5pt] (-4,-0.03)-- (-4,3.97);
\draw [line width=1.2pt] (-2,3.97)-- (-2.5291466124506976,2.1080250049279345);
\draw [line width=1.2pt] (-2,-0.03)-- (-1.7760191668863112,0.852812595653957);
\draw [line width=1.2pt] (0,-0.03)-- (-1.6923383396013794,1.7398293648742345);
\draw [line width=1.2pt] (-4,3.97)-- (-2.4454657851657657,2.442748314067662);
\draw [line width=1.2pt] (-3.408437643747404,3.388819735648454)-- (-4,3);
\draw [line width=1.2pt] (-2.2209663594421816,3.192457771905734)-- (-1.863904059516607,0.5064173746738677);
\draw [line width=1.2pt] (0,3)-- (-1.999655629985765,1.5276234844867227);
\end{tikzpicture}
\\
$C^4_2$
\end{minipage}
\medskip
\\

\begin{minipage}{\bwt}
    \centering
\begin{tikzpicture}[line cap=round,line join=round,>=triangle 45,x=0.8cm,y=0.8cm]
\draw [line width=1.2pt,dash pattern=on 5pt off 5pt] (-4.244372093023259,6.234883720930241)-- (-0.24437209302325913,6.234883720930241);
\draw [line width=1.2pt,dash pattern=on 5pt off 5pt] (-0.24437209302325913,6.234883720930241)-- (-0.24437209302325913,2.234883720930241);
\draw [line width=1.2pt,dash pattern=on 5pt off 5pt] (-0.24437209302325913,2.234883720930241)-- (-4.244372093023259,2.234883720930241);
\draw [line width=1.2pt,dash pattern=on 5pt off 5pt] (-4.244372093023259,2.234883720930241)-- (-4.244372093023259,6.234883720930241);
\draw [line width=1.2pt] (-3.1818394100778367,5.247006964358869)-- (-3.18,3.72);
\draw [line width=1.2pt] (-3.0111429741909053,5.4920065546907)-- (-0.9,3.4);
\draw [line width=1.2pt] (-3.1806254115484194,4.2391924310591165)-- (-4.244372093023259,6.234883720930241);
\draw [line width=1.2pt] (-0.24437209302325913,4)-- (-2.0253503106210844,4.5151495919146765);
\draw [line width=1.2pt] (-1.12,3.78)-- (-0.2443720930232591,2.2348837209302412);
\draw [line width=1.2pt] (-3.38,4.98)-- (-2.48,6.234883720930241);
\draw [line width=1.2pt] (-4.244372093023259,4.02)-- (-0.7102056256651272,3.0568846746305933);
\draw [line width=1.2pt] (-2.46,2.234883720930241)-- (-1.6932473804732147,3.3247788409858767);
\end{tikzpicture}
\\
$C^5_1$
\end{minipage}
&

\begin{minipage}{\bwt}
    \centering
\begin{tikzpicture}[line cap=round,line join=round,>=triangle 45,x=0.8cm,y=0.8cm]
\draw [line width=1.2pt,dash pattern=on 5pt off 5pt] (-3.98,6.03)-- (0.02,6.03);
\draw [line width=1.2pt,dash pattern=on 5pt off 5pt] (0.02,6.03)-- (0.02,2.03);
\draw [line width=1.2pt,dash pattern=on 5pt off 5pt] (0.02,2.03)-- (-3.98,2.03);
\draw [line width=1.2pt,dash pattern=on 5pt off 5pt] (-3.98,2.03)-- (-3.98,6.03);
\draw [line width=1.2pt] (-2.3800215598217402,6.03)-- (-2.3800215598217402,4.72920880407138);
\draw [line width=1.2pt] (-2.3800215598217402,5.298645315433246)-- (-1.1742148760330577,4.312396694214869);
\draw [line width=1.2pt] (-1.372561983471074,6.03)-- (-1.4456157454197966,4.534379817072604);
\draw [line width=1.2pt] (-3.98,4)-- (-2.3800215598217402,4.9286026732663775);
\draw [line width=1.2pt] (-1.3329295405905508,4.442211963272042)-- (-1.3725619834710743,2.03);
\draw [line width=1.2pt] (-1.3517818222490676,3.294775770939919)-- (0.02,4);
\draw [line width=1.2pt] (-1.586792499894012,4.649850548260457)-- (-3.598375907707616,4.221488704714096);
\draw [line width=1.2pt] (-2.3901141294819195,4.478785141340563)-- (-2.3642975206611565,2.03);
\end{tikzpicture}
\\
$C^5_2$
\end{minipage}

&
\begin{minipage}{\bwt}
\centering
\begin{tikzpicture}[line cap=round,line join=round,>=triangle 45,x=0.8cm,y=0.8cm]
\draw [line width=1.2pt,dash pattern=on 5pt off 5pt] (-3.98,6.03)-- (0.02,6.03);
\draw [line width=1.2pt,dash pattern=on 5pt off 5pt] (0.02,6.03)-- (0.02,2.03);
\draw [line width=1.2pt,dash pattern=on 5pt off 5pt] (0.02,2.03)-- (-3.98,2.03);
\draw [line width=1.2pt,dash pattern=on 5pt off 5pt] (-3.98,2.03)-- (-3.98,6.03);
\draw [line width=1.2pt] (-3.98,2.03)-- (-2.837874726823409,4.258763422709005);
\draw [line width=1.2pt] (-3,6.03)-- (-3.0086672611987977,3.9254758557252893);
\draw [line width=1.2pt] (-3.006899580370139,4.354691832716991)-- (-1.7395857329269875,4.592638688466817);
\draw [line width=1.2pt] (-2,2.03)-- (-2.304685364040811,3.4744628226458474);
\draw [line width=1.2pt] (-2.021186376493138,4.539766238718967)-- (-2.057626318874564,2.303196828742628);
\draw [line width=1.2pt] (0.02,6.03)-- (-2.0528581416007263,2.595852599720083);
\draw [line width=1.2pt] (-2.195442553364651,4.507048447885057)-- (-3.789591819076609,2.4015658859787776);
\draw [line width=1.2pt] (-3,2.03)-- (-2.9993229094424865,3.4453184656396454);
\draw [line width=1.2pt] (-2,6.03)-- (-1.5118753895977786,3.4921100300295858);
\end{tikzpicture}
\\
$C^6_1$
\end{minipage}
\medskip
\\

\begin{minipage}{\bwt}
    \centering
\begin{tikzpicture}[line cap=round,line join=round,>=triangle 45,x=0.8cm,y=0.8cm]
\draw [line width=1.2pt,dash pattern=on 5pt off 5pt] (-3.98,6.03)-- (0.02,6.03);
\draw [line width=1.2pt,dash pattern=on 5pt off 5pt] (0.02,6.03)-- (0.02,2.03);
\draw [line width=1.2pt,dash pattern=on 5pt off 5pt] (0.02,2.03)-- (-3.98,2.03);
\draw [line width=1.2pt,dash pattern=on 5pt off 5pt] (-3.98,2.03)-- (-3.98,6.03);
\draw [line width=1.2pt] (-3.98,6.03)-- (-2.888501051580805,4.8238757679010265);
\draw [line width=1.2pt] (-3.6074659801104976,5.618343800678216)-- (0.02,4.577944274657233);
\draw [line width=1.2pt] (-2.6263750100232537,5.336955474541208)-- (-2,6.03);
\draw [line width=1.2pt] (0.02,2.03)-- (-1.9615285001234326,3.726642951890258);
\draw [line width=1.2pt] (-1.034289704168265,2.932713837179797)-- (-2,2.03);
\draw [line width=1.2pt] (-1.5401932452782536,2.4598120478860253)-- (-3.588459916966984,4.540108660312034);
\draw [line width=1.2pt] (-3.98,4.540108660312034)-- (-3.3063287803965284,4.253565674586682);
\draw [line width=1.2pt] (-3.0020078946164017,5.1643962970078165)-- (-2.3457126660969507,3.277927845762539);
\draw [line width=1.2pt] (-2.501500259587641,3.725726896318796)-- (0.02,3.688807337545059);
\draw [line width=1.2pt] (-2.781208142397646,3.720233418858767)-- (-3.98,3.7455607590628577);
\end{tikzpicture}
\\
$C^6_2$
\end{minipage}
&

\begin{minipage}{\bwt}
    \centering
\begin{tikzpicture}[line cap=round,line join=round,>=triangle 45,x=0.8cm,y=0.8cm]
\draw [line width=1.2pt,dash pattern=on 5pt off 5pt] (-3.98,6.03)-- (0.02,6.03);
\draw [line width=1.2pt,dash pattern=on 5pt off 5pt] (0.02,6.03)-- (0.02,2.03);
\draw [line width=1.2pt,dash pattern=on 5pt off 5pt] (0.02,2.03)-- (-3.98,2.03);
\draw [line width=1.2pt,dash pattern=on 5pt off 5pt] (-3.98,2.03)-- (-3.98,6.03);
\draw [line width=1.2pt] (0.02,6.03)-- (-1.2662655152627962,4.4845808224844355);
\draw [line width=1.2pt] (-3.98,2.03)-- (-2.6689517056213883,3.4694789741986107);
\draw [line width=1.2pt] (-3.0308787663332195,3.0720974312592535)-- (-2.5479805224740173,2.03);
\draw [line width=1.2pt] (-2.7156267067282918,2.391781514415142)-- (0.02,4);
\draw [line width=1.2pt] (-2.8282141710466115,3.2946151104175283)-- (-2.6504953083798277,4.779883178349403);
\draw [line width=1.2pt] (-2.7152936350996746,4.238337480339524)-- (-1.9300046049997746,2.8536325350885257);
\draw [line width=1.2pt] (-0.8702260193578789,4.960413397300848)-- (-2.500393937933161,3.8594035311901798);
\draw [line width=1.2pt] (-2.7058645001045094,6.03)-- (-1.8363832368724728,4.307874099614703);
\draw [line width=1.2pt] (-2.2141230965002117,5.056039268182842)-- (-3.98,4);
\end{tikzpicture}
\\
$C^6_3$
\end{minipage}

&

\begin{minipage}{\bwt}
    \centering
\begin{tikzpicture}[line cap=round,line join=round,>=triangle 45,x=0.8cm,y=0.8cm]
\draw [line width=1.2pt,dash pattern=on 5pt off 5pt] (-3.98,6.03)-- (0.02,6.03);
\draw [line width=1.2pt,dash pattern=on 5pt off 5pt] (0.02,6.03)-- (0.02,2.03);
\draw [line width=1.2pt,dash pattern=on 5pt off 5pt] (0.02,2.03)-- (-3.98,2.03);
\draw [line width=1.2pt,dash pattern=on 5pt off 5pt] (-3.98,2.03)-- (-3.98,6.03);
\draw [line width=1.2pt] (-2.816602883553872,6.03)-- (-3.3887511980422445,4.816795972832524);
\draw [line width=1.2pt] (-3.2261947866767056,5.161486480374992)-- (-2.115259788374576,4.133909274894788);
\draw [line width=1.2pt] (-2.2768874531759504,4.283409384603024)-- (-3.98,4.595319205933799);
\draw [line width=1.2pt] (-2.759789236556856,4.7300771607072445)-- (-1.709219049060247,6.03);
\draw [line width=1.2pt] (-2.5335527249038603,5.010010869163635)-- (0.02,2.03);
\draw [line width=1.2pt] (-2.4513669403253933,5.111703435457867)-- (-0.06659969456136963,4.558406411450678);
\draw [line width=1.2pt] (-3.5161181886915367,4.510363403597873)-- (-1.2332190099494376,3.4925138673149654);
\draw [line width=1.2pt] (-2.0777590143960176,3.8690590340498243)-- (-2.835059280795432,2.03);
\draw [line width=1.2pt] (-1.5615678711277634,2.03)-- (-0.9505707122266673,3.1626616613472325);
\draw [line width=1.2pt] (-3.98,6.03)-- (-3.2689173133314533,5.070896086108676);
\end{tikzpicture}
\\
$C^6_4$
\end{minipage}
\medskip
\\

\begin{minipage}{\bwt}
    \centering
\begin{tikzpicture}[line cap=round,line join=round,>=triangle 45,x=0.8cm,y=0.8cm]
\draw [line width=1.2pt,dash pattern=on 5pt off 5pt] (-3.98,6.03)-- (0.02,6.03);
\draw [line width=1.2pt,dash pattern=on 5pt off 5pt] (0.02,6.03)-- (0.02,2.03);
\draw [line width=1.2pt,dash pattern=on 5pt off 5pt] (0.02,2.03)-- (-3.98,2.03);
\draw [line width=1.2pt,dash pattern=on 5pt off 5pt] (-3.98,2.03)-- (-3.98,6.03);
\draw [line width=1.2pt] (-2,6.03)-- (-2.3945454545454568,4.094545454545453);
\draw [line width=1.2pt] (-2.221701687411174,4.942435731570535)-- (-0.3545454545454591,4.774545454545453);
\draw [line width=1.2pt] (-1.1173696462329852,4.843136803576428)-- (-2.6145454545454565,3.5745454545454534);
\draw [line width=1.2pt] (-1.4939246023683814,4.524073165801963)-- (-2,2.03);
\draw [line width=1.2pt] (-0.6927416908237087,4.804955261072703)-- (-1,2.03);
\draw [line width=1.2pt] (-1,6.03)-- (-1.3536608363255913,4.8643835491164875);
\draw [line width=1.2pt] (-2.004488589943014,4.091460606763902)-- (-3.98,4);
\draw [line width=1.2pt] (-3.3952115042403603,4.0270740580785525)-- (-3.98,6.03);
\draw [line width=1.2pt] (0.02,2.03)-- (-0.7395395709377851,4.382307541706726);
\draw [line width=1.2pt] (0.02,4)-- (-0.5870022473063158,3.909896741138369);
\end{tikzpicture}
\\
$C^6_5$
\end{minipage}

&

\begin{minipage}{\bwt}
    \centering
\begin{tikzpicture}[line cap=round,line join=round,>=triangle 45,x=0.8cm,y=0.8cm]
\draw [line width=1.2pt,dash pattern=on 5pt off 5pt] (-4.2,6.22)-- (-0.2,6.22);
\draw [line width=1.2pt,dash pattern=on 5pt off 5pt] (-0.2,6.22)-- (-0.2,2.22);
\draw [line width=1.2pt,dash pattern=on 5pt off 5pt] (-0.2,2.22)-- (-4.2,2.22);
\draw [line width=1.2pt,dash pattern=on 5pt off 5pt] (-4.2,2.22)-- (-4.2,6.22);
\draw [line width=1.2pt] (-3.1627869996024573,4.414563173341968)-- (-4.2,2.22);
\draw [line width=1.2pt] (-3.3056279724898294,4.112336399648602)-- (-3.297980240478875,6.22);
\draw [line width=1.2pt] (-3.3026782263687613,4.925266627949985)-- (-1.8881078713390935,5.5733623808541175);
\draw [line width=1.2pt] (-2.2339424557773606,5.414915736906807)-- (-2.1198677128415233,6.22);
\draw [line width=1.2pt] (-2.0304357352247737,5.5081538291718)-- (-2.100554392716321,2.22);
\draw [line width=1.2pt] (-2.2339424557773606,5.414915736906807)-- (-4.053095477513528,2.530824541268411);
\draw [line width=1.2pt] (-3.2973493631710387,3.7289869214913196)-- (-3.33660688072928,2.22);
\draw [line width=1.2pt] (-1.559781429210651,6.22)-- (-2.069356792208532,3.6829845071267187);
\draw [line width=1.2pt] (-0.2,6.22)-- (-1.8928991601196206,4.561511568975985);
\draw [line width=1.2pt] (-1.1600483149266816,5.279466115266136)-- (-2.09295543442101,2.5763465803585226);
\draw [line width=1.2pt] (-1.3033205749321213,4.864331492817346)-- (-1.6370347097114613,2.22);
\end{tikzpicture}
\\
$C^7_1$
\end{minipage}

&

\begin{minipage}{\bwt}
    \centering
\begin{tikzpicture}[line cap=round,line join=round,>=triangle 45,x=0.8cm,y=0.8cm]
\draw [line width=1.2pt,dash pattern=on 5pt off 5pt] (-4,3.97)-- (0,3.97);
\draw [line width=1.2pt,dash pattern=on 5pt off 5pt] (0,3.97)-- (0,-0.03);
\draw [line width=1.2pt,dash pattern=on 5pt off 5pt] (0,-0.03)-- (-4,-0.03);
\draw [line width=1.2pt,dash pattern=on 5pt off 5pt] (-4,-0.03)-- (-4,3.97);
\draw [line width=1.2pt] (-2,-0.03)-- (-1.8764361596282293,2.677054630465474);
\draw [line width=1.2pt] (-1.8999377877888497,2.1621775341865876)-- (-3,3);
\draw [line width=1.2pt] (-2.498590374979651,2.6181195538271127)-- (-3.2153293961871388,1.0034380847668356);
\draw [line width=1.2pt] (-2.091923112801844,2.3083961903833723)-- (-2,3.97);
\draw [line width=1.2pt] (-1,3.97)-- (-2.042041070917226,3.210064859998687);
\draw [line width=1.2pt] (-1.937041115753094,1.3493124154828662)-- (-1.3297523423184434,3.729519658658989);
\draw [line width=1.2pt] (-1.495442620069284,3.080113286471454)-- (0,-0.03);
\draw [line width=1.2pt] (-2.771334329770391,2.003677481373892)-- (-4,2);
\draw [line width=1.2pt] (0,2)-- (-0.9812454950535229,2.0107233353526004);
\draw [line width=1.2pt] (-3.62427350561736,2.0011245753977254)-- (-4,3.97);
\draw [line width=1.2pt] (-3.0000089583829785,2.0029930423863136)-- (-3,3.97);
\draw [line width=1.2pt] (-3.0106609176021197,1.4645186034683868)-- (-3,-0.03);
\draw [line width=1.2pt] (-1,-0.03)-- (-0.27752969393738103,0.5471861634283335);
\end{tikzpicture}
\\
$C^8_1$
\end{minipage}

\end{tabular}
    \caption{CCA representations of the 12 irreducible torus graphs 
    with no vertices of degree two: \( C^i_j \) is a CCA representation of 
\( G^i_j \) from Figure \ref{fig_no_deg_two}.}
    \label{fig_CCArepsnodegree2}
\end{figure}

Finally we observe that allowing divalent vertex additions we have the following inductive 
construction for \( (2,2) \)-tight torus graphs.

\begin{thm}
    If \( G \) is a \( (2,2) \)-tight torus graph then \( G \) can be constructed
    from one of the torus graphs in Figure \ref{fig_no_deg_two} by a sequence 
    of moves each of which is either a digon split, triangle split, 
    quadrilateral split or a divalent vertex addition. \qed
\end{thm}

\end{document}